\documentclass[11pt,reqno]{amsart}
\usepackage{amssymb}
\usepackage{mathptmx}
\usepackage{mathtools}
\usepackage{mathrsfs}
\usepackage[export]{adjustbox}
\usepackage[all]{xy}
\usepackage{stmaryrd}
\usepackage{bbm}
\usepackage{fancyhdr}
\usepackage{xcolor}
\usepackage{hyperref}
\usepackage{tikz-cd}
\usepackage{cite}
\usepackage{amsmath,amsfonts}
\usepackage{graphicx}
\usepackage{wrapfig}
\usepackage{array}
\usepackage{amsthm}
\usepackage[left=1.2in, right=1.2in, top=1in, bottom=1in]{geometry}
\usepackage{etoolbox}
\patchcmd{\section}{\scshape}{\bfseries}{}{}
\makeatletter
\renewcommand{\@secnumfont}{\bfseries}
\makeatother
\xyoption{all}
\usepackage{secdot}

\theoremstyle{plain}
\DeclareMathOperator{\id}{\textrm{id}}

\newcommand{\FF}{\mathbb{F}}    
 
\newcommand{\ZZ}{\mathbb{Z}}
\newcommand{\wild}{\mathbb{L}}
\newcommand{\NI}{\operatorname{NI}} 
\newcommand{\Vect}{\operatorname{Vect}} 
\newcommand{\Rep}{\operatorname{Rep}}  
\newcommand{\Res}{\operatorname{Res}}
\newcommand{\nil}{\operatorname{nil}}
\newcommand{\Hom}{\operatorname{Hom}} 
 
\newcommand{\mycirc}[1][black]{\Large\textcolor{#1}{\ensuremath\bullet}}

\input xy
\xyoption{all}

\theoremstyle{definition}
\newtheorem{mydef}{\textbf{Definition}}[section]
\newtheorem{myeg}[mydef]{\textbf{Example}}

\newtheorem{question}[mydef]{\textbf{Question}}

\newtheorem{rmk}[mydef]{\textbf{Remark}}
\newtheorem{construction}[mydef]{\textbf{Construction}}

\theoremstyle{plain}
\newtheorem*{nothm}{\textbf{Theorem}}

\newtheorem*{nothma}{\textbf{Theorem A}}
\newtheorem*{nothmb}{\textbf{Theorem B}}
\newtheorem*{nothmc}{\textbf{Theorem C}}
\newtheorem*{nothmd}{\textbf{Theorem D}}
\newtheorem*{nothme}{\textbf{Theorem E}}
\newtheorem*{nothmf}{\textbf{Theorem F}}
\newtheorem*{nothmg}{\textbf{Theorem G}} 
\newtheorem*{nothmh}{\textbf{Theorem H}}

\newtheorem{mythm}[mydef]{\textbf{Theorem}}
\newtheorem{lem}[mydef]{\textbf{Lemma}}
\newtheorem{pro}[mydef]{\textbf{Proposition}}

\newtheorem{cor}[mydef]{\textbf{Corollary}}

\tikzset{main node/.style={circle,fill=black,draw,minimum size=0.3cm,inner sep=0pt},
}

\begin{document}

	\title[$\FF_1$-representations and Euler characteristics]{Coefficient quivers, $\FF_1$-representations, and Euler characteristics of quiver Grassmannians}

	\author{Jaiung Jun}
	\address{State University of New York at New Paltz, NY, USA}
	\curraddr{}
	\email{junj@newpaltz.edu}

	\author{Alex Sistko}
	\address{Manhattan College, NY, USA}
	\curraddr{}
	\email{asistko01@manhattan.edu}
	
	\makeatletter
	\@namedef{subjclassname@2020}{%
		\textup{2020} Mathematics Subject Classification}
	\makeatother
	
	\date{}
	
	\dedicatory{}

	\maketitle
	
	
\begin{abstract}
A quiver representation assigns a vector space to each vertex, and a linear map to each arrow of a quiver. When one considers the category $\textrm{Vect}(\mathbb{F}_1)$ of vector spaces ``over $\mathbb{F}_1$'' (the field with one element), one obtains $\mathbb{F}_1$-representations of a quiver. In this paper, we study representations of a quiver over the field with one element in connection to coefficient quivers. To be precise, we prove that the category $\textrm{Rep}(Q,\mathbb{F}_1)$ is equivalent to the (suitably defined) category of coefficient quivers over $Q$. This provides a conceptual way to see Euler characteristics of a class of quiver Grassmannians as the number of ``$\mathbb{F}_1$-rational points'' of quiver Grassmannians. We generalize techniques originally developed for string and band modules to compute the Euler characteristics of quiver Grassmannians associated to $\mathbb{F}_1$-representations. These techniques apply to a large class of $\mathbb{F}_1$-representations, which we call the $\mathbb{F}_1$-representations with finite nice length: we prove sufficient conditions for an $\mathbb{F}_1$-representation to have finite nice length, and classify such representations for certain families of quivers. Finally, we explore the Hall algebras associated to $\mathbb{F}_1$-representations of quivers. We answer the question of how a change in orientation affects the Hall algebra of nilpotent $\mathbb{F}_1$-representations of a quiver with bounded representation type. We also discuss Hall algebras associated to representations with finite nice length, and compute them for certain families of quivers.
\end{abstract} 

	
\section{Introduction}
Mathematics over the ``field with one element'' $\FF_1$ is a recent area of research that draws primarily from considerations in algebraic geometry, number theory, and combinatorics. The term ``field of characteristic one'' was originally coined by J.~Tits. In  \cite{tits1956analogues}, Tits observed that incidence geometries over a finite field $\FF_q$ have a combinatorial counterpart\footnote{It is ``thin geometries'' where one does not require a projective line to contain at least three points.} which could be interpreted as incidence geometries defined over ``the field of characteristic one''. 

In \cite{soule2004varietes}, C.~Soul\'e first introduced the notion of algebraic varieties over the field with one element (denoted by $\FF_1$) by taking the functor of points approach, and suggested several research directions to pursue. Soul\'e also asked whether or not Chevalley group schemes $G$ could be defined over $\FF_1$ in such a way as to relate the set $G(\FF_1)$ of ``$\FF_1$-rational points'' to the Weyl group $W_G$ of $G$. In \cite{connes2011notion}, A.~Connes and C.~Consani provided a positive answer to the question posed by Soul\'e. This question was further studied by O.~Lorscheid \cite{lorscheid2018geometry} by using the algebraic structure of a ``blueprint''; see \cite{lorscheid2016blueprinted} for blueprints.

One heuristic idea of $\FF_1$-geometry is that when an algebraic variety $X$ over $\mathbb{Z}$ has an ``$\FF_1$-model'' $X_{\FF_1}$, then the number of ``$\FF_1$-rational points'' of $X_{\FF_1}$ should be the Euler characteristic of $X(\mathbb{C})$.\footnote{There are several (non-equivalent) definitions for an $\FF_1$-model of $X$ and $\FF_1$-rational points.} The heuristic is essentially based on the relation between the counting function of $X(\FF_q)$ and the Euler characteristic of $X(\mathbb{C})$ for a smooth projective scheme $X$. For example, the cardinality of the set of $\FF_q$-rational points of the Grassmannian $\textrm{Gr}(k,n)$ is given by the $q$-binomial coefficients:
\begin{equation}\label{eq: Euler characteristic for Grassmannian}
|\textrm{Gr}(k,n)(\FF_q)|=\begin{bmatrix}
	n \\
	k
\end{bmatrix}_q, 
\end{equation}
and by evaluating \eqref{eq: Euler characteristic for Grassmannian} at $q=1$, we obtain 
\begin{equation}\label{eq: Euler charac 2}
{n\choose k} = \chi (\textrm{Gr}(k,n)(\mathbb{C})), 
\end{equation}
the Euler characteristic of $\textrm{Gr}(k,n)(\mathbb{C})$.\footnote{Note that for a Chevalley group $G$ and its Weyl group $W_G$, the cardinality $|W_G|$ can be computed from the counting function of $G(\FF_q)$ at $q=1$ after removing zeroes at $q=1$.}  We refer the reader to \cite[Section 4]{lorscheid2016blueprinted} for more details. 

Quiver Grassmannians are projective varieties whose points parameterize certain subrepresentations of a given quiver representation. The usual Grassmannians can be recovered by taking the quiver to have a single vertex and no arrows. In \cite{reineke2013every}, M.~Reineke showed that \emph{any} projective variety is a quiver Grassmannian. In particular, the class of quiver Grassmannians is not just a special class of projective varieties, but they are all projective varieties. What's more surprising is the result of C~M.~Ringel \cite{ringel2018quiver} which shows that there exists a \emph{single} quiver $Q$ (independent of a projective variety) such that for a given projective variety $X$, one can find a representation $M$ of $Q$ and a dimension vector $\textbf{e}$ such that $X=\operatorname{Gr}^Q_{\textbf{e}}(M)$, the quiver Grassmannian of $\textbf{e}$-dimensional subrepresentations of $M$:
\[
\textrm{Gr}^Q_{\textbf{e}}(M):=\{N \leq M \mid \textbf{\textrm{dim}}(N)=\textbf{e}\}.
\]
 We refer the reader to \cite{irelli2020three} for further results. 

Let $Q$ be a finite quiver throughout. Lorscheid has proved that if $M$ admits a basis in which the arrows of $Q$ act via integer matrices, then $\operatorname{Gr}^Q_{{\bf{e}}}(M)$ admits an  $\mathbb{F}_1$-model \cite{lorscheid2016blueprinted}. If $M$ is a tree module, it can be proved that the Euler characteristic of $\operatorname{Gr}^Q_{{\bf{e}}}(M)$ counts the number of $\FF_1$-rational points. Cerulli Irelli computes the Euler characteristic of $\operatorname{Gr}^Q_{\textbf{e}}(M)$ when $M$ is a string module in \cite{irelli2011quiver}. Haupt extends this work to tree modules, and produces results of a similar flavor for band modules \cite{Haupt2012euler}. For some tree modules and modules over certain tame quivers, Lorscheid and Weist develop techniques for computing Schubert decompositions of quiver Grassmannians which can be used to compute their Euler characteristics \cite{lorscheid2014schubert, lorscheid2015schubert,LorscheidWeist2019DynkinI,LorscheidWeist2019RepType,LorscheidWeistDynkinII}. It is therefore natural to pose the following questions, following Lorscheid.


\begin{question}\label{question: main question}(cf. \cite{lorscheid2016blueprinted})
\begin{enumerate}
\item If $M$ is such that $\operatorname{Gr}^Q_{\textbf{e}}(M)$ admits an $\FF_1$-model, what conditions on $M$ guarantee that the Euler characteristic of $\operatorname{Gr}^Q_{\textbf{e}}(M)$ counts its number of $\FF_1$-rational points?
\item If the number of $\FF_1$-rational points of $\operatorname{Gr}^Q_{\textbf{e}}(M)$ is given by its Euler characteristic, can one find an efficient combinatorial formula for computing it?
\end{enumerate}
\end{question}

In this paper, we partially provide an affirmative answer to Question \ref{question: main question}. We begin with the $\FF_1$-representations of $Q$, first introduced by M.~Szczesny in \cite{szczesny2011representations} and studied further by the authors in \cite{jun2020quiver} using the coefficient quivers of Ringel \cite{ringel1998exceptional}. Using the combinatorial techniques developed in \cite{Haupt2012euler}, we describe a class of $\FF_1$-representations whose associated quiver Grassmannians admit filtrations by locally-closed subsets, each of which is the fixed point set of a torus action on the previous one, and whose last piece is finite. We show that this class, which we call the $\FF_1$-representations with \emph{finite nice length}, contains many of the representations considered in \cite{irelli2011quiver,Haupt2012euler} as well as new ones. In addition to exhibiting new representations, we show that the class of $\FF_1$-representations with finite nice length includes representations whose coefficient quivers have first homology groups with arbitrarily high rank. Taken together, this demonstrates that the basic techniques of \cite{irelli2011quiver,Haupt2012euler} can be successfully applied to a broad class of representations beyond those previously considered in the literature.


Szczesny explored several aspects of representation theory over $\FF_1$ in \cite{szczesny2012hall,szczesny2014hall,szczesny2018hopf,jun2020toric}. In particular, he introduced a notion of quiver representations over $\FF_1$ based on an idea that vector spaces over $\FF_1$ are \emph{finite pointed sets}. To be precise, to define an $\FF_1$-representation of a quiver $Q$, one may replace vector spaces with finite pointed sets ($\FF_1$-vector spaces) and linear maps with pointed functions satisfying an injectivity condition ($\FF_1$-linear maps). This defines the category $\textrm{Rep}(Q,\FF_1)$ of $\FF_1$-representations of $Q$, which can be considered as a degenerate combinatorial model of the category $\textrm{Rep}(Q,\FF_q)$. In fact, Szczesny's main observation was that the Hall algebra $H_{Q,\FF_1}$ of $\textrm{Rep}(Q,\FF_1)$ behaves in some ways like the specialization at $q=1$ of the Hall algebra $H_{Q,\FF_q}$ of $\textrm{Rep}(Q,\FF_q)$.\footnote{The Hall algebra $H_{Q,\FF_1}$ can be constructed directly by mimicking the construction of $H_{Q,\FF_q}$. One may also directly appeal to the framework of Dyckerhoff and Kapranov \cite{dyckerhoff2012higher}.} This line of ideas was further pursued with the first named author in \cite{jun2020toric} to compute the Hall algebra of coherent sheaves on $\mathbb{P}^2$ by using its degenerate combinatorial model of monoid schemes. 

In our recent work \cite{jun2020quiver}, we stratified quivers according to the asymptotic growth of their indecomposable nilpotent $\FF_1$-representations. To this end, we defined the growth function $\NI_Q:\mathbb{N} \to \mathbb{Z}_{\geq 0}$ such that 
\[
\NI_Q(n) := \#\{\text{isoclasses of $n$-dimensional nilpotent indecomposables in $\Rep(Q,\FF_1)$} \}. 
\]
Then we used the growth function to define an order relation among quivers: for two quivers $Q$ and $Q'$, write $Q \le_{\nil} Q'$ if there exists a natural number $C$ such that $\NI_Q = O(\NI_{Q'} \circ \mu_C)$ in big-$O$ notation, where $\mu_C$ is a multiplication function by $C$.
This order relation induces an equivalence relation $\approx_{\nil}$ on quivers as follows:
\[
Q \approx_{\nil} Q' \iff Q \le_{\nil} Q'\textrm{ and } Q' \le_{\nil} Q.
\]
In \cite{jun2020quiver}, we proved the following: 
\begin{nothm}\cite{jun2020quiver}
Let $\wild_n$ be the quiver with one vertex and $n$-loops, and let $Q$ be connected.
\begin{enumerate}
	\item[(i)]
	$\wild_0,\wild_1,\wild_2$ are not equivalent to each other, and $\wild_m  \approx_{\nil} \wild_n$ whenever $\min\{m,n\} \geq 2$.
	\item[(ii)]
	$Q \approx_{\nil} \wild_0$ if and only if $Q$ is a tree quiver. 
	\item[(iii)]
	$Q \approx_{\nil} \wild_1$ if and only if $Q$ is a cycle quiver. 
	\item[(iv)]
	For any quiver $Q$, one has $Q \le_{\nil} \wild_2$.
\end{enumerate}
\end{nothm}



When $S$ is a coefficient quiver of $Q$, one naturally obtains a quiver map $F:S \to Q$ satisfying some conditions. The class of quiver maps satisfying this condition, called \emph{windings}, was studied by P.~Gabriel \cite{gabriel1981quivers}, W.W.~Crawley-Boevey \cite{crawley1989maps}, H.~Krause \cite{krause1991maps} and N.~Haupt \cite{Haupt2012euler}. More explicitly, a winding is a morphism of quivers $F:S\to Q$ consisting of a pair of functions\footnote{For a quiver $Q$, $Q_0$ is the vertex set of $Q$ and $Q_1$ is the arrow set of $Q$.}
\[
F_0:S_0 \to Q_0, \quad F_1:S_1 \to Q_1
\]
satisfying the following condition:
\[
F_1(\alpha)=F_1(\beta) \textrm{ implies } s(\alpha) \neq s(\beta) \textrm{ and } t(\beta) \neq t(\beta),
\]
where $s(\alpha)$ (resp.~$t(\alpha)$) is the source (resp.~target) of an arrow $\alpha$. We consider the category $\mathcal{C}_Q$ of quivers over a quiver $Q$ as follows: Let $\mathcal{C}_Q$ be the category whose objects are windings of quivers $F:S \to Q$. A morphism $\phi : (S,F)\rightarrow (S',F')$ is an ordered triple $\phi = (\mathcal{U}_{\phi}, \mathcal{D}_{\phi}, c_{\phi})$ satisfying some technical conditions, where $\mathcal{U}_{\phi}$ is a full subquiver of $S$, $\mathcal{D}_{\phi}$ is a full subquiver of $S'$, and $c_{\phi} : \mathcal{U}_{\phi} \rightarrow \mathcal{D}_{\phi}$ is a quiver isomorphism. We first upgrade the correspondence between coefficient quivers and $\FF_1$-representations in \cite{jun2020quiver} to the categorical equivalence as follows. 

\begin{nothma}(Proposition \ref{proposition: equivalence of categories})\label{thm: nothmA}
The categories $\emph{Rep}(Q,\FF_1)$ and $\mathcal{C}_Q$ are equivalent. This restricts to an equivalence between $\Rep(Q,\FF_1)_{\nil}$ and the full subcategory of $\mathcal{C}_Q$ whose objects are windings $F:S \to Q$ with $S$ acyclic. 
\end{nothma}

The above theorem also provides a conceptual framework to compute the Euler characteristic of a quiver Grassmannian through a base-change functor. To be precise, one always obtains a representation of $Q$ over $\mathbb{C}$ from a representation of $Q$ over $\FF_1$ functorially as follows: a finite pointed set $V$ defines a vector space $V_\mathbb{C}$ whose basis is $V-\{0_V\}$. This induces a functor for quiver representations:
\[
\mathbb{C}\otimes_{\FF_1}-:\textrm{Rep}(Q,\FF_1) \to \textrm{Rep}(Q,\mathbb{C}), \quad M \mapsto M_\mathbb{C}.
\]
This functor is always faithful, but generally not full (cf. the example at the end of \cite{krause1991maps}).  


The methods of Cerulli Irelli and Haupt to compute Euler characteristics of quiver Grassmannians are based on the following idea: that when a projective variety $X$ is equipped with a torus action admitting a finite number of fixed points, one may compute the Euler characteristic $\chi(X)$ of $X$ as the number of fixed points by the torus action. To ensure the existence of a torus action, Cerulli Irelli introduced a certain condition for string modules. This idea was generalized by Haupt by introducing a notion of gradings on representations. We show that Haupt's definition applies to $\FF_1$-representations of quivers.

For an $\FF_1$-representation $M$ of $Q$ and its corresponding winding $\Gamma_M \to Q$ (from Theorem $A$), we define a \emph{nice sequence} on $M$ to be a collection $\underline{\partial} = ( \partial_i)_{i \geq 0}$ of functions $\partial_i : (\Gamma_M)_0 \rightarrow \mathbb{Z}$ satisfying the following two conditions:\footnote{For the precise definition, see Definition \ref{definition: gradings}}
\begin{enumerate} 
\item $\partial_0$ is a nice grading.
\item For all $i>0$, $\partial_i$ is a $(\partial_0,\ldots , \partial_{i-1})$-nice grading.
\end{enumerate}  
See Section \ref{s: gradings} or \cite{Haupt2012euler} for the terminology on gradings. If there exists a nice sequence $\underline{\partial} = ( \partial_i)_{i\geq 0}$ on $M$ with the property that for all distinct $x,y\in (\Gamma_M)_0$, $\partial_i(x) \neq \partial_i(y)$ for some $i\geq 0$, then Haupt proves in \cite{Haupt2012euler} that the following formula holds for all dimension vectors $\textbf{e}$: 
\begin{equation}\label{e: intro nice}
\chi_{\textbf{e}}^Q(M_\mathbb{C}) = |\{ N \le M \mid \textbf{\textrm{dim}}(N) = \textbf{e} \}|,
\end{equation}
where $\chi^Q_{\textbf{e}}(M_\mathbb{C})$ is the Euler characteristic of $\textrm{Gr}^Q_{\textbf{e}}(M_\mathbb{C})$. In other words, the Euler characteristics of the associated quiver Grassmannians can all be computed by counting $\FF_1$-subrepresentations of $M$, which is a combinatorial task. See Proposition \ref{rmk: distinguish vertices} below for more details. We say that $M$ has \emph{finite nice length} in this case: specifically, we say $\operatorname{nice}(M) = n$ if there exists a nice sequence $\underline{\partial}$ for $M$ such that each pair of vertices in $\Gamma_M$ can be distinguished by the first $n+1$ gradings (and $\operatorname{nice}(M) = \infty$ otherwise). More generally, if \eqref{e: intro nice} holds for all dimension vectors $\textbf{e}$ we say that $M$ is \emph{nice}. Our task is then to identify all nice $\FF_1$-representations of a given quiver, or at least the $\FF_1$-representations with finite nice length. 

In Construction \ref{con: universal}, we generalize notions from \cite{Haupt2012euler} to test for the existence of nice sequences with prescribed properties. To any indecomposable $\FF_1$-representation $M$, we construct a sequence of finitely-generated free abelian groups $\mathcal{V}^{(i)}_M$. For each $i$, we then define a function  
\[
X^{(i)} : (\Gamma_M)_0 \rightarrow \mathcal{V}^{(i)}_M 
\] 
\[ 
 v \mapsto X^{(i)}_v 
\] called the \emph{universal $i$-nice grading on $M$}. The image of $v$ under $X^{(i)}$ is called the \emph{$i$-nice variable associated to $v$}. The universal $i$-nice grading is only unique up to translation, but this is easily dealt with by specifying a \emph{basepoint} $b \in (\Gamma_M)_0$. The name of this function is justified by the following theorem.

\begin{nothmb}(Theorem \ref{t: universal property})
Let $M$ be an indecomposable $\FF_1$-representation of $Q$, with associated winding $c : \Gamma \rightarrow Q$ and basepoint $b \in \Gamma_0$. Let $\underline{\partial} = (\partial_i)_{i=0}^{\infty}$ be a nice sequence for $M$. Then for each $i$, there exists a unique affine map 
\[ 
\operatorname{ev}^{(i)}(\underline{\partial}) : \mathcal{V}^{(i)}_M \rightarrow \mathbb{Z}
\] 
such that $\partial_i = \operatorname{ev}^{(i)}(\underline{\partial}) \circ X^{(i)}$. We write $X^{(i)}(\underline{\partial}) :=  \operatorname{ev}^{(i)}(\underline{\partial}) \circ X^{(i)}$ and call it the evaluation of $X^{(i)}$ at $\underline{\partial}$. Conversely, any such sequence of affine maps defines a nice sequence on $M$. 
\end{nothmb}
\noindent As a consequence, we have $\partial_i(u) \neq \partial_i(v)$ for \emph{some} nice sequence $\underline{\partial}$ if and only if $X^{(i)}_u \neq X^{(i)}_v$. 

In Section \ref{s: Euler}, we apply the machinery of Section \ref{s: gradings} to identify representations with finite nice length.  To begin, we prove several sufficient conditions for $M$ to have finite nice length (see Sections \ref{s: gradings} and \ref{s: Euler} for all associated terminology): 

\begin{nothmc} (Propositions \ref{p: pos gradings} and \ref{p: gluing})
Let $Q$ be a quiver and $M$ an indecomposable $\FF_1$-representation of $Q$ with associated winding $c : \Gamma \rightarrow Q$. Then the following statements hold: 
\begin{enumerate} 
\item Suppose that $M$ admits a positive or negative nice grading, and that $Q$ has no loops. If for each $\alpha \in Q_1$, $\partial$ restricts to an injection on the set $\{s(\beta) \mid \beta \in \Gamma_1, c(\beta) =\alpha \}$, then $\operatorname{nice}(M) \le 1$ and $M$ is nice.
\item Suppose that $M$ admits a non-degenerate grading, and that for all $\alpha \in Q_1$, the minimal subquiver of $\Gamma$ containing the arrows $\{ \beta \in \Gamma_1 \mid c(\beta) = \alpha\}$ is connected. Then $\operatorname{nice}(M)\le 1$. 
\item Suppose that $N$ is another indecomposable representation of $Q$, with associated winding $c' : \Gamma' \rightarrow Q$. Suppose that $c(\Gamma)$ and $c'(\Gamma')$ have disjoint arrow sets, and that there exist vertices $u \in \Gamma_0$ and $v \in \Gamma'_0$ with $c(u) = c'(v)$. If $\operatorname{nice}(M)<\infty$ and $\operatorname{nice}(N)<\infty$, then the $\FF_1$-representation associated to the amalgam $c\sqcup_{u\sim v}c' : \Gamma\sqcup_{u\sim v}\Gamma' \rightarrow Q$ has finite nice length.
\end{enumerate}
\end{nothmc}

We then exhibit a nontrivial class of representations which satisfy $\operatorname{nice}(M) \le 1$. In general, this class of representations will have nice length $1$, and coefficient quivers that are neither trees nor affine Dynkin quivers of type $\tilde{\mathbb{A}}$. In particular, this class contains representations different than the cases considered in \cite{irelli2011quiver, Haupt2012euler}:

\begin{nothmd}(Proposition \ref{p: connected fibers})
Let $M$ be a nilpotent $\FF_1$-representation of a quiver $Q$ with associated winding $c: \Gamma \rightarrow Q$. Suppose that $c^{-1}(\alpha)$ is connected for all $\alpha \in Q_1$, and that $\Gamma$ contains a set of $\mathbb{Z}$-linearly independent cycles $\{ X_1,\ldots , X_n\}$ with following properties: 
\begin{enumerate} 
\item [(i)] The cycles $[\iota\circ H_1(c)](X_1)$, $\ldots , [\iota\circ H_1(c)](X_n)$ form a $\mathbb{Q}$-basis for $\mathbb{Q}\otimes_{\mathbb{Z}}\operatorname{Im}(\iota \circ H_1(c))$, where $\iota\circ H_1(c)$ is as in \eqref{eq: construction homology}.
\item[(ii)] 
For all $i$ we can write $X_i = p_i - q_i$, where $p_i$ and $q_i$ are directed paths of positive length in $\Gamma$ with common source and target, but no interior vertices in common. 
\item[(iii)] 
For each $i\le n$, either $c(p_i)$ or $c(q_i)$ consists of arrows that do not appear in $c(X_j)$ for $j \neq i$, where we consider $c(X_j)$ as a subquiver of $Q$.
\end{enumerate}  
Then $\operatorname{nice}(M)\le 1$ and $M$ is nice.
\end{nothmd}

 We illustrate the above results with several examples. Furthermore, we patch apparent gaps in the proofs of Lemmas 6.3-6.4 in \cite{Haupt2012euler}, at least for the case of $\FF_1$-representations. We outline the gaps we believe we have uncovered in an appendix to this article, and explain how our results resolve them. It should be noted that the class of $\FF_1$-representations with finite nice length contains representations which have been previously studied in the literature, in addition to the new ones described above. Special cases include the following: the $\FF_1$-representations satisfying the conditions of Theorem 1 in \cite{irelli2011quiver}, which are recovered as the $\FF_1$-representations with nice length $0$; the $\FF_1$-representations whose coefficient quiver is a tree, correcting the proof of \cite[Lemma 6.3]{Haupt2012euler}; and the $\FF_1$-representations whose coefficient quiver is a primitive affine Dynkin quiver of type $\tilde{\mathbb{A}}$, recovering a special case of \cite[Lemma 6.4]{Haupt2012euler}. These cases cover all $\FF_1$-representations with finite nice length when $Q$ is a pseudotree, but not in general. We summarize these remarks with the following theorem:

\begin{nothme}[cf. Theorem 1 \cite{irelli2011quiver}, Lemmas 6.3-6.4 \cite{Haupt2012euler}] 
Let $Q$ be a quiver, and for each $n \in \mathbb{N}$ let $I_{Q,\nil}^{n}$ denote the set of isomorphism classes of finite-dimensional indecomposable, nilpotent $\FF_1$-representations of $Q$ with nice length $n$. Furthermore, let $I_{Q,\nil}^{\operatorname{nice}} = \bigcup_{n \in \mathbb{N}}{I_{Q,\nil}^{n}}$. Then the following hold: 
\begin{enumerate} 
\item $M \in I_{Q,\nil}^0$ if and only if $M$ admits a nice grading $\partial$ that restricts to an injection on $c_M^{-1}(v)$, for all $v \in Q_0$. 
\item If the coefficient quiver of $M$ is a tree, then $M \in I_{Q,\nil}^{\operatorname{nice}}$. In general, $M \not\in I_{Q,\nil}^0$. 
\item If the coefficient quiver of $M$ is an affine Dynkin quiver of type $\tilde{\mathbb{A}}$, then $M \in I_{Q,\nil}^{\operatorname{nice}}$ if and only if the associated winding $c_M : \Gamma_M \rightarrow Q$ is primitive. In general, $M \not\in I_{Q,\nil}^0$.  
\end{enumerate}   
If $Q$ is a (not necessarily proper) pseudotree, then every $M \in I_{Q,\nil}^{\operatorname{nice}}$ belongs to at least one of the three cases described above. If $Q = \wild_n$ with $n\geq 2$, then $I_{Q,\nil}^{\operatorname{nice}}$ contains $\FF_1$-representations that do not belong to any of the these three cases.
\end{nothme}

Case (1) is discussed in Example \ref{ex: irelli}. Case (2) is discussed in Corollary \ref{c: distinguishing trees}, and case (3) is discussed in Theorem \ref{l: distinguishing bands}. Examples of representations with strictly positive nice length are discussed throughout Sections \ref{s: gradings} and \ref{s: Euler}. The results on pseudotrees follows readily from Corollary  \ref{c: nice pseudotree} and the classification of nilpotent indecomposable $\FF_1$-representations for bounded type given in \cite{jun2020quiver}. Finally, explicit examples for $\wild_2$ and $\wild_3$ are computed in Section \ref{s: Euler}.

In Section \ref{section: Hall}, we turn our attention to the Hall algebras $H_Q$ and $H_{Q,\nil}$ of $\Rep(Q,\FF_1)$ and $\Rep(Q,\FF_1)_{\nil}$. Recall that $H_Q$ and $H_{Q,\nil}$ are graded, connected, cocommutative Hopf algebras: by Milnor-Moore Theorem, they are isomorphic to the universal enveloping algebras of their Lie subalgebras of primitive elements. Throughout, we denote the Lie algebra of $H_{Q,\nil}$ by $\mathfrak{n}_Q$. Motivated by a question from \cite{szczesny2011representations}, in Section \ref{s: Hall Algebras} we study how a change in the orientation of $Q$ affects $H_{Q,\nil}$. We obtain the following results for quivers of bounded representation type.

\begin{nothmf}
Let $Q$ be a quiver, with $\mathfrak{n}_Q$ the Lie algebra of primitive elements in $H_{Q,\nil}$. Let $Q'$ be a quiver with the same underlying graph as $Q$. 
\begin{enumerate} 
\item (Proposition \ref{p: trees}) If $Q$ and $Q'$ are trees, then $\mathfrak{n}_Q \cong \mathfrak{n}_{Q'}$ as Lie algebras.
\item (Proposition \ref{p: affine}) Suppose that $Q'$ is an equioriented affine Dynkin quiver of type $\tilde{\mathbb{A}}_n$. Then $\mathfrak{n}_Q$ is a central extension of $\mathfrak{n}_{Q'}$. 
\end{enumerate}
\end{nothmf} 

Note that Szczesny computed $\mathfrak{n}_Q$ in the case that $Q$ is an equioriented $\tilde{\mathbb{A}}_n$ in \cite{szczesny2011representations}. Finally, in Section \ref{s: nice length} we construct Hall algebras associated to representations of finite nice length and identify them in specific instances. Indeed, consider the full subcategories $\Rep(Q,\FF_1)^{\operatorname{nice}}$ and $\Rep(Q,\FF_1)_{\nil}^{\operatorname{nice}}$ of $\Rep(Q,\FF_1)$ and $\Rep(Q,\FF_1)_{\nil}$ whose objects $M$ satisfy $\operatorname{nice}(M)<\infty$. These categories are finitary and proto-exact, and so one can associate Hall algebras $H_Q^{\operatorname{nice}}$ and $H_{Q,\nil}^{\operatorname{nice}}$. We prove the following theorem, which relates $H_Q$ (resp. $H_{Q,\nil}$) to $H_Q^{\operatorname{nice}}$ (resp. $H_{Q,\nil}^{\operatorname{nice}}$).

\begin{nothmg}(Proposition \ref{p: hopf ideal})
Let $Q$ be a quiver. Then the $\mathbb{C}$-subspaces  
\[
I=\langle [M] \mid \operatorname{nice}(M) = \infty \rangle 
\]
\[ 
I_{\nil}=\langle [M] \mid \text{$M$ is nilpotent and } \operatorname{nice}(M) = \infty \rangle 
\] 
are Hopf ideals in $H_Q$ and $H_{Q,\nil}$, respectively. We have Hopf algebra isomorphisms $H_Q/I \cong H_Q^{\operatorname{nice}}$ and $H_{Q,\nil}/I_{\nil} \cong H_{Q,\nil}^{\operatorname{nice}}$.
\end{nothmg} 

\noindent This allows us to identify $H_{Q,\nil}^{\operatorname{nice}}$ in the case that $Q$ is a (not-necessarily proper) pseudotree. Specific formulas can be found in Corollaries \ref{c: nice bounded}, \ref{c: nice pseudotree hall}. We can also give the following, more conceptual interpretation to our results. We say that an $\FF_1$-representation $M$ is \emph{absolutely indecomposable} if $k\otimes_{\FF_1}M$ is indecomposable for every algebraically-closed field $k$.

\begin{nothmh}(Corollaries \ref{c: nice bounded}, \ref{c: nice pseudotree}, \ref{c: nice pseudotree hall})
Let $Q$ be a (not-necessarily proper) pseudotree. Then the following statements hold.\footnote{See \cite[Section 5.2]{jun2020quiver} for the notion of bounded representation type over $\mathbb{F}_1$.}
\begin{enumerate} 
\item Let $Q$ be of bounded representation type over $\FF_1$. Then $H_{Q,\nil}^{\operatorname{nice}}$ has a generating set that may be naturally identified with the absolutely indecomposable $\FF_1$-representations of $Q$.
\item Let $Q$ be a proper pseudotree with central cycle $C$. Then $H_{Q,\nil}^{\operatorname{nice}}$ has a generating set that may be naturally identified with the indecomposable $\FF_1$-representations $M$ such that $\Res_C(M)$ is absolutely indecomposable.
\end{enumerate}
\end{nothmh}

\bigskip 

\textbf{Acknowledgment}\hspace{0.1cm} The authors would like to thank Ryan Kinser for several insightful conversations on the connection between quiver representations and $\FF_1$-geometry. 

\bigskip

\textbf{Competing Interests:}\hspace{0.1cm} The authors declare none.

\bigskip


\section{Preliminaries}  

\subsection{Representations of quivers over $\FF_1$}

In this section, we review basic definitions and properties of representations of quivers over $\FF_1$. We first recall the definition of $\FF_1$-vector spaces and $\FF_1$-linear maps.  

\begin{mydef}
The category $\textrm{Vect}(\mathbb{F}_1)$ of \emph{finite dimensional vector spaces} over $\mathbb{F}_1$ consists of the following:
\begin{enumerate}
	\item 
Objects are finite pointed sets $(V,0_V)$, called $\FF_1$-vector spaces; 
\item 
Morphisms are pointed functions $\varphi:V \to W$ such that $\varphi|_{V-\varphi^{-1}(0_W)}$ is an injection, called $\FF_1$-linear maps.
\end{enumerate}
For an $\mathbb{F}_1$-vector space $V$, the \emph{dimension} of $V$, denoted by $\dim_{\FF_1}(V)$, is the number of nonzero elements of $V$. In other words, $\dim_{\FF_1}(V)=|V|-1$.
\end{mydef}

\begin{mydef}\label{definition: $F_1$-vectorspace}
	Let $V$ and $W$ be $\mathbb{F}_1$-vector spaces. 
	\begin{enumerate}
		\item 
		The \emph{direct sum} $V\oplus W$ is the following $\FF_1$-vector space: $V\oplus W:=V\sqcup W /\langle 0_V \sim 0_W\rangle$.
		\item 
		A unique $\mathbb{F}_1$-linear map $0: V \to W$ sending any element in $V$ to $0_W$ is said to be the \emph{zero map}.
		\item 
		$W$ is said to be a \emph{subspace} of $V$ if $W$ is a subset of $V$ containing $0_V$ and $0_W=0_V$. 
		\item 
		For a subspace $W$ of $V$, the \emph{quotient space} $V/W$ is defined as $V-(W-\{0_V\})$. 
		\item 
		For an $\mathbb{F}_1$-linear map $\varphi:V \to W$, the \emph{kernel} (resp.~\emph{cokernel}) of $\varphi$ is defined to be $\ker(\varphi):=\varphi^{-1}(0_W)$ (resp.~$\textrm{coker}(\varphi):=W/\varphi(V)$).
	\end{enumerate}
\end{mydef}

\begin{mydef}
	A \emph{quiver} $Q$ is a finite directed graph, where we allow multiple arrows and loops. To be precise, a quiver $Q$ consists of a quadruple $Q=(Q_0,Q_1,s,t)$ as follows:
	\begin{enumerate}
		\item 
		$Q_0$ (resp.~$Q_1$) is the finite set of vertices (resp.~arrows),
		\item
		$s$ and $t$ are functions
		\[
		s,t:Q_1 \to Q_0
		\] 
		assigning to each arrow in $Q_1$ its \emph{source} and \emph{target} in $Q_0$. 
	\end{enumerate}
A quiver $Q$ is \emph{connected} if its underlying undirected graph is connected. A quiver is \emph{acyclic} if it does not contain any directed cycles. By a \emph{subquiver} $S$ of $Q$, we mean a quiver $S = (S_0,S_1)$ such that $S_i \subseteq Q_i$ for $i = 0,1$.
\end{mydef}

\begin{myeg}
For a nonnegative integer $n$, we let $\wild_n$ be the quiver with one vertex and $n$ loops. $\wild_1$ is called the \emph{Jordan quiver}.
\end{myeg}

Throughout this paper, we will simply denote a quiver by $Q$ and the underlying undirected graph of a quiver $Q$ by $\overline{Q}$. We will also use the basic terminology of undirected graphs to $Q$; when we say some graph-theoretic property holds for $Q$, it means that it holds for $\overline{Q}$. For example, when we say $Q$ is a tree, it means that $\overline{Q}$ is a tree.

\begin{mydef}
Let $S$ and $Q$ be quivers.
\begin{enumerate}
	\item 
 A \emph{quiver map} $F : S \rightarrow Q$ consists of a pair of functions  
\[ 
F_i : S_i \rightarrow Q_i, \quad i=0,1
\] 
which satisfy the following conditions: for all $\alpha \in S_1$,
\[
s(F_1(\alpha)) = F_0(s(\alpha)), \quad t(F_1(\alpha)) = F_0(t(\alpha)).
\] 
\item 
A quiver map $F$ is said to be injective (resp.~surjective) if and only if $F_0$ and $F_1$ are both injective (resp.~surjective). 
\end{enumerate}
\end{mydef}

With the notion of quiver maps, we can identify a subquiver $S$ of $Q$ with the image of the inclusion map $S \hookrightarrow Q$. A subquiver $S$ is said to be \emph{full} if for any $u, v \in S_0$, any arrow $\alpha$ in $Q_1$ such that $u=s(\alpha)$ and $v=t(\alpha)$ (or $u=t(\alpha)$ and $v=s(\alpha)$) is also in $S_1$.

\begin{mydef} 
Let $Q$ be a quiver with underlying graph $G =\overline{Q}$. Recall that a \emph{walk} in $G$ is a sequence of edges $ w = (e_1, \ldots , e_d)$ such that $e_i$ and $e_{i+1}$ share an endpoint. In other words, it is a sequence of vertices $(v_1,v_2,\ldots , v_d, v_{d+1})$ together with a specification of an edge $e_i$ between $v_i$ and $v_{i+1}$ for all $i\le d$. Let $\alpha_i$ be the arrow in $Q$ corresponding to the edge $e_i$ of $G$. Define $\epsilon_i = +1$ if $s(\alpha_i) = v_i$ and $t(\alpha_i) = v_{i+1}$, and $\epsilon_i = -1$ if $s(\alpha_i) = v_{i+1}$ and $t(\alpha_i) = v_i$. Then we will write a walk of $Q$ as follows: 
\[ 
w = \alpha_1^{\epsilon_1}\cdots \alpha_d^{\epsilon_d}.
\]
A \emph{path} is a directed walk in the sense that all $\epsilon_i$ have the same sign.
\end{mydef} 

\begin{mydef} 
Let $Q$ be a quiver with underlying graph $\overline{Q}$. Then $\overline{Q}$ is a $1$-simplex, whose $0$-simplices can be identified with $Q_0$ and whose $1$-simplices can be identified with $Q_1$. We then obtain a chain complex 
\[ 
0 \rightarrow \mathbb{Z}Q_1 \xrightarrow[]{\partial} \mathbb{Z}Q_0 \rightarrow 0
\] 
where $\mathbb{Z}Q_1$ and $\mathbb{Z}Q_0$ denote the free abelian groups generated by arrows and vertices respectively with $\partial$ defined via the formula $\partial (\alpha) = t(\alpha) - s(\alpha)$, for $\alpha \in Q_1$. The \emph{(integral) cycle space} is defined to be the first homology group $H_1(Q, \mathbb{Z}) = \operatorname{ker}(\partial)$. Note that the cycle space is a finitely-generated free abelian group, and a basis for $H_1(Q,\mathbb{Z})$ is called a \emph{(integral) cycle basis} for $Q$. Note that the cycle space does not depend on the choice of orientation of $\overline{Q}$: indeed, the orientation functions simply as a device for writing elements of $\operatorname{ker}(\partial)$ as $\mathbb{Z}$-linear combinations of $1$-simplices. 
\end{mydef} 

\begin{myeg}
Consider the following quiver:
\[
Q=
\left(\begin{tikzcd}  
	& v_2 \arrow[dl,swap,"\alpha_1"]\arrow[dr,"\alpha_3"]  &  \\ 
	v_1  & & v_4 \\ 
	& v_3 \arrow[ul,"\alpha_2"] \arrow[ur,swap,"\alpha_4"]&  \\
\end{tikzcd} \right).
\]	
Then, we obtain the following chain complex
\[ 
0 \rightarrow \mathbb{Z}^4 \xrightarrow[]{\partial} \mathbb{Z}^4 \rightarrow 0
\]
where $\partial$ is given by the following matrix: 
\[
\partial = \begin{bmatrix*}[r] 
1& 1& 0&0  \\
-1&0 &-1 &0  \\
0& -1&0 &-1  \\
0& 0&1 &1 
\end{bmatrix*}
\]
Then $\ker(\partial)$ is generated by $\begin{bmatrix*}[r]
	1 \\
	-1\\
	-1\\
	1
\end{bmatrix*}$, or $(\alpha_1-\alpha_2-\alpha_3+\alpha_4)$ corresponding to the unique cycle of $Q$. 

\end{myeg}

\begin{mydef} 
Let $Q$ be a connected quiver. We say that $Q$ is a \emph{pseudotree} if $\operatorname{rank}(H_1(Q,\mathbb{Z})) \le 1$. If $Q$ is not a tree or an affine Dynkin quiver of type $\tilde{\mathbb{A}}_n$, we say that $Q$ is a \emph{proper pseudotree}. Any proper pseudotree contains a unique subquiver $C$ that is an affine Dynkin quiver of type $\tilde{\mathbb{A}}_n$: we call $C$ the \emph{central cycle} of $Q$.
\end{mydef}

\begin{mydef}\cite[Definition 4.1]{szczesny2011representations}\label{definition: representation of a quiver over $F_1$}
Let $Q$ be a quiver.
\begin{enumerate}
	\item 
A representation of $Q$ over $\FF_1$ (or an $\FF_1$-representation of $Q$) is the collection of data $\mathbb{V}=(V_i,f_\alpha)$ consisting of an $\FF_1$-vector space $V_i$ for each vertex $i \in Q_0$ and an $\FF_1$-linear map $f_\alpha \in \Hom(V_{s(\alpha)},V_{t(\alpha)})$ for each arrow $\alpha \in Q_1$. 
\item 
By a subrepresentation $\mathbb{W}=(W_i,g_\alpha)$ of $\mathbb{V}=(V_i,f_\alpha)$ we mean an $\FF_1$-representation such that each $W_i$ is a subspace of $V_i$ and $g_\alpha$ is a restriction of $f_\alpha$. When $\mathbb{W}$ is a subrepresentation of $\mathbb{V}$, we write $\mathbb{W} \leq \mathbb{V}$. 
\end{enumerate}
\end{mydef}

\begin{mydef}\cite[Definition 4.3]{szczesny2011representations}
	Let $Q$ be a quiver and $\mathbb{V}=(V_i,f_\alpha)$ be an $\FF_1$-representation of $Q$. 
	\begin{enumerate}
		\item 
		The \emph{dimension} of $\mathbb{V}$ is defined to be the sum of dimensions of $V_i$:
		\[
		\dim(\mathbb{V})=\sum_{i \in Q_0} \dim_{\FF_1}(V_i).
		\]
		\item 
		The \emph{dimension vector} of $\mathbb{V}$ is an element of $\mathbb{N}^{|Q_0|}$:
		\[
		\textbf{\textrm{dim}}(\mathbb{V})=(\dim_{\FF_1}(V_i))_{i \in Q_0}. 
		\]
	\end{enumerate}
	An $\FF_1$-representation $\mathbb{V}=(V_i,f_\alpha)$ is \emph{nilpotent} if there exists a positive integer $N$ such that $\forall~n \geq N$ and any path $\alpha_1\alpha_2\dots \alpha_n$ in $Q$ (left-to-right in the order of traversal), one has
	\[
	f_{\alpha_n}f_{\alpha_{n-1}}\cdots f_{\alpha_1}=0 \textrm{ (zero map). }
	\]
\end{mydef}

\begin{mydef}
Let $\mathbb{V}=(V_i,f_\alpha)$ and $\mathbb{W}=(W_i,g_\alpha)$ be $\FF_1$-representations of a quiver $Q$. A \emph{morphism} $\Phi:\mathbb{V} \to \mathbb{W}$ is a collection of $\FF_1$-linear maps $\{\phi_i:V_i\to W_i\}_{i \in Q_0}$ making the following diagram commute: 
\begin{equation}
	\begin{tikzcd}[row sep=large, column sep=1.5cm]
		V_{s(\alpha)}\arrow{r}{\phi_{s(\alpha)}}\arrow{d}{f_\alpha}
		& W_{s(\alpha)} \arrow{d}{g_\alpha} \\
		V_{t(\alpha)} \arrow{r}{\phi_{t(\alpha)}} 
		& W_{t(\alpha)}
	\end{tikzcd}
\end{equation}
We denote by $\textrm{Rep}(Q,\mathbb{F}_1)$ the category whose objects are $\FF_1$-representations of $Q$ and whose morphisms are defined as above. We let $\textrm{Rep}(Q,\mathbb{F}_1)_{\textrm{nil}}$ be the full subcategory of $\textrm{Rep}(Q,\mathbb{F}_1)$ consisting of nilpotent representations. 
\end{mydef}

We note that each morphism $\Phi \in \Hom(\mathbb{V},\mathbb{W})$ has a kernel and cokernel obtained from a kernel and cokernel at each vertex. Similarly, one obtains the notions of \emph{subrepresentations} and \emph{quotient representations}. We refer the reader to \cite[Definition 4.3]{szczesny2011representations} for details. 

\begin{rmk}
Recall that an $\FF_1$-representation $\mathbb{V}$ of a quiver $Q$ is \emph{indecomposable} if $\mathbb{V}$ cannot be written as a nontrivial direct sum of subrepresentations. In \cite{szczesny2011representations}, Szczesny proves that the Krull-Schmidt Theorem holds for the category $\Rep(Q,\FF_1)$. To be precise, any $\FF_1$-representation $M$ can be written uniquely (up to permutation) as a finite direct sum of indecomposable representations. 
\end{rmk}

Let $k$ be a field and $\Vect(k)$ be the category of finite-dimensional vector spaces over $k$. Then one may define ``the base change functor'' as follows\footnote{Tom Zaslavsky suggested to use the term ``basis functor'' as it does not change bases. In $\FF_1$-geometry, the functor was first introduced to define the base change from an algebraic variety over $\FF_1$ to an algebraic variety over a field. For this reason (to be compatible with already existing convention), we use ``base change functor''.}:
\begin{equation}\label{eq: base change1}
k\otimes_{\FF_1} -:\textrm{Vect}(\mathbb{F}_1) \to \textrm{Vect}(k),
\end{equation}
where any $\FF_1$-vector space $V$ goes to the vector space whose basis is $V-\{0_V\}$.

Note that representations of a quiver can be defined in a more categorical way. Let $Q$ be a quiver. One can consider a discrete category $\mathcal{Q}$: objects are vertices of $Q$ and morphisms are directed paths. Then, a representation $M$ of $Q$ over $\FF_1$ is nothing but a functor $\mathbf{M}: \mathcal{Q} \to \textrm{Vect}(\mathbb{F}_1)$. In particular, $\textrm{Rep}(Q,\FF_1)$ is equivalent to the functor category $\textrm{Vect}(\mathbb{F}_1)^{\mathcal{Q}}$. In fact, the same description holds for representations of $Q$ over a field $k$. Therefore, from the base change functor $k\otimes_{\FF_1} -$,
one has the following base change functor which is faithful (but not full in general):
\begin{equation}\label{eq: base change2}
k\otimes_{\FF_1} -:\textrm{Rep}(Q,\FF_1) \to \textrm{Rep}(Q,k), \quad M \to M_k. 
\end{equation}
The base change functor will be considered in \S \ref{s: Euler}  to compute Euler characteristic of quiver Grassmannians associated to a class of quiver representations. 

\subsection{Hall algebras for $\textrm{Rep}(Q,\FF_1)$} \label{subsection: Hall algebras}

There are two (equivalent) ways to construct the Hall algebra $H_Q$ of $\textrm{Rep}(Q,\FF_1)$. The first way is to appeal to some categorical interpretation of $\FF_1$-representations of $Q$. As we mentioned above, $\textrm{Rep}(Q,\FF_1)$ is equivalent to the functor category $\textrm{Vect}(\mathbb{F}_1)^{\mathcal{Q}}$. Since $\Vect(\FF_1)$ is \emph{proto-exact} in the sense of Dyckerhoff and Kapranov \cite{dyckerhoff2012higher}, where they also prove that for a small category $\mathcal{I}$ the functor category $\mathcal{C}^\mathcal{I}$ is proto-exact for a proto-exact category $\mathcal{C}$. The construction of $H_Q$ then follows from a more general construction of Hall algebras in \cite{dyckerhoff2012higher}. 

The second construction is to mimic the classical construction of the Hall algebra of representations of $Q$ over $\FF_q$. To be precise, let $\textrm{Iso}(Q)$ be the set of isomorphism classes of objects in $\textrm{Rep}(Q,\mathbb{F}_1)$. The Hall algebra $H_Q$ of $\textrm{Rep}(Q,\mathbb{F}_1)$ has the following underlying set:
\begin{equation}\label{eq: hall algebra}
	H_Q:=\{f:\textrm{Iso}(Q) \to \mathbb{C} \mid f([M])=0\textrm{ for all but finitely many $[M]$.}\}
\end{equation}
For each $[M] \in \textrm{Iso}(Q)$, let $\delta_{[M]}$ be the delta function in $H_Q$ supported at $[M]$. In particular, we consider $H_Q$ as the vector space spanned by $\{\delta_{[M]}\}_{[M] \in \textrm{Iso}(Q)}$ over $\mathbb{C}$. To ease the notation, we will denote the delta function $\delta_{[M]}$ by $[M]$. One defines the following multiplication on the elements in $\textrm{Iso}(Q)$:
\begin{equation}\label{eq: hall product}
	[M]\cdot[N]:=\sum_{R \in \textrm{Iso}(Q)} \frac{a^R_{M,N}}{a_Ma_N}[R],
\end{equation}
where $a_M=|\textrm{Aut}(M)|$ and $a^R_{M,N}$ is the number of ``short exact sequences'' of the form:\footnote{As in the classical case, by a short exact sequence we mean that $\ker=\textrm{coker}$.}
\[
0 \to N \to R \to M \to 0.
\]
Then, one can easily check the following equality as in the classical case:
\[
\frac{a^R_{M,N}}{a_Ma_N}=|\{L \leq R \mid L\simeq N\textrm{ and } R/L \simeq M\}|.
\]
By linearly extending the multiplication \eqref{eq: hall product} to $H_Q$, we obtain an associative algebra $H_Q$ over $\mathbb{C}$. Moreover, one may check that $H_Q$ is also equipped with the coproduct defined as follows:
\begin{equation}\label{eq: hall coprod}
	\Delta:H_Q\to H_Q\otimes_\mathbb{C}H_Q, \quad \Delta(f)([M],[N])=f([M\oplus N]).
\end{equation}

With \eqref{eq: hall product} and \eqref{eq: hall coprod}, Szczesny proves various interesting results. We refer the interested reader to \cite[Section 6]{szczesny2011representations} for details. 

\subsection{Coefficient quivers} \label{subsection: coefficient quivers}

Coefficient quivers\footnote{We emphasize that even if we are using the same terminology ``coefficient quivers'' our notion of coefficient quivers is different from that of Ringel.} were first introduced by Ringel \cite{ringel1998exceptional} as a combinatorial gadget to study representations of quivers. 

Let $\mathbb{V}=(V_i,f_\alpha)$ be a representation of a quiver $Q$ over $\mathbb{C}$. We fix a basis $B(i)$ for each vector space $i$ and let $B=\bigsqcup_{i \in Q_0}B(i)$, i.e., $B$ is a basis for the vector space $\bigoplus_{i \in Q_0}V_i$. We simply call $B$ a basis for $\mathbb{V}$. 

\begin{mydef}\label{definition: coefficient quiver}
The coefficient quiver $\tilde{Q}=\tilde{Q}(\mathbb{V},B)$ is a quiver defined as follows:
\begin{enumerate}
	\item 
$\tilde{Q}_0=B$. 
\item 
For every arrow $\alpha:v \to w$ of $Q$ and every element $x \in B(v)$ if we can write
\[
f_\alpha(x)=\sum c_b b, \quad b \in B(w), c_b\neq 0, 
\]
then we draw an arrow from $x$ to $b \in B(w)$ in $\tilde{Q}$, and label it with $\alpha$. 
\end{enumerate}
\end{mydef}

The coefficient quiver depends on a choice of a basis for a representation $\mathbb{V}$ of $Q$.

\begin{myeg}
	\[
Q=\left( 
\begin{tikzcd}[row sep=2em]
	v_1  \arrow[dr,swap,"\alpha"]
	& v_2 \arrow[d,"\beta"] & \\
	& v_{3} \arrow[loop right,"\gamma"]
\end{tikzcd}
\right)
\]	
Consider the following representation $\mathbb{V}$ of $Q$:
\[
V_1=\mathbb{C}^2, \quad V_2=\mathbb{C}, \quad V_3=\mathbb{C}^2.
\]
\[
f_\alpha=\begin{bmatrix}
1 & 0 \\
0 & 0 
\end{bmatrix}, \quad f_\beta=\begin{bmatrix}
0  \\
1 
\end{bmatrix}, \quad f_\gamma=\begin{bmatrix}
0 & -1 \\
0 & 0
\end{bmatrix}
\]
Let's fix bases: $B_1=\{e_1,e_2\}$, $B_2=\{1\}$, and $B_3=\{e_1+e_2,e_2\}$. Then we obtain the following coefficient quiver. 
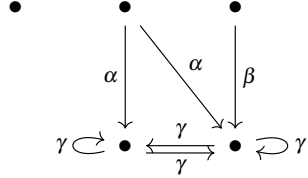
\[
\begin{tikzcd}
\bullet	& \bullet \arrow[dd,swap, "\alpha"] \arrow[ddr, "\alpha"]& \bullet \arrow[dd,"\beta"] \\
&  &  &\\
& \bullet \arrow[loop left,looseness=10,"\gamma"] \arrow[r,swap, shift right,"\gamma"] & \bullet \arrow[l,swap,"\gamma"] \arrow[loop right,looseness=10,"\gamma"]
\end{tikzcd} 
\]
Let's change bases: $B_1'=\{e_1-e_2,e_2\}$, $B_2'={1}$, and $B_3'=\{e_1,e_2\}$. Then we have the following coefficient quiver:
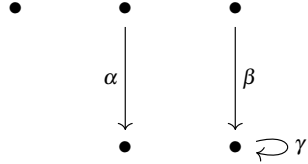
\[
\begin{tikzcd}
	\bullet	& \bullet \arrow[dd,swap, "\alpha"] & \bullet \arrow[dd,"\beta"] \\
	&  &  &\\
	& \bullet & \bullet    \arrow[l,"\gamma"] 
\end{tikzcd} 
\]
\end{myeg}


\section{The slice category over Q and $\textrm{Rep}(Q,\FF_1)$}\label{s: slice}

A notion of \emph{windings} of quivers was first introduced by Crawley-Boevey \cite{crawley1989maps} and Krause \cite{krause1991maps} to define morphisms between tree and band modules. Later, Haupt \cite[Section 2.3]{Haupt2012euler} generalized Krause's definition of windings as follows.\footnote{Crawley-Boevey considered tree modules and Krause considered tree and band modules, and they have one more condition. For instance, Krause \cite{krause1991maps} has an extra condition (W2).} Let $Q$ and $S$ be quivers. A winding of quivers $F:S\to Q$  is a morphisms of quivers
\[
F_0:S_0 \to Q_0, \quad F_1:S_1 \to Q_1
\]
satisfying the following two conditions:
\begin{enumerate}
	\item 
If $\alpha,\beta \in S_1$ with $\alpha \neq \beta$ and $s(\alpha)=s(\beta)$, then $F_1(\alpha) \neq F_1(\beta)$. 	
	\item 
If $\alpha,\beta \in S_1$ with $\alpha \neq \beta$ and $t(\alpha)=t(\beta)$, then $F_1(\alpha) \neq F_1(\beta)$. 	
\end{enumerate}

For a winding map $F:S \to Q$, it is easy to check that for an $S$-representation $V$ over a field $k$, the pushforward $F_*(V)$ is a $Q$-representation over $k$.\footnote{See Remark \ref{remar: scalar extension} for the definition of the pushforward.} In particular, $F_*(\mathbbm{1}_S)$ is said to be a \emph{tree module} if $S$ is a tree, where $\mathbbm{1}_S$ is the representation of $S$ assigning $k$ to each vertex and the identity map to each arrow. The following lemma shows that the same is true for $\mathbb{F}_1$-representations. 

\begin{lem}\label{lemma: pushforward reps}
Let $F:S \to Q$ be a winding map of quivers. For each $v \in Q_0$, let $M_v=F^{-1}(v) \cup \{0\}$. For each $\alpha \in Q_1$, consider the following function: let $v=s(\alpha)$, $w=t(\alpha)$, 
\begin{equation}\label{eq: map}
\tilde{\alpha}: M_v \to M_w, \quad x \mapsto \begin{cases}
	y, \emph{ if } ~\exists \beta \in S_1 \emph{ such that } s(\beta)=x,~t(\beta)=y,~F(\beta)=\alpha,\\
	0, \emph{ otherwise.}
\end{cases}
\end{equation}
Then $(M_v,\tilde{\alpha})_{v \in Q_0, \alpha \in Q_1}$ is an $\mathbb{F}_1$-representation of $Q$. By abuse of notation, we denote this $\mathbb{F}_1$-representation of $Q$ by $F_*(S)$. 
\end{lem}
\begin{proof}
We only have to check that $\tilde{\alpha}$ is a well-defined $\mathbb{F}_1$-linear map. In fact, suppose that we have $\beta,\beta' \in S_1$ and $x \in M_v$ such that $s(\beta)=s(\beta')=x$ and $F(\beta)=F(\beta')=\alpha \in Q_1$. Since $F$ is a winding map, this implies that $\beta=\beta'$, and hence $\tilde{\alpha}$ is well defined.

Next, suppose that $\tilde{\alpha}(x)=\tilde{\alpha}(z) = e\neq 0$. In other words, there exist $\beta_x, \beta_z \in S_1$ such that 
\[
s(\beta_x)=x, \quad s(\beta_z)=z, \quad t(\beta_x)=e=t(\beta_z), \quad F(\beta_x)=F(\beta_z)=\alpha. 
\]
But, again since $F$ is a winding map, this implies that $\beta_x=\beta_z$, showing that $x=z$. Hence, $\tilde{\alpha}$ is an $\mathbb{F}_1$-linear map. 
\end{proof}

\begin{lem}\label{lemma: essentially surjective}
Let $Q$ be a quiver and $\mathbb{V}=(M_v,f_\alpha)$ an $\FF_1$-representation of $Q$. Then, there exists a winding map of quivers $F:S\to Q$ such that $F_*(S) \simeq \mathbb{V}$. 
\end{lem}
\begin{proof}
This construction is essentially same as the one given in \cite{jun2020quiver}, but we include it here for completeness. We first define the set of vertices of a quiver $S$ as follows:
\[
S_0:= \bigsqcup_{v \in Q_0}M_v\setminus \{0\}.
\]	
For each $x, y \in S_0$, we draw an arrow $\beta_\alpha:x \to y$ in $S_1$ if and only if $x \in M_v$, $y \in M_w$ and there exist $\alpha \in Q_1$ and $f_\alpha$ such that $s(\alpha)=v$, $t(\alpha)=w$, and $f_\alpha(x)=y$. Note that if there is another $\alpha' \in Q_1$ with the same property then we draw two different arrows $\beta_\alpha$ and $\beta_{\alpha'}$. 

Next, we define a winding map $F:S \to Q$ as follows: for each $x_v \in S_0$, $x_v \in M_v \setminus \{0\}$ for some $v \in Q_0$, we let $F(x_v)=v$. We send each arrow $\beta_{\alpha}\in S_1$ to $\alpha$. One can easily check that $F$ is a quiver map. To check the winding condition, suppose  $\beta_{\alpha} \neq\beta_{\alpha'} \in S_1$ and $s(\beta_{\alpha})=s(\beta_{\alpha'})$. If $t(\beta_{\alpha}) = t(\beta_{\alpha'})$, then $\beta_{\alpha} \neq \beta_{\alpha'}$ implies $F(\beta_{\alpha}) =\alpha \neq \alpha' = F(\beta_{\alpha'})$, since there is at most one $\alpha$-labeled arrow between any two vertices of $S$. Now, suppose that $t(\beta_{\alpha}) \neq t(\beta_{\alpha'})$. Since $t(\beta_{\alpha}) = f_{\alpha}(s(\beta_{\alpha}))$ and $t(\beta_{\alpha'}) = f_{\alpha'}(s(\beta_{\alpha'}))$, we must have $\alpha \neq \alpha'$, which implies $F(\beta_{\alpha}) \neq F(\beta_{\alpha'})$ again. Hence, the first condition for $F$ to be a winding has been verified. The second condition is similar. 
\end{proof}

To define the category of quivers over a quiver, we first recall some definitions. Let $T$ be a full subquiver of $S$. We say that the $T_0$ is \emph{predecessor closed} if the following condition holds: for any oriented path in $S$ from $v$ to $w$ if $w \in T_0$, then $v \in T_0$. Analogously $T_0$ is \emph{successor closed} if the following condition holds: for any oriented path in $S$ from $v$ to $w$ if $v \in T_0$, then $w \in T_0$.

\begin{rmk}
We caution the reader that the authors use the \emph{opposite} convention for successor- and predecessor-closed subsets in \cite{jun2020quiver}. 
\end{rmk}

Let $Q$ be a fixed quiver. Let $\mathcal{C}_Q$ be the category whose objects are windings of quivers $F:S \to Q$. A morphism $\phi : (S,F)\rightarrow (S',F')$ is an ordered triple $\phi = (\mathcal{U}_{\phi}, \mathcal{D}_{\phi}, c_{\phi})$, where 
\begin{enumerate} 
\item $\mathcal{U}_{\phi}$ is a full subquiver of $S$ whose vertex set is predecessor closed,
\item $\mathcal{D}_{\phi}$ is a full subquiver of $S'$ whose vertex set is successor closed, and
\item $c_{\phi} : \mathcal{U}_{\phi} \rightarrow \mathcal{D}_{\phi}$ is a quiver isomorphism such that the diagram below commutes. 
\end{enumerate}   
\begin{equation}\label{eq: morphism of windings} 
\begin{tikzcd}[row sep=2em]
\mathcal{U}_{\phi} \arrow[rr,"c_{\phi}"] \arrow[dr,swap,"F|_{\mathcal{U}_\phi}"]
&& \mathcal{D}_{\phi} \arrow[dl,,"F'|_{\mathcal{D}_\phi}"] \\
& Q
\end{tikzcd}
\end{equation}

If $(S,F) \xrightarrow[]{\phi} (S',F')$ and $(S',F') \xrightarrow[]{\psi} (S'', F'')$ are two morphisms in $\mathcal{C}_Q$, their composition $(S,F)\xrightarrow[]{\psi\circ\phi} (S'',F'')$ is the ordered triple
\begin{equation}\label{eq: composition} 
\psi \circ \phi =  \left( \mathcal{U}_{\psi\circ \phi}, \mathcal{D}_{\psi \circ \phi}, c_{\psi\circ\phi} \right) = \left( c_{\phi}^{-1}\left( \mathcal{U}_{\psi} \cap \mathcal{D}_{\phi} \right), c_{\psi}\left( \mathcal{U}_{\psi}\cap \mathcal{D}_{\phi} \right), c_{\psi} \circ c_{\phi} \right).
\end{equation}
Of course, the composition $c_{\psi}\circ c_{\phi}$ is understood to be restricted to $c_{\phi}^{-1}\left( \mathcal{U}_{\psi} \cap \mathcal{D}_{\phi} \right)$. Loosely, one can think of composition as gluing the top of $S$ to the bottom of $S'$ in a way that respects the mappings $F$ and $F'$. One can check $\mathcal{C}_Q$ indeed satisfies the axioms of a category. 

Let $\phi:(S,F) \to (S',F')$ be a morphism in $\mathcal{C}_Q$. Then $\phi$ induces a morphism $\phi_*:F_*(S) \to F'_*(S')$ of $\FF_1$-representations of $Q$ as follows: for each $v \in Q_0$, we define the map
\begin{equation}\label{eq: 1}
(\phi_*)_v:F_*(S)_v\to F'_*(S')_v 
\end{equation}
as
\begin{equation}\label{eq: 2}
(\phi_*)_v(x) = \begin{cases}
	c_\phi(x), \textrm{ if $x \in \mathcal{U}_\phi \cap F_*(S)_v$,}\\
	0, \textrm{ otherwise.}
\end{cases}
\end{equation}
Since $c_\phi$ is an isomorphism, clearly $(\phi_*)_v$ is an $\FF_1$-linear map. Next, let $\alpha\in Q_1$ with $v=s(\alpha)$ and $w=t(\alpha)$. Suppose first that $x \not \in \mathcal{U}_\phi \cap F_*(S)_v$, in particular, $(\phi_*)_v(x)=0$. Since $x \not \in \mathcal{U}_\phi$ and $\mathcal{U}_\phi$ is predecessor closed, we have that $\tilde{\alpha}(x) \not \in \mathcal{U}_\phi$, where $\tilde{\alpha}$ is a map defined in \eqref{eq: map}. Hence, in this case, we have
\begin{equation}\label{eq: comm}
\tilde{\alpha}(\phi_*)_v(x) = (\phi_*)_w\tilde{\alpha}(x).
\end{equation}
Now, suppose that $x \in \mathcal{U}_\phi \cap F_*(S)_v$ and $y=c_\phi(x)$. If $\tilde{\alpha}(y)=0$, then $\tilde{\alpha}(x)\not\in \mathcal{U}_{\phi}$, since $c_{\phi}$ is an isomorphism. In particular, $(\phi_*)_w\tilde{\alpha}(x) = 0$, and we have \eqref{eq: comm} in this case. Finally, if $\tilde{\alpha}(y)=z$, then $z \in \mathcal{D}_\phi$ since $y \in \mathcal{D}_\phi$ and it is successor closed. In particular, there exists an arrow $\beta$ in $\mathcal{D}_\phi$ such that $s(\beta)=y$ and $t(\beta)=z$. Since $c_\phi$ is an isomorphism, this implies that $\tilde{\alpha}(x) \in \mathcal{U}_\phi$ and $(\phi_*)_w\tilde{\alpha}(x)=z$, showing that \eqref{eq: comm} is valid in this case as well. Therefore, the following diagram commutes and $\phi_*$ is indeed a morphism of $\FF_1$-representations. 
\begin{equation}
	\begin{tikzcd}[row sep=large, column sep=1.5cm]
		(F_*(S))_v\arrow{r}{(\phi_*)_v}\arrow{d}{\tilde{\alpha}}
		& (F'_*(S'))_v \arrow{d}{\tilde{\alpha}} \\
		(F_*(S))_w \arrow{r}{(\phi_*)_w} 
		& (F'_*(S'))_w
	\end{tikzcd}
\end{equation}

\begin{rmk}
Before we proceed to prove an equivalence between $\mathcal{C}_Q$ and $\Rep(Q,\FF_1)$, we remark the following. 
\begin{enumerate}
\item 
Our definition for morphisms in $\mathcal{C}_Q$ generalizes the notion of maps between tree modules in \cite[Section 2]{crawley1989maps}. We also note that the notion of \emph{admissable triple} in \cite{krause1991maps} is similar to our definition in the sense that to define a morphism $\phi:(S,F) \to (S',F')$ we first assume an isomorphism of a subquiver of $S$ and a subquiver of $S'$ satisfying certain conditions whereas in \cite{krause1991maps}, this is done by using ``connecting triple'' rather than directly identifying subquivers. For details, see \cite[pages 189-191]{krause1991maps}. 	
\item 
The notion of tree modules in \cite{crawley1989maps} and \cite{krause1991maps} is more restrictive than the notion of tree modules given in \cite{ringel1998exceptional} due to the ``winding'' conditions imposed on quiver maps. 
\end{enumerate}
\end{rmk}

\begin{lem}\label{lemma: functor}
For an object $F:S\to Q$ in $\mathcal{C}_Q$, we let $\mathbf{F}(S)$ be the $\mathbb{F}_1$-representation of $Q$ as in Lemma \ref{lemma: pushforward reps}. For a morphism $\phi:(S,F) \to (S',F)$, we let $\mathbf{F}(\phi)$ be the morphism between $\mathbf{F}(S)$ and $\mathbf{F}(S')$ which we described above. Then $\mathbf{F}:\mathcal{C}_Q \to \emph{Rep}(Q,\FF_1)$ defines a functor. 
\end{lem}
\begin{proof}
One can easily check that an identity map $\phi:(S,F) \to (S,F)$ in $\mathcal{C}_Q$ maps to the identity map in $\textrm{Rep}(Q,\FF_1)$ since in this case $\mathcal{U}_\phi=\mathcal{D}_\phi=S$ and $c_\phi$ is the identity map. 

Next, suppose that $\phi:(S,F) \to (S',F')$ and $\psi: (S',F') \to (S'', F'')$ are two morphisms in $\mathcal{C}_Q$. We want to check that $\mathbf{F}(\psi\phi)=\mathbf{F}(\psi)\mathbf{F}(\phi)$. With the same notation as in \eqref{eq: 1} and \eqref{eq: 2}, we only have to show that
\begin{equation}
((\psi\phi)_*)_v=((\psi)_*)_v((\phi)_*)_v.
\end{equation}
But, this is clear from \eqref{eq: composition}. 
\end{proof}

\begin{lem}\label{lemma: lifting}
Let $S=(S,F),S'=(S',F')$ be objects in $\mathcal{C}_Q$,  $\phi:F_*(S) \to F'_*(S')$ a morphism in $\emph{Rep}(Q,\FF_1)$, and $\mathcal{U}_\phi$ the full subquiver of $S$ with the vertex set $S_0\setminus \ker(\phi)$. Then $\phi$ induces a quiver map $f:\mathcal{U}_\phi \to S'$ such that $f(\mathcal{U}_\phi)$ is successor closed. 
\end{lem}
\begin{proof} 
This follows from Construction 3.9 and the proof of Lemma 3.10 in \cite{jun2020quiver}.
\end{proof}

\begin{pro}\label{proposition: equivalence of categories}
The functor $\mathbf{F}:\mathcal{C}_Q \to \emph{Rep}(Q,\FF_1)$ is an equivalence of categories. This restricts to an equivalence between $\Rep(Q,\FF_1)_{\nil}$ and the full subcategory of $\mathcal{C}_Q$ whose objects are windings $F:S \to Q$ with $S$ acyclic. 
\end{pro} 
\begin{proof} 
Lemmas \ref{lemma: essentially surjective} and \ref{lemma: functor} show that $\mathbf{F}$ is an essentially surjective functor, and hence we only have to prove that $\mathbf{F}$ is fully faithful. But this is just Lemma 3.10 of \cite{jun2020quiver}.\footnote{An explicit construction showing that the functor is full is recalled in the proof of Lemma \ref{lemma: lifting}.} 
\end{proof} 

In what follows, we often denote the winding corresponding to an $\FF_1$-representation $M$ of a quiver $Q$ by $c:\Gamma_M \to Q$. We will simply call $\Gamma_M$ \emph{the coefficient quiver} of $M$.

The following example is taken from \cite{Haupt2012euler} and is restated in terms of $\mathbb{F}_1$-representation.

\begin{myeg}\cite[Example 2.5]{Haupt2012euler}
Let $F:S \to Q$ be an object in $\mathcal{C}_Q$ described in the following picture:
\[
F:S= \left( 
\begin{tikzcd}[row sep=2em]
	v_1  \arrow[dr,swap,"\alpha"]
	& v_2 \arrow[d,swap,"\beta"]& v_3 \arrow[dl,swap,"\gamma"] \\
	& v_{3'}\arrow[r,swap,"\gamma'"] & v_{3''}
\end{tikzcd}
\right) \longrightarrow Q=\left( 
\begin{tikzcd}[row sep=2em]
	v_1  \arrow[dr,swap,"\alpha"]
	& v_2 \arrow[d,"\beta"] & \\
	& v_{3'} \arrow[loop right,"\gamma"]
\end{tikzcd}
\right)
\]	
Then, $\mathbbm{1}_S$ is the following representation of $S$ over a field $k$:
\[
\mathbbm{1}_S= \left( 
\begin{tikzcd}[row sep=2em]
	k  \arrow[dr,swap,"\id"]
	& k \arrow[d,swap,"\id"]& k \arrow[dl,swap,"\id"] \\
	& k\arrow[r,swap,"\id"] & k
\end{tikzcd}
\right)
\]
Note that by definition, $F_*(\mathbbm{1}_S)$ is a tree module, given as follows:
\[
F_*(\mathbbm{1}_S)=\left( 
\begin{tikzcd}[row sep=2em]
	k \arrow[dr,swap,"A"]
	& k \arrow[d,"B"] & \\
	& k^3 \arrow[loop right,"C"]
\end{tikzcd}
\right)
\] 
where
\[
A=\begin{bmatrix}
	 0 \\
	 1 \\
	 0
\end{bmatrix}, \quad B=\begin{bmatrix}
0 \\
1\\
0
\end{bmatrix}, \quad C=\begin{bmatrix}
0 & 0 & 0 \\
1 & 0 & 0\\
0 & 1 & 0
\end{bmatrix}
\]
Now, we view $S$ as the coefficient quiver of an $\FF_1$-representation. The $\FF_1$-representation $\mathbb{V}=F_*(S)$ of $Q$ is the following: at each vertex of $Q$, we have
\[
M_{v_1}=\{0, v_1\}, \quad M_{v_2}=\{0, v_2\},\quad  M_{v_3}=\{0,v_3,v_{3'},v_{3''}\}.
\]
$\FF_1$-linear maps between vertices are given as follows:
\[
\tilde{\alpha}:M_{v_1} \to M_{v_3}, \quad v_1 \mapsto v_{3'},
\]
\[
\tilde{\beta}:M_{v_2} \to M_{v_3}, \quad v_2 \mapsto v_{3'},
\]
\[
\tilde{\gamma}: M_{v_3} \to M_{v_3}, \quad v_3 \mapsto v_{3'}, \quad v_{3'} \mapsto v_{3''}, \quad v_{3''} \mapsto 0.
\]
One can easily see that $\mathbb{V}_\mathbb{C} = F_*(\mathbbm{1}_S)$. 

Now, we depict the corresponding coefficient quiver $\Gamma_\mathbb{V}$ for $\mathbb{V}$. First, we consider the following coloring. 
\[
v_1=\mycirc, \quad v_2=\mycirc[red], \quad v_3=v_{3'}=v_{3''}=\mycirc[blue]
\]
We let $\alpha$ in green, $\beta$ in purple, and $\gamma$ in black. Then the coefficient quiver $\Gamma_\mathbb{V}$ is as follows:
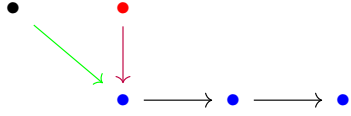
\[
\begin{tikzcd}[row sep=2em]
&	\mycirc  \arrow[green, dr,swap]
	& \mycirc[red] \arrow[purple,d,swap] \\
	& \mycirc[blue]\arrow[black,r,swap] & \mycirc[blue]\arrow[black,r,swap] & \mycirc[blue]
\end{tikzcd}
\]
\end{myeg}


\begin{rmk}\label{remar: scalar extension}
Let $f:S \to Q$ be a quiver map (without winding condition) and $k$ be a field. The map $f$ induces a functor (pushforward) 
\begin{equation}
f_*:\Rep(S,k) \to \Rep(Q,k)
\end{equation}
which we briefly recall here. We refer the read to \cite{kinser2010rank,kinser2013tree} for more details. For a representation $V$ of $S$ over $k$, the pushforward $f_*(V)$ is defined as follows:
\begin{equation}\label{eq: pushforward}
f_*(V)v:=\bigoplus_{w \in f^{{-1}}(v)} V_w, \quad v \in Q_0, \qquad f_*(V)_\alpha=\sum_{\beta \in f^{-1}(\alpha)} V_\beta, \quad \beta \in Q_1
\end{equation}
Recall that by the identity representation $\mathbbm{1}_S$ of $S$, we mean a representation of $S$ over $k$ consisting of one dimensional vector space at each vertex and identity map at each arrow. In particular, each quiver map $f:S \to Q$, defines a representation of $Q$, namely $f_*(\mathbbm{1}_S)$. 

Let $\mathcal{D}_Q$ be the category whose objects are quiver maps $f:Q' \to Q$ (without winding condition) and morphisms are same as $\mathcal{C}_Q$. Then we have a faithful inclusion functor $i:\mathcal{C}_Q \to \mathcal{D}_Q$.

For each $f:S \to Q$ in $\mathcal{D}_Q$, we let $\mathbbm{1}(S)=f_*(\mathbbm{1}_S)$. For a morphism $\phi:S \to S'$ in $\mathcal{D}_Q$, where $S$ (resp.~$S'$) means $f:S \to Q$ (resp.~$f':S' \to Q$). In fact, the same definition as in \eqref{eq: 1} and \eqref{eq: 2} can be used to define a functor:
\[
\mathbbm{1}:\mathcal{D}_Q \to \Rep(Q,k), \quad (f:S\to Q) \mapsto f_*(\mathbbm{1}_S). 
\]
From the definition, one can easily check that the functor $\mathbbm{1}$ is faithful. To summarize, we have the following commutative diagram of categories. 

\begin{equation}\label{eq: scalar extension}
	\begin{tikzcd}[row sep=large, column sep=1.5cm]
		\mathcal{C}_Q\arrow{r}{\mathbf{F}}\arrow[swap,hook]{d}{i}
		& \Rep(Q,\FF_1) \arrow[hook]{d}{\otimes_{\FF_1}k} \\
		\mathcal{D}_Q \arrow[hook,swap]{r}{\mathbbm{1}}
		& \Rep(Q,k)
	\end{tikzcd}
\end{equation}  
\end{rmk}

\begin{mydef}\label{d: trees and bands} 
Let $Q$ be a quiver. An $\FF_1$-representation $M$ is called a \emph{tree module} if $\Gamma_M$ is a tree. This is equivalent to $M\otimes_{\FF_1}{k} \cong f_*(\mathbbm{1}_S)$ with $f : S\rightarrow Q$ a winding and $S$ a tree. $M$ will be called an \emph{$\FF_1$-band module} if the coefficient quiver of $M$ is an affine Dynkin quiver of type $\tilde{\mathbb{A}}$. 
\end{mydef} 

\begin{rmk} 
Let $M$ be an $\FF_1$-band. Then $\Gamma_M$ is connected, so $M$ is always indecomposable as an $\FF_1$-representation. However, since we do not require that the associated winding map $c_M : \Gamma_M \rightarrow Q$ is primitive (see Definition \ref{definition: primitive}), $M\otimes_{\FF_1}k$ may be decomposable. If $M$ is an $\FF_1$-band module and $c_M$ is primitive then $M\otimes_{\FF_1}k$ is a band module in the usual sense, but the converse does not generally hold.
\end{rmk}


\section{Gradings on representations}\label{s: gradings}

In \cite{irelli2011quiver}, Cerulli Irelli proved that when a quiver representation over $\mathbb{C}$ satisfies certain conditions\footnote{To be precise, Cerulli Irelli considered the coefficient quiver of a representation in a fixed basis.}, then one can compute the Euler characteristics of quiver Grassmannians for some special classes of quivers in a purely combinatorial way. Later, in \cite{Haupt2012euler}, Haupt generalized Cerulli Irelli's results by introducing the notion of a \emph{grading} on a representation of $Q$. The following appear as Definitions 4.1 and 4.2 of \cite{Haupt2012euler}.

\begin{mydef}[Haupt]\label{definition: gradings}
Let $M$ be an $\FF_1$-representation of $Q$, and let $(\Gamma,c) := (\Gamma_M,c_M)$ denote the associated coefficient quiver. By a \emph{grading} of $M$, we mean a map $\partial : \Gamma_0 \rightarrow \ZZ$. Suppose that $\partial_1,\ldots, \partial_n$ and $\partial$ are gradings for $M$. Suppose further that for any two arrows $\beta, \beta' \in \Gamma_1$, the equalities 
\begin{equation}\label{eq: con1}
\partial_i(s(\beta)) = \partial_i(s(\beta')),\text{ $i = 1,\ldots , n$}
\end{equation}
\begin{equation}\label{eq: con2}
\partial_i(t(\beta)) = \partial_i(t(\beta')), \text{ $i = 1,\ldots, n$}
\end{equation}
\begin{equation}\label{eq: con3}
c(\beta) = c(\beta')
\end{equation}
imply 
\begin{equation}\label{eq.nice}
\partial(t(\beta)) - \partial(s(\beta)) = \partial(t(\beta')) - \partial(s(\beta')). 
\end{equation}
Then we say that $\partial$ is a \emph{nice $(\partial_1,\ldots,\partial_n)$-grading}. A \emph{nice $\emptyset$-grading} (or \emph{nice grading} for short) is a grading such that $c(\beta) = c(\beta')$ implies \eqref{eq.nice}.
\end{mydef}

\begin{rmk} 
Haupt's original definition applies to representations over fields, where it is necessary to first specify a basis for $M$. Since representations over $\mathbb{F}_1$ have a unique basis, the definition above is unambiguous. In other words, a nice grading for an $\mathbb{F}_1$-representation $M$ is the same thing as a nice grading for $M_{\mathbb{C}}$ with respect to the basis $M\setminus \{0\}$, and so on. 
\end{rmk}

\begin{rmk} 
An $\FF_1$-representation $M$ is completely determined by its associated winding map $c_M$. Hence, we will use the terms ``nice grading for $M$'' and ``nice grading for $c_M : \Gamma_M \rightarrow Q$'' interchangeably. From Proposition \ref{proposition: equivalence of categories}, if $c : \Gamma\rightarrow Q$ is a winding then there is a unique (up to isomorphism) $\FF_1$-representation of $Q$, call it $M$, such that $(c_M, \Gamma_M) = (c,\Gamma)$. Hence, we can also discuss a nice grading for a general winding, without explicit reference to its associated representation. 
\end{rmk}

\begin{mydef} 
Let $F : S \rightarrow Q$ be a winding, and $\partial : S_0 \rightarrow \ZZ$ a nice grading. If $\beta \in S_1$ with $F(\beta) = \alpha$, define $\Delta_{\alpha}^{\partial} = \partial (t(\beta)) - \partial (s(\beta))$.\footnote{$\Delta_{\alpha}^{\partial}$ is well-defined since $\partial$ is a nice grading.} If $\partial$ is understood from context, we will abbreviate $\Delta_{\alpha}^{\partial} = \Delta_{\alpha}$. We say that $\partial$ is 
\begin{enumerate} 
\item \emph{non-trivial} if $\Delta_{\alpha} \neq 0$ for some $\alpha \in Q_1$,
\item \emph{non-degenerate} if $\Delta_{\alpha} \neq 0$ for \emph{all} $\alpha \in Q_1$, and  
\item \emph{positive} (resp. \emph{negative}) if $\Delta_{\alpha} > 0$ (resp. $\Delta_{\alpha} < 0$) for all $\alpha \in Q_1$. 
\end{enumerate} 
More generally, suppose that $\partial$ is a $(\partial_1,\ldots , \partial_n)$-nice grading. If $\beta \in S_1$ with 
\begin{equation} 
F_1(\beta) = \alpha
\end{equation} 
\begin{equation} 
(\partial_1(s(\beta)),\ldots , \partial_n(s(\beta))) = \textbf{s} \in \mathbb{Z}^n
\end{equation} 
\begin{equation} 
(\partial_1(t(\beta)),\ldots , \partial_n(t(\beta))) = \textbf{t} \in \mathbb{Z}^n,
\end{equation} 
then we define  
\begin{equation}
\Delta_{\alpha,\textbf{s},\textbf{t}}^{\partial} := \partial(t(\beta)) - \partial(s(\beta)). 
\end{equation} 
As before, when $\partial$ is understood we abbreviate $\Delta_{\alpha,\textbf{s},\textbf{t}}^{\partial} = \Delta_{\alpha,\textbf{s},\textbf{t}}$. The notions of non-trivial~/~non-degenerate~/~positive~/~negative $(\partial_1,\ldots ,\partial_n)$-nice gradings are defined in the obvious manner. 
\end{mydef}         

As in \cite{Haupt2012euler}, we will be interested in building sequences of nice gradings which have certain desirable properties. The definition below helps us formalize this process.

\begin{mydef} 
Let $c : \Gamma \rightarrow Q$ be a winding. A \emph{nice sequence} for $c$ is a sequence $\underline{\partial} = ( \partial_i)_{i=0}^{\infty}$ of maps $\Gamma_0 \rightarrow \mathbb{Z}$ such that 
\begin{enumerate} 
\item $\partial_0$ is a nice grading.
\item For all all $i>0$, $\partial_i$ is a $(\partial_0,\ldots , \partial_{i-1})$-nice grading.
\end{enumerate} 
Note that any finite sequence $\{\partial_i\}_{i=1}^n$ satisfying the conditions above can be extended to a nice sequence $\underline{\partial} = \{ \partial_i\}_{i=1}^{\infty}$ by defining $\partial_i = 0$ for $i>n$ (such a completion is not unique). 
\end{mydef}

\begin{mydef}\label{def: distinguish}
Let $c : \Gamma \rightarrow Q$ be a winding, and let $(\partial_i)$ be a nice sequence for $c$. If $x$ and $y$ are distinct vertices of $\Gamma$, we say that $(\partial_i)$ \emph{distinguishes $x$ and $y$} if there exists an index $i \in \mathbb{N}$ (depending on $x$ and $y$) such that $\partial_i(x) \neq \partial_i(y)$. We say that $(\partial_i)$ \emph{distinguishes vertices} if it distinguishes each pair of distinct vertices in $\Gamma$. Of course, if $(\partial_i)$ distinguishes vertices and $\Gamma$ has finitely-many vertices, then there exists an $N \in \mathbb{N}$ such that for all distinct $x$ and $y$, there exists an $i \le N$ satisfying $\partial_i(x) \neq \partial_i(y)$. 
\end{mydef}

\begin{myeg}\label{ex: string}
Let $Q = \wild_2$ with arrow set $Q_1 = \{\alpha_1,\alpha_2\}$. Let $M$ be the $\FF_1$-representation whose coefficient quiver is the following: 
\begin{equation}
\Gamma_M = 
\begin{tikzcd} 
	\bullet \arrow[r,blue,"\alpha_1"] & \bullet \arrow[r,red,"\alpha_2"] & \bullet & \bullet \arrow[l,blue,"\alpha_1",swap] & \bullet \arrow[l,red,"\alpha_2",swap] 
\end{tikzcd}
\end{equation}
where the winding $c:\Gamma_M \to Q$ sends the blue arrows (resp.~red arrows) to $\alpha_1$ (resp.~$\alpha_2$). 
\begin{enumerate}
	\item 
One can easily check that the following is a nice grading $\partial_0$ on $
\Gamma_M$:  
\begin{equation}\label{eq: e1}
\begin{tikzcd} 
	0 \arrow[r,blue,"+1"] & 1 \arrow[r,red,"+2"] & 3 & 2 \arrow[l,blue,"+1",swap] & 0 \arrow[l,red,"+2",swap] 
\end{tikzcd}
\end{equation}
where numbers on the arrows are $\partial_0(t(\beta))-\partial_0(s(\beta))$ for each arrow $\beta$ of $\Gamma_M$. 
\item 
The following is a $\partial_0$-nice grading $\partial_1$ on $M$ that is not nice: 
\begin{equation}\label{eq: e2}
\begin{tikzcd} 
	0 \arrow[r,blue,"+1"] & 1 \arrow[r,red,"+1"] & 2 & 3 \arrow[l,blue,"-1",swap] & 4 \arrow[l,red,"-1",swap] 
\end{tikzcd}
\end{equation}
Note that with the grading $\partial_0$ as in \eqref{eq: e1}, we only have to consider condition \eqref{eq: con3} when assigning images to the vertices. When building a $\partial_0$-nice grading as in \eqref{eq: e2}, no two arrows satisfy conditions \eqref{eq: con1}-\eqref{eq: con3} simultaneously, so that any integer function on the vertices is permissible. Note that the nice sequence $\underline{\partial} = (\partial_0, \partial_1,0,0,\ldots )$ distinguishes vertices.
\end{enumerate}
\end{myeg}

\begin{myeg}\label{ex: string2} 
Let $Q = \wild_3$ with arrow set $Q_1 = \{\alpha_1, \alpha_2, \alpha_3\}$. In the free group generated by the arrows of $Q$, set $p$ to be the element  
\[
p = \alpha_1[\alpha_2,\alpha_3]\alpha_1^{-1}[\alpha_3,\alpha_2]\alpha_1[\alpha_3,\alpha_2]\alpha_1^{-1}[\alpha_2,\alpha_3]. 
\] 
Here, $[\gamma,\delta] = \gamma \delta\gamma^{-1}\delta^{-1}$ is the usual commutator. We may consider $p$ to be a walk in $Q$, defining an $\mathbb{F}_1$-representation $M$ whose coefficient quiver is: 
\begin{equation}
\Gamma_M = 
\begin{tikzcd}  
 & \bullet \arrow[d,blue,"\alpha_1",swap] & & & & \\
	& \bullet \arrow[r,red,"\alpha_2"] & \bullet \arrow[r,green,"\alpha_3"] & \bullet & \bullet \arrow[l,red,"\alpha_2",swap] & \bullet \arrow[l,green,"\alpha_3",swap]  \\ 
  & \arrow[d,blue,"\alpha_1",swap] \bullet  \arrow[r,red,"\alpha_2"] & \bullet \arrow[r,green,"\alpha_3"] & \bullet  & \arrow[l,red,"\alpha_2",swap] \bullet & \arrow[u,blue,"\alpha_1",swap] \arrow[l,green,"\alpha_3",swap] \bullet  \\ 
 & \bullet \arrow[r,green,"\alpha_3"]   &  \bullet \arrow[r,red,"\alpha_2"] &  \bullet  & \arrow[l,green,"\alpha_3",swap] \bullet & \arrow[l,red,"\alpha_2",swap] \bullet \\ 
& \bullet \arrow[r,green,"\alpha_3"] & \bullet \arrow[r,red,"\alpha_2"] & \bullet & \arrow[l,green,"\alpha_3",swap] \bullet  & \arrow[l,red,"\alpha_2",swap] \arrow[u,blue,"\alpha_1",swap] \bullet \\
\end{tikzcd}
\end{equation} 
Using the notational conventions of the previous example, we can define a nice grading $\partial_0$ on $M$ as follows:  
\begin{equation} 
\begin{tikzcd}  
 & 0 \arrow[d,blue,"+1",swap] & & & & \\
	& 1 \arrow[r,red,"+10"] & 11 \arrow[r,green,"+100"] & 111 & 101 \arrow[l,red,"+10",swap] & 1 \arrow[l,green,"+100",swap]  \\ 
  & \arrow[d,blue,"+1",swap] 0  \arrow[r,red,"+10"] & 10 \arrow[r,green,"+100"] & 110  & \arrow[l,red,"+10",swap] 100 & \arrow[u,blue,"+1",swap] \arrow[l,green,"+100",swap] 0  \\ 
 & 1 \arrow[r,green,"+100"]   &  101 \arrow[r,red,"+10"] &  111  & \arrow[l,green,"+100",swap] 11 & \arrow[l,red,"+10",swap] 1 \\ 
& 0 \arrow[r,green,"+100"] & 100 \arrow[r,red,"+10"] & 110 & \arrow[l,green,"+100",swap] 10  & \arrow[l,red,"+10",swap] \arrow[u,blue,"+1",swap] 0 \\
\end{tikzcd}
\end{equation} 
We now define a $\partial_0$-nice grading which is not itself nice. Informally, the conditions \eqref{eq: con1}-\eqref{eq: con3} on a $\partial_0$-nice grading $\partial_1$ state that whenever two arrows with the same color start at the same number and end at the same number, their increments from the source to the target must be equal. Clearly, all four of the $\alpha_1$-colored arrows require the same increment. However, one can check that each remaining arrow is only required to share an increment with one other arrow. For instance, the following defines a $\partial_0$-nice grading $\partial_1$:  
\begin{equation} 
\begin{tikzcd}  
 & 0 \arrow[d,blue,"+1",swap] & & & & \\
	& 1 \arrow[r,red,"+1"] & 2 \arrow[r,green,"+1"] & 3 & 4 \arrow[l,red,"-1",swap] & 104 \arrow[l,green,"-100",swap]  \\ 
  & \arrow[d,blue,"+1",swap] 108  \arrow[r,red,"-1"] & 107 \arrow[r,green,"-1"] & 106  & \arrow[l,red,"+1",swap] 105 & \arrow[u,blue,"+1",swap] \arrow[l,green,"+2",swap] 103  \\ 
 & 109 \arrow[r,green,"-100"]   &  9 \arrow[r,red,"-1"] &  8  & \arrow[l,green,"+1",swap] 7 & \arrow[l,red,"+1",swap] 6 \\ 
& 0 \arrow[r,green,"+2"] & 2 \arrow[r,red,"+1"] & 3 & \arrow[l,green,"-1",swap] 4  & \arrow[l,red,"-1",swap] \arrow[u,blue,"+1",swap] 5 \\
\end{tikzcd}
\end{equation} 
The function $\partial_1$ still fails to distinguish several pairs of vertices, for instance the first and last vertices of the walk. However, note that any function now qualifies as a $(\partial_0,\partial_1)$-nice grading. For instance, there are exactly two arrows that start at $2$ and end at $3$, but one is $\alpha_2$-colored and the other is $\alpha_3$-colored: it follows that they violate Condition \eqref{eq: con3}, so their increments can be unequal. In particular, we may choose $\partial_2$ to be an injective integer-valued function on the vertices of $\Gamma_M$. Then $\partial_2$ will distinguish each pair of vertices of $\Gamma_M$, and so the nice sequence $(\partial_0, \partial_1,\partial_2,0,0,0\ldots )$ distinguishes vertices. Note that since $\partial_0$ and $\partial_1$ distinguish some pairs of vertices, there are non-injective choices for $\partial_2$ which still produce a nice sequence distinguishing vertices. However, we will show in Example \ref{example: nice length two} that no choice of $\partial_0$ and $\partial_1$ is enough to distinguish all vertices of $\Gamma_M$: we say that the \emph{nice length of $M$} is equal to $2$. 
\end{myeg}

\begin{rmk}\label{rmk: weave} 
Let $M$ be an $\FF_1$-representation of $Q$ and $F= (\delta_i)_{i\geq 0} $ a nice sequence for $M$ which is \emph{finite}, in the sense that there exists an $n \in \mathbb{N}$ such that $\delta_i = 0$ for all $i>n$. Let $( \partial_i)_{i \geq 0}$ be any other nice sequence for $M$. Then we can ``weave'' $(\delta_i)$ and $(\partial_i)$ together to create a new nice sequence $(\partial_i')_{i\geq 0}$ as follows:
\begin{enumerate} 
\item $\partial_i' = \delta_i$ for $0 \le i \le n$; 
\item $\partial_{i}' = \partial_{i-n-1}$ for $i > n$.
\end{enumerate} 
This is due to the following elementary observation: for \emph{any} sequence of gradings $\{ \gamma_1,\ldots, \gamma_n\}$ on $M$ and any subset $S \subseteq \{ \gamma_1,\ldots , \gamma_n\}$ (possibly empty), an $S$-nice grading (defined in the obvious way) is also $(\gamma_1,\ldots , \gamma_n)$-nice. 
\end{rmk}

In the following, we construct \emph{the universal $i$-nice gradings} of a winding $c:\Gamma \to Q$ under which one obtains all nice sequences (Theorem \ref{t: universal property}). 

\begin{construction}(Universal Nice Grading)\label{con: universal1} 
Let $c : \Gamma \rightarrow Q$ be an indecomposable winding (i.e. $\Gamma$ is connected), with $M$ the associated $\FF_1$-representation. Then we have a sequence of maps 
\begin{equation}\label{eq: construction homology}
H_1(\Gamma,\mathbb{Z}) \xrightarrow[]{H_1(c)} H_1(Q,\mathbb{Z}) \xrightarrow[]{\iota} \mathbb{Z}Q_1, 
\end{equation}
where $\iota$ is the inclusion map\footnote{The underlying graph of $Q$ is a $1$-simplex with associated chain complex $0 \rightarrow \mathbb{Z}Q_1 \xrightarrow[]{\delta} \mathbb{Z}Q_0 \rightarrow 0$, and $H_1(Q,\mathbb{Z}) = \operatorname{ker}(\delta)$. Then $\iota$ is the inclusion $\operatorname{ker}(\delta) \subseteq \mathbb{Z}Q_1$.}. Set $\mathcal{V}_M := G/t(G)$, where $G:= \mathbb{Z}Q_1/\operatorname{Im}(\iota \circ H_1(c)) = \operatorname{coker}(\iota \circ H_1(c))$ and $t(G)$ is the torsion subgroup of $G$. Since $G$ is a finitely-generated abelian group, $G/t(G)$ is a direct summand of $G$: by abuse of notation, we will use the coset notation of $G$ to denote elements of $\mathcal{V}_M$. Loosely, one can think of the elements of $\mathcal{V}_M$ as formal $\mathbb{Z}$-linear combinations of arrows of $Q$ subject to certain linear equations.

Fix a vertex $b \in \Gamma_0$ which we call the \emph{basepoint}. To each vertex $v \in \Gamma_0$ we assign an element $X(M)_v \in \mathcal{V}_M$ as follows: 
\begin{enumerate} 
\item 
If $v = b$ then $X(M)_v = 0$;
\item 
If $v \neq b$, pick a walk $p = \alpha_1^{\epsilon_1}\cdots \alpha_d^{\epsilon_d}$ from $b$ to $v$ for $\epsilon_j \in \{1,-1\}$. This is possible because $M$ is indecomposable. Then set 
\[ 
X(M)_v = \sum_{i}{\epsilon_ic(\alpha_i)} + \operatorname{Im}(\iota \circ H_1(c)). 
\]
\end{enumerate} 
Note that $X(M)_v$ does not depend on the choice of walk. The function 
\[ 
X(M) : \Gamma_0 \rightarrow \mathcal{V}_M 
\]
\[
v \mapsto X(M)_v
\] 
will be called the \emph{universal nice grading on $M$}. The image $X(M)_v$ of $v$ is the \emph{nice variable associated to $v$}. When there is no chance of confusion, we will denote $X(M)_v$ simply as $X_v$. 
\end{construction}  

\begin{construction}(Iterative Step)\label{con: universal2}
Let $c : \Gamma \rightarrow Q$ be an indecomposable winding, with $M$ the associated $\FF_1$-representation. Let $b \in \Gamma_0$ be a basepoint, and $X = X(M)$ the corresponding universal nice grading on $M$. We define a new quiver $Q^+ = Q^+(M)$ as follows. The vertices of $Q^{+}$ are 
\[ 
Q^{+}_0 = \{ (X_v,c(v)) \mid v \in \Gamma_0\}.
\]  
The arrows of $Q^{+}$ are 
\[ 
Q^{+}_1 = \{ (X_{s(\alpha)}, X_{t(\alpha)}, c(\alpha)) \mid \alpha \in \Gamma_1 \}.
\] 
To ease notation we define $\alpha^{+} :=  (X_{s(\alpha)}, X_{t(\alpha)}, c(\alpha))$. Then the source and target of $\alpha^{+}$ are as follows: 
\[ 
s(\alpha^{+}) = (X_{s(\alpha)},c(s(\alpha)))
\] 
\[ 
t(\alpha^{+}) =  (X_{t(\alpha)},c(t(\alpha))).
\] 
Note that $Q^{+}$ is connected. We have quiver maps  
\[ 
\Gamma \xrightarrow[]{c^+} Q^+ \xrightarrow[]{c^-} Q
\] 
defined as follows: $c^+$ is the unique quiver morphism which satisfies $c^+(\alpha) = \alpha^{+}$ for all $\alpha \in \Gamma_1$. Since $c^+(\alpha) = c^+(\beta)$ implies $c(\alpha) = c(\beta)$, it is clear that $c^+$ is a winding. Also note that $c^+$ is surjective on arrows, so that any arrow in $Q^{+}$ can be written as $\alpha^{+}$ for some $\alpha \in \Gamma_1$. Then $c^-$ is defined to be the unique quiver morphism satisfying $c^-(\alpha^{+}) = c(\alpha)$ for all $\alpha \in \Gamma_1$. Note that $c = c^-c^+$.

It turns out that $c^-$ is also a winding, which we can see as follows. By the surjectivity of $c^+$ on arrows, it suffices to show that whenever $\alpha, \beta \in \Gamma_1$ satisfy $c(\alpha) = c(\beta)$ and $\alpha^{+} \neq \beta^{+}$, they also satisfy $s(\alpha^{+}) \neq s(\beta^{+})$ and $t(\alpha^{+}) \neq t(\beta^{+})$. If $s(\alpha^{+}) = s(\beta^{+})$, then  
\begin{align*}
t(\alpha^{+}) & = (X_{t(\alpha)}, c(t(\alpha)) ) \\ 
& = (X_{s(\alpha)} + c(\alpha), c(t(\alpha))) \\ 
& =  (X_{s(\beta)} + c(\beta), c(t(\beta))) \\ 
& = t(\beta^{+}).
\end{align*} 
In turn, this implies $\alpha^{+} = (X_{s(\alpha)}, X_{t(\alpha)}, c(\alpha)) = (X_{s(\beta)}, X_{t(\beta)}, c(\beta)) = \beta^{+}$, contrary to hypothesis. The possibility $t(\alpha^{+}) = t(\beta^{+})$ can be ruled out in a similar fashion. 
\end{construction} 

\begin{construction}(Universal $i$-Nice Grading)\label{con: universal} 
Let $w : \Gamma \rightarrow Q$ be an indecomposable winding, with $R$ the corresponding $\mathbb{F}_1$-representation of $Q$.\footnote{The change of notation is purely cosmetic, to allow us to think of $c$ and $M$ in Constructions \ref{con: universal1} and \ref{con: universal2} as ``variables'' into which we can plug other windings/representations.} Let $b \in \Gamma_0$ be a basepoint, with $X = X(R)$ the universal nice grading of $R$ as in Construction \ref{con: universal1}. Then Construction \ref{con: universal2} yields a new indecomposable winding $w^+: \Gamma \rightarrow Q^+$, whose universal nice grading can be constructed with the same basepoint $b$. Iterating this process indefinitely yields a sequence of maps that will play a fundamental role in proving the existence (or non-existence) of nice sequences distinguishing the vertices of a given $\mathbb{F}_1$-representation. More precisely, consider the following algorithm: 
\begin{enumerate} 
\item Set $\sigma^{(0)} = w$, $M^{(0)} = R$, $Q^{(0)} = Q$ and $\tau^{(0)} = \operatorname{id}_Q$ (the identity quiver morphism of $Q$).
\item Suppose that for a fixed $i \in \mathbb{N}$, an indecomposable winding $\sigma^{(i)}: \Gamma \rightarrow Q^{(i)}$ with corresponding $\mathbb{F}_1$-representation $M^{(i)}$ has been defined. Apply Construction \ref{con: universal1} with $c = \sigma^{(i)}$ and $M = M^{(i)}$ to define  
\[ 
\mathcal{V}^{(i)}_R := \mathcal{V}_{M^{(i)}} 
\]   
\[ 
X(R)^{(i)}:=X(M^{(i)}).
\] 
\item Use Construction \ref{con: universal2} with $c = \sigma^{(i)}$ and $M = M^{(i)}$ to define 
\[ 
Q^{(i+1)} := Q^{(i)+} 
\] 
\[ 
\sigma^{(i+1)}:=\sigma^{(i)+}
\] 
\[ 
\tau^{(i+1)} :=\sigma^{(i)-}.
\]  
Furthermore, define $M^{(i+1)}$ to be the indecomposable $\mathbb{F}_1$-representation corresponding to $\sigma^{(i+1)}$.  
\item Replace $i$ with $i+1$ and go back to Step 2. 
\end{enumerate}  

When no confusion will arise, we abbreviate $X(R)^{(i)}$ to $X^{(i)}$. The image of $v \in \Gamma_0$ under $X^{(i)}$ will be denoted $X^{(i)}_v$. The function 
\[ 
X^{(i)} : \Gamma_0 \rightarrow \mathcal{V}^{(i)}_R
\] 
is called the \emph{universal $i$-nice grading of $R$}. Note that the universal $0$-nice grading of $R$ is simply the universal nice grading of $R$. To summarize, this algorithm takes the data $(w,R,b)$ and constructs each of the following (in no particular order): 
\begin{enumerate} 
\item A sequence of finitely-generated torsion free abelian groups: $\mathcal{V}^{(0)}_R, \mathcal{V}^{(1)}_R, \mathcal{V}^{(2)}_R,\ldots$. 
\item A function $X^{(i)}: \Gamma_0 \rightarrow \mathcal{V}^{(i)}_R$ for each $i \in \mathbb{N}$, called the universal $i$-nice grading of $R$. 
\item A sequence of connected quivers $Q = Q^{(0)}, Q^{(1)}, Q^{(2)}, \ldots$. We also define $Q^{(-1)} := Q$ for what follows below.
\item An indecomposable winding $\sigma^{(i)}: \Gamma \rightarrow Q^{(i)}$ for each $i \in \mathbb{N}$, with $\sigma^{(0)} = w$. 
\item A winding $\tau^{(i)}: Q^{(i)} \rightarrow Q^{(i-1)}$ for each $i \in \mathbb{N}$, satisfying $\sigma^{(i)} = \tau^{(i+1)}\sigma^{(i+1)}$ for all $i$.
\end{enumerate}
\end{construction}

\begin{rmk} 
The quiver $Q^+$ in Construction \ref{con: universal2} is adapted directly from the quiver $Q'$ defined in Proposition 6.1 of \cite{Haupt2012euler}. The key difference is that $Q'$ depends on a specific nice grading, whereas $Q^+$ only depends on the universal nice grading defined in Construction \ref{con: universal1}. This means that $Q^+$ is a general enough object to study the existence of nice gradings for the associated representation, as we shall see in Theorem \ref{t: universal property}.
\end{rmk}

\begin{rmk}\label{rmk: universal}
Use the notation of Construction \ref{con: universal}. We have now recursively defined $X^{(i)}$, $Q^{(i)}$, $\sigma^{(i)} : \Gamma \rightarrow Q^{(i)}$ and $\tau^{(i)}: Q^{(i)} \rightarrow Q^{(i-1)}$ for all $i \geq 0$, assuming that we define $Q^{(-1)}:=Q$. For all $v \in \Gamma_0$, $\alpha \in \Gamma_1$ and $i \in \mathbb{N}$ define $v^{(i)} = \sigma^{(i)}(v)$ and $\alpha^{(i)} = \sigma^{(i)}(\alpha)$. Then note that for $i \geq 1$ and $\alpha \in \Gamma_1$, $\tau^{(i)} : Q^{(i)} \rightarrow Q^{(i-1)}$ satisfies $\tau^{(i)}(\alpha^{(i)}) = \alpha^{(i-1)}$. Then we have the following elementary properties: 
\begin{enumerate} 
\item $\sigma^{(i)} = \tau^{(i+1)}\sigma^{(i+1)}$, for all $i \geq 0$. 
\item $c = \tau^{(1)}\cdots \tau^{(i)}\sigma^{(i)}$ for all $i \geq 1$.   
\item The map $H_1(\tau^{(i)}) : H_1(Q^{(i)},\mathbb{Z}) \rightarrow H_1(Q^{(i-1)},\mathbb{Z})$ induces a map $\mathcal{V}^{(i)}_M \rightarrow \mathcal{V}^{(i-1)}_M$ (still denoted $H_1(\tau^{(i)})$) that satisfies $H_1(\tau^{(i)})(X^{(i)}_v) = X^{(i-1)}_v$ for all $v \in \Gamma_0$ and $i \geq 1$. 
\item If $\alpha, \beta \in \Gamma_1$ satisfy $\alpha^{(k)} = \beta^{(k)}$ for some $k \geq 1$, then  
\[ 
c(\alpha) = c(\beta)  
\]
\[
 X^{(k-1)}_{s(\alpha)} = X^{(k-1)}_{s(\beta)}
\] 
\[
 X^{(k-1)}_{t(\alpha)} = X^{(k-1)}_{t(\beta)} 
\] 
by the definition of $Q^{(k)}$. Applying the maps $H_1(\tau^{(i)})$ for $0 \le i < k$ implies that  
\[ 
X^{(i)}_{s(\alpha)} = X^{(i)}_{s(\beta)}
\] 
\[
 X^{(i)}_{t(\alpha)} = X^{(i)}_{t(\beta)}
\] 
for all $0 \le i < k$. 
\end{enumerate} 
\end{rmk}  

\begin{myeg}\label{ex: nice1} 
Let us compute the $i$-nice variables for Example \ref{ex: string}. To begin, let us label the vertices and arrows of $\Gamma:= \Gamma_M$ as follows: 
\[ 
\Gamma = 
\begin{tikzcd} 
	v_1 \arrow[r,blue,"\beta_1"] & v_2 \arrow[r,red,"\beta_2"] & v_3 & v_4 \arrow[l,blue,"\gamma_1",swap] & v_5 \arrow[l,red,"\gamma_2",swap]
\end{tikzcd},
\]
where the coloring map $c : \Gamma \rightarrow \wild_2$ is understood to satisfy $c(\beta_i) = c(\gamma_i) = \alpha_i$ for all $i = 1, 2$.  We will choose $v_1$ as a basepoint throughout. Since $\Gamma$ is a tree, $H_1(\Gamma,\mathbb{Z}) = 0$ and $\mathcal{V}^{(0)}_M \cong \mathbb{Z}\alpha_1 \oplus \mathbb{Z}\alpha_2$. Hence, $X^{(0)}$ is the function 
\[ 
X^{(0)} : \Gamma_0 \rightarrow \mathcal{V}^{(0)}_M
\] 
\[ 
\left( \begin{array}{c} v_1 \\ v_2 \\ v_3 \\ v_4 \\ v_5 \end{array} \right) \mapsto \left( \begin{array}{c} 0 \\ \alpha_1 \\ \alpha_1+\alpha_2 \\ \alpha_2 \\ 0 \end{array} \right).
\] 
Then the arrows of $Q^{(1)}$ are described as follows: 
\[ 
\beta_1^{(1)} = (0,\alpha_1, \alpha_1) 
\] 
\[
\beta_2^{(1)} = (\alpha_1, \alpha_1+\alpha_2,\alpha_2 )
\] 
\[ 
\gamma_1^{(1)} = (\alpha_2, \alpha_1+\alpha_2, \alpha_1)
\] 
\[ 
\gamma_2^{(1)} = (0, \alpha_2, \alpha_2).
\] 
Note that $\sigma^{(1)} : \Gamma \rightarrow Q^{(1)}$ is injective on arrows, so $Q^{(1)}$ is the quiver 
\[ 
Q^{(1)} = 
\begin{tikzcd} 
& v_2^{(1)} \arrow[dr,"\beta_2^{(1)}"] & & \\  
 v_1^{(1)} \arrow[ur,"\beta_1^{(1)}"] \arrow[dr,"\gamma_2^{(1)}"] & & v_3^{(1)}  &\\ 
&  v_4^{(1)} \arrow[ur,"\gamma_1^{(1)}"]  & & \\
\end{tikzcd}
\] 
It follows that $\mathcal{V}^{(1)}_M = \mathbb{Z}\beta_1^{(1)} \oplus  \mathbb{Z}\beta_2^{(1)} \oplus  \mathbb{Z}\gamma_1^{(1)} \oplus  \mathbb{Z}\gamma_2^{(1)} \cong \mathbb{Z}^4$, and $X^{(1)}$ is the function 
\[ 
X^{(1)} : \Gamma_0 \rightarrow \mathcal{V}^{(1)}_M
\] 
\[ 
\left(\begin{array}{c} v_1 \\ v_2 \\ v_3 \\ v_4 \\ v_5 \end{array}\right) \mapsto \left(\begin{array}{c} 0 \\ \beta_1^{(1)} \\ \beta_1^{(1)}+\beta_2^{(1)} \\ \beta_1^{(1)}+\beta_2^{(1)} - \gamma_1^{(1)}\\ \beta_1^{(1)}+\beta_2^{(1)} - \gamma_1^{(1)} - \gamma_2^{(1)} \end{array}\right).
\] 
A similar description holds for all $i \geq 2$. For all such $i$ we have $\mathcal{V}^{(i)}_M =  \mathbb{Z}\beta_1^{(i)} \oplus  \mathbb{Z}\beta_2^{(i)} \oplus  \mathbb{Z}\gamma_1^{(i)} \oplus  \mathbb{Z}\gamma_2^{(i)} \cong \mathbb{Z}^4$, and $X^{(i)}$ is the function  
\[ 
X^{(i)} : \Gamma_0 \rightarrow \mathcal{V}^{(i)}_M
\] 
\[ 
\left(\begin{array}{c} v_1 \\ v_2 \\ v_3 \\ v_4 \\ v_5 \end{array}\right) \mapsto \left(\begin{array}{c} 0 \\ \beta_1^{(i)} \\ \beta_1^{(i)}+\beta_2^{(i)} \\ \beta_1^{(i)}+\beta_2^{(i)} - \gamma_1^{(i)}\\ \beta_1^{(i)}+\beta_2^{(i)} - \gamma_1^{(i)} - \gamma_2^{(i)} \end{array}\right).
\]  
It will turn out that such behavior is typical whenever $\Gamma$ is a tree.
\end{myeg} 

\begin{myeg} 
Let $Q = \wild_2$, and let $M$ be the representation with the following coefficient quiver 
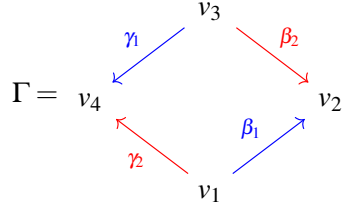
\[
\Gamma =  
\begin{tikzcd}  
 & v_3 \arrow[dl,blue,swap,"\gamma_1"]\arrow[dr,red,"\beta_2"]  & \\ 
v_4  & & v_2   \\ 
 & v_1 \arrow[ur,blue,"\beta_1"] \arrow[ul,red,"\gamma_2"] &  
\end{tikzcd}
\] 
Here, the winding $c : \Gamma \rightarrow \wild_2$ is understood to satisfy $c(\beta_i) = c(\gamma_i) = \alpha_i$ for $i=1,2$. We choose $v_1$ to be the basepoint throughout. Note that $\operatorname{Im}(\iota \circ H_1(c)) = \mathbb{Z}(2\alpha_1-2\alpha_2)$, and so $\mathcal{V}^{(0)}_M$ is the torsion-free quotient of $\frac{\mathbb{Z}\alpha_1\oplus \mathbb{Z}\alpha_2}{\langle 2(\alpha_1-\alpha_2) \rangle} \cong \mathbb{Z} \oplus \mathbb{Z}_2$. Thus $\mathcal{V}^{(0)}_M \cong \mathbb{Z}$, and we will denote the generator corresponding to the coset of $\alpha_1$ by $\alpha$. The $0$-nice variables are then described via the function 
\[ 
X^{(0)} : \Gamma_0 \rightarrow \mathcal{V}^{(0)}_M
\] 
\[ 
\left( \begin{array}{c} v_1 \\ v_2 \\ v_3 \\ v_4 \end{array} \right)\mapsto \left( \begin{array}{c} 0 \\ \alpha \\ 0 \\ \alpha \end{array} \right).
\] 
It follows that
\[ 
Q^{(1)} = 
\begin{tikzcd}  
(0,\bullet) \arrow[r, bend left, "\beta_1^{(1)}=\gamma_1^{(1)}"] \arrow[r,bend right,swap, "\beta_2^{(1)}=\gamma_2^{(1)}"] & (\alpha,\bullet)
\end{tikzcd},
\] 
where $\bullet$ denotes the unique vertex of $\wild_2$. The winding map $\sigma^{(1)}$ is given by
\[ 
 \left( \begin{array}{c} v_1 \\ v_2 \\ v_3 \\ v_4 \end{array} \right)\mapsto \left( \begin{array}{c} (0,\bullet) \\ (\alpha,\bullet) \\ (0,\bullet) \\ (\alpha,\bullet) \end{array} \right)
\]  
\[ 
\left( \begin{array}{c} \beta_1 \\ \beta_2 \\ \gamma_1 \\ \gamma_2 \end{array} \right)\mapsto \left( \begin{array}{c} \beta_1^{(1)} \\ \beta_2^{(1)} \\ \beta_1^{(1)} \\ \beta_2^{(1)} \end{array} \right).
\] 
Hence, $\mathcal{V}^{(1)}_M$ is the torsion-free quotient of $\frac{\mathbb{Z}\beta_1^{(1)}\oplus \mathbb{Z}\beta_2^{(1)}}{\langle 2(\beta_1^{(1)} - \beta_2^{(1)}) \rangle} \cong \mathbb{Z} \oplus \mathbb{Z}_2$. Denoting its generator by $\beta$, we find that the $1$-nice variables are given by  
\[ 
X^{(1)} : \Gamma_0 \rightarrow \mathcal{V}^{(1)}_M
\] 
\[ 
\left( \begin{array}{c} v_1 \\ v_2 \\ v_3 \\ v_4 \end{array} \right)\mapsto \left( \begin{array}{c} 0 \\ \beta \\ 0 \\ \beta \end{array} \right).
\]  
We see that for this representation, no new information is obtained from iteration. This reflects the fact that for any nice grading $\partial$ of $M$, a $\partial$-nice grading is the same as a nice grading. Note that $v_1$ and $v_3$ cannot be distinguished by any nice sequence, nor can $v_2$ and $v_4$. This will turn out to be typical behavior for $\Gamma$ a ``non-primitive quiver of type $\tilde{\mathbb{A}}_n$'' (see Definition \ref{definition: primitive}).
\end{myeg}

\begin{mydef} 
Let $M$ and $N$ be abelian groups. A function $f : M \rightarrow N$ is said to be \emph{affine} if there exists a $z \in N$ such that $x \mapsto f(x)-z$ is a group homomorphism. If $M = N$ and $f(x) = x+z$ we say that $f$ is a \emph{translation}. The translations of $M$ form a group under composition isomorphic to $M$.
\end{mydef} 

\begin{rmk}\label{r: basepoint}
Assuming that $c : \Gamma \rightarrow Q$ is indecomposable, the variables $( X^{(i)})_{i \geq 0}$ are unique up to translation in $\mathcal{V}_M^{(i)}$. Hence, the condition $X^{(i)}_u = X^{(i)}_v$ does not depend on the choice of basepoint.
\end{rmk}

The following theorem explains the name ``universal nice grading'' as one obtains any nice sequence as an evaluation of universal nice gradings. In other words, any nice sequence $\underline{\partial}$ uniquely factors through universal nice gradings as follows:
\begin{equation}
	\begin{tikzcd}[row sep=2em]
		\Gamma_0 \arrow[rr,"\partial_i"] \arrow[dr,swap,"X^{(i)}"]
		&& \mathbb{Z}  \\
		& \mathcal{V}^{(i)}_M \arrow[ur,swap,"\operatorname{ev}^{(i)}(\underline{\partial})"]
	\end{tikzcd}
\end{equation}

To be precise, we prove the following. 

\begin{mythm}\label{t: universal property}
Let $M$ be an indecomposable $\FF_1$-representation of $Q$ and $c : \Gamma \rightarrow Q$ be the associated winding with basepoint $b$. Let $\underline{\partial} = (\partial_i)_{i=0}^{\infty}$ be a nice sequence for $M$. Then for each $i$, there exists a unique affine map 
\[ 
\operatorname{ev}^{(i)}(\underline{\partial}) : \mathcal{V}^{(i)}_M \rightarrow \mathbb{Z}
\] 
such that $\partial_i = \operatorname{ev}^{(i)}(\underline{\partial}) \circ X^{(i)}$. We write $X^{(i)}(\underline{\partial}) :=  \operatorname{ev}^{(i)}(\underline{\partial}) \circ X^{(i)}$ and call it the evaluation of $X^{(i)}$ at $\underline{\partial}$.
\end{mythm}  
\begin{proof} 
We first prove the claim for $i=0$. Define a map $g_0 : \mathcal{V}^{(0)}_M \rightarrow \mathbb{Z}$ via the formula 
\begin{equation}\label{e: base case}
\overline{c(\alpha)} \mapsto \partial_0(t(\alpha)) - \partial_0(s(\alpha)),
\end{equation}
where $\alpha \in \Gamma_1$. We must show that this map is well-defined. First note that $\partial_0 : \Gamma_0 \rightarrow \mathbb{Z}$ extends uniquely to a group map $\mathbb{Z}\Gamma_0 \rightarrow \mathbb{Z}$ (also denoted $\partial_0$). Since $\partial_0$ is a nice grading, we have a well-defined map $\hat{g}_0 : \mathbb{Z}Q_1 \rightarrow \mathbb{Z}$ defined via the following formulas: 
\begin{enumerate} 
\item $\hat{g}_0(c(\alpha)) = \partial_0(t(\alpha)) - \partial_0(s(\alpha))$, for all $c(\alpha) \in c(\Gamma_1)$;
\item $\hat{g}_0(\beta) = 0$, for all $\beta \in Q_1\setminus c(\Gamma_1)$. 
\end{enumerate}
If $Z :=\sum{\lambda_{\alpha}\alpha}$ is an element of $H_1(\Gamma,\mathbb{Z})$ then  
\[ 
0 = \sum{\lambda_{\alpha}[t(\alpha)-s(\alpha)]}
\] 
and hence the element $\iota^{(0)}H_1(c)(Z) = \sum{\lambda_{\alpha}c(\alpha)}$ satisfies
\begin{align*}
 \hat{g}_0\left(\sum{\lambda_{\alpha}c(\alpha)}\right) & = \sum{\lambda_{\alpha}\hat{g}_0(c(\alpha))}\\ 
& =  \sum{\lambda_{\alpha}[\partial_0(t(\alpha)) - \partial_0(s(\alpha))]} \\ 
& = \partial_0\left(  \sum{\lambda_{\alpha}[t(\alpha)-s(\alpha)]} \right) \\ 
& = 0.
\end{align*} 
In other words, $\hat{g}_0$ descends to a map on $\mathcal{V}^{(0)}_M$ and $g_0$ is well-defined\footnote{Note that any homomorphism $f : M \rightarrow N$ between abelian groups induces a homomorphism $M/t(M) \rightarrow N/t(N)$.}. We now set 
\[
\operatorname{ev}^{(0)}(\underline{\partial}) := g_0 + \partial_0(b). 
\]  
Let $v \in \Gamma_0$ with $\alpha_1^{\epsilon_1}\cdots \alpha_d^{\epsilon_d}$ a walk in $\Gamma$ from $b$ to $v$. Then 
\begin{align*}
\partial_0(v) & = \partial_0(b) + \sum_j{\epsilon_j[\partial_0(t(\alpha_j)) - \partial_0(s(\alpha_j))]} \\ 
& = \partial_0(b) + g_0\left(\sum_j{\epsilon_jc(\alpha_j)} + \operatorname{Im}(\iota^{(0)}\circ H_1(c)) \right) \\ 
& = \partial_0(b) + g_0\left(X^{(0)}_v\right) \\ 
& = [\operatorname{ev}^{(0)}(\underline{\partial})\circ X^{(0)}](v), 
\end{align*} 
so we have the desired factorization. To prove uniqueness, suppose $a : \mathcal{V}^{(0)}_M \rightarrow \mathbb{Z}$ is an affine map such that $\partial_0 = a\circ X^{(0)}$. Then in particular,
\begin{align*} 
 \operatorname{ev}^{(0)}(\underline{\partial})(0) & = \partial_0(b) \\
& = [a\circ X^{(0)}](b) \\ 
 & = a\left(X^{(0)}_b \right) \\ 
& = a(0). 
\end{align*}  
In other words, $g_0 = \operatorname{ev}^{(0)}(\underline{\partial}) - \partial_0(b)$ and $f := a-\partial_0(b)$ are both group homomorphisms. Pick $\alpha \in \Gamma_1$ and let $ p =\alpha_1^{\epsilon_1}\cdots \alpha_d^{\epsilon_d}$ denote a walk from $b$ to $s(\alpha)$ in $\Gamma$. Then $p\alpha$ is a walk from $b$ to $t(\alpha)$ in $\Gamma$ and
\begin{align*} 
g_0(\overline{c(\alpha)}) & = \partial_0(t(\alpha)) - \partial_0(s(\alpha)) \\  
& = \left[g_0\left(X^{(0)}_{t(\alpha)}\right) + \partial_0(b)\right] - \left[g_0\left(X^{(0)}_{s(\alpha)} \right) + \partial_0(b) \right] \\ 
& = a\left(X^{(0)}_{t(\alpha)}\right) - a\left(X^{(0)}_{s(\alpha)} \right)  \\ 
& = \left[ f\left(X^{(0)}_{t(\alpha)}\right) + \partial_0(b) \right] - \left[ f\left( X^{(0)}_{s(\alpha)} \right) + \partial_0(b) \right] \\ 
& = f\left( X^{(0)}_{t(\alpha)} - X^{(0)}_{s(\alpha)} \right) \\ 
& = f(\overline{c(\alpha)}). 
\end{align*} 
It follows that $a = \operatorname{ev}^{(0)}(\underline{\partial})$. Thus, the claim holds for $i = 0$. Now suppose that $i > 0$, and that the result has been proved for all $j < i$. Since $\partial_0$ is a nice grading for $\sigma^{(0)} = c$, we may also assume by induction that $\partial_j$ is a nice grading for $\sigma^{(j)}$ for all $0 \le j < i$. Under these assumptions, we claim that $\partial_i$ induces a nice grading on $\sigma^{(i)}$. Suppose that $\alpha, \beta \in \Gamma_1$ are given such that $\alpha^{(i)} = \sigma^{(i)}(\alpha) = \sigma^{(i)}(\beta) = \beta^{(i)}$. Then by definition of $Q^{(i)}$, we must have
\[ 
c(\alpha) = c(\beta)
\] 
\[ 
X^{(i-1)}_{s(\alpha)} = X^{(i-1)}_{s(\beta)} 
\] 
\[ 
X^{(i-1)}_{t(\alpha)} = X^{(i-1)}_{t(\beta)}.
\] 
By Remark \ref{rmk: universal}, this also implies 
\[ 
X^{(j)}_{s(\alpha)} = X^{(j)}_{s(\beta)}
\] 
\[ 
X^{(j)}_{t(\alpha)} = X^{(j)}_{t(\beta)}
\] 
for all $0 \le j < i$. By induction, for all $0 \le j < i$ there exists a unique affine map $\operatorname{ev}^{(j)}(\underline{\partial}) : \mathcal{V}^{(j)}_M \rightarrow \mathbb{Z}$ satisfying $\partial_j = \operatorname{ev}^{(j)}(\underline{\partial}) \circ X^{(j)}$. In particular, we must have 
\begin{align*} 
\partial_j(s(\alpha)) & = \operatorname{ev}^{(j)}(\underline{\partial})(X^{(j)}_{s(\alpha)})\\ 
& = \operatorname{ev}^{(j)}(\underline{\partial})(X^{(j)}_{s(\beta)}) \\ 
& = \partial_j(s(\beta))
\end{align*} 
for all $0 \le j < i$. Similarly, $\partial_j(t(\alpha)) = \partial_j(t(\beta))$ for all $0 \le j < i$ as well. Since $c(\alpha) = c(\beta)$ as well, the assumption that $\partial_i$ is a $(\partial_0,\ldots , \partial_{i-1})$-nice grading implies that 
\[ 
\partial_i(t(\alpha)) - \partial_i(s(\alpha)) = \partial_i(t(\beta)) - \partial_i (s(\beta)),
\]  
from which it follows that $\partial_i$ is a nice grading for $\sigma^{(i)}$. By the base case, there exists a unique affine map $\operatorname{ev}^{(i)}(\underline{\partial}) : \mathcal{V}^{(0)}_{M^(i)} \rightarrow \mathbb{Z}$ satisfying $\partial_i = \operatorname{ev}^{(i)}(\underline{\partial}) \circ X(M^{(i)})$. Since $\mathcal{V}^{(i)}_M = \mathcal{V}^{(0)}_{M^{(i)}}$ and $X^{(i)} = X(M^{(i)})$ by Construction \ref{con: universal}, the result now follows from induction. 
\end{proof}

The proof above readily implies the following:

\begin{cor}  
Let $M$ be an indecomposable $\mathbb{F}_1$-representation of $Q$ and $c : \Gamma \rightarrow Q$ be the associated winding with basepoint $b$. Let $(\partial_i)_{i\geq 0}$ be a nice sequence for $M$. Then for each $i$, $\partial_i$ is a nice grading for $\sigma^{(i)}: \Gamma \rightarrow Q^{(i)}$. 
\end{cor}

\begin{rmk}\label{r: arrows}
We can view a nice grading as a certain integer-valued function on the arrows of $Q$. Indeed, a nice grading on $M$ is the same thing as a nice sequence $\underline{\partial} = ( \partial_i)_{i \geq 0}$ with $\partial_i = 0$ for all $i > 0$. By Theorem \ref{t: universal property}, a nice grading on $M$ is then an affine map $f : \mathcal{V}^{(0)}_M \rightarrow \mathbb{Z}$. Writing $f  = g+z$ with $z$ an integer and $g : \mathcal{V}^{(0)}_M \rightarrow \mathbb{Z}$ a group homomorphism, the map $g$ can be identified with a group homomorphism $\mathbb{Z}Q_1 \rightarrow \mathbb{Z}$ that vanishes on $\operatorname{Im}(\iota \circ H_1(c))$. If $b_1,\ldots, b_k$ is a cycle basis for $H_1(\Gamma,\mathbb{Z})$, then a map $\mathbb{Z}Q_1 \rightarrow \mathbb{Z}$ vanishes on $\operatorname{Im}(\iota \circ H_1(c))$ if and only if it vanishes on $[\iota\circ H_1(c)](b_i)$ for each $i$. Finally, any such homomorphism comes from a unique function $Q_1 \rightarrow \mathbb{Z}$.
\end{rmk}

\begin{cor}\label{c: torsion}
Let $c : \Gamma \rightarrow Q$ be an indecomposable winding with associated representation $M$. Let $X^{(i)}$ denote the universal $i$-nice grading of $M$. Then for any two $w, z \in \Gamma_0$, there exists a nice sequence distinguishing $w$ and $z$ if and only if $X_w^{(i)} \neq X_z^{(i)}$ for some $i$.
\end{cor} 
\begin{proof}
 By Theorem \ref{t: universal property}, a nice sequence $\underline{\partial}$ is equivalent to a collection of affine maps $E_i :=\operatorname{ev}^{(i)}(\underline{\partial}): \mathcal{V}_M^{(i)} \rightarrow \mathbb{Z}$. Write $E_i = h_i + x_i$, where $h_i : \mathcal{V}_M^{(i)} \rightarrow \mathbb{Z}$ is a group homomorphism and $x_i \in \mathbb{Z}$. The $i^{th}$ map will distinguish $w$ and $z$ if and only if $E_i(X_w^{(i)}) \neq E_i(X_z^{(i)})$, hence $0 \neq E_i(X_w^{(i)}) - E_i(X_z^{(i)}) = h_i(X_w^{(i)}) - h_i(X_z^{(i)}) = h_i(X_w^{(i)} - X_z^{(i)})$. This implies that $X_w^{(i)} \neq X_z^{(i)}$. Conversely, assume that  $X_w^{(i)} \neq X_z^{(i)}$. Since $\mathcal{V}^{(i)}_M$ is a finitely-generated torsion-free abelian group, it is free abelian and hence has a basis $\{ b_1,\ldots, b_m\}$. If $b_i^* : \mathcal{V}^{(i)}_M \rightarrow \mathbb{Z}$ is the homomorphism defined by $b_i^*(b_j) = \delta_{ij}$, then $X_w^{(i)} \neq X_z^{(i)}$ implies $b_j^*(X_w^{(i)}) \neq b_j^*(X_z^{(i)})$ for some $j$. Setting $E_i = b_j$ and $E_{k} = 0$ for all $k \neq 0$, $\{ E_k \}_{k \geq 0}$ is a collection of affine maps which induce a nice sequence distinguishing $w$ and $z$. 
\end{proof}


\section{Euler characteristics of quiver Grassmannians} \label{s: Euler}

\subsection{Nice representations} 

Our reason for considering nice sequences for $\FF_1$-representations is that, under certain circumstances, they allow us to combinatorially compute the Euler characteristics of associated quiver Grassmannians. In this section, we introduce the notion of a \emph{nice} $\FF_1$-representation of $Q$. If $M$ is a nice $\FF_1$-representation, then the Euler characteristic of $\operatorname{Gr}_{\textbf{e}}^Q(M_\mathbb{C})$, the quiver Grassmannian associated to the $\mathbb{C}$-representation $M_\mathbb{C}$ of $Q$\footnote{Recall that $M_\mathbb{C}$ is the $\mathbb{C}$-representation of $Q$ obtained from $M$ via ``base change'' as in \eqref{eq: base change1} and \eqref{eq: base change2}.}, can be computed as the number of $\textbf{e}$-dimensional subrepresentations of $M$. This is a combinatorial task, since $M$ is a finite set. We prove several sufficient conditions for $M$ to be nice: more precisely, we prove sufficient conditions for the existence of a nice sequence on $M$ distinguishing vertices. We formalize this idea with the \emph{nice length} of a representation, which we use to construct new Hopf algebras from $H_Q$ and $H_{Q,\nil}$ in Section \ref{s: nice length}. We conclude this subsection with several examples of nice representations, including a family of nice representations for $Q = \wild_2$ whose cycle spaces can have arbitrarily high rank. Previously in the literature, only tree modules, band modules and representations of nice length $0$ had been explicitly considered. The above discussion motivates the following definition.

\begin{mydef} 
Let $M$ be an $\FF_1$-representation of $Q$ with associated winding 
\[
c_M : \Gamma_M \rightarrow Q. 
\]
 We say that $M$ is \emph{nice} if the following equation holds, for all dimension vectors $\textbf{e} \leq \textbf{\textrm{dim}}(M)$: 
\begin{equation}\label{e: nice} 
\chi_{\textbf{e}}(M_\mathbb{C}) = |\{ N \le M \mid \textbf{\textrm{dim}}(N) = \textbf{e} \}|,
\end{equation}
where $\chi_{\textbf{e}}(M_\mathbb{C})$ is the Euler characteristic of $\operatorname{Gr}_{\textbf{e}}^Q(M_\mathbb{C})$.
\end{mydef}  

In what follows, we will simply denote $\operatorname{Gr}_{\textbf{e}}^Q(M_\mathbb{C})$ by $\operatorname{Gr}_{\textbf{e}}(M_\mathbb{C})$ whenever there is no possible confusion. 

In Lemma 1 of \cite{irelli2011quiver}, Cerulli-Irelli shows that a nice grading on a string module $M$ induces an algebraic action of $\mathbb{C}^{\times}$ on $\operatorname{Gr}_{{\bf{e}}}(M)$ for each dimension vector ${\bf{e}}$. Writing $M = \bigoplus_{i\in Q_0}{M_i}$, this action has finitely-many fixed points when the grading distinguishes the basis elements of $M_i$ for each $i$, in which case the number of fixed points equals $\chi_{{\bf{e}}}(M)$. In Theorem 1.1 of \cite{Haupt2012euler}, Haupt generalizes this result to arbitrary representations and gradings. As a special case, Lemma 4.11 and Corollary 5.2 of \cite{Haupt2012euler} show how $(\partial_0,\ldots , \partial_i$)-nice gradings distinguishing basis elements can be used to compute Euler characteristics of quiver Grassmannians by counting fixed points of the induced torus actions. Haupt then applies these results to tree and band modules. The proposition below is essentially a restatement of the result on Euler characteristics from \cite{Haupt2012euler}. We include a proof for the convenience of the reader. 

\begin{pro}[\cite{Haupt2012euler}]  \label{rmk: distinguish vertices}
Let $M$ be an $\FF_1$-representation of $Q$, and $\underline{\partial}$ a nice sequence for $M$ which distinguishes vertices. Then $M$ is a nice representation.
\end{pro} 

\begin{proof}  
Note that $\Gamma_M$ can be considered as a coefficient quiver of $M_\mathbb{C}$ in the sense of Ringel \cite{ringel1998exceptional} with a choice of a basis as follows\footnote{See Section \ref{subsection: coefficient quivers} for the definition of coefficient quivers by Ringel.}:
\begin{equation}
	\mathcal{B}=\bigsqcup_{v \in Q_0}B_v, \qquad B(v)=M_v \backslash \{0\}.
\end{equation}
Since $\underline{\partial}$ distinguishes vertices, there exists an $n \in \mathbb{N}$ such that for all $x, y \in (\Gamma_M)_0$, $\partial_i(x) \neq \partial_i(y)$ for some $i\le n$. In particular, the elements of $M_\mathbb{C}$ which are $\partial_i$-homogeneous (as in \cite[pp 756]{Haupt2012euler}) for all $i\le n$ are precisely the scalar multiples of elements in the set $\mathcal{B}$.

Now, it follows from \cite[Theorem 1.1]{Haupt2012euler} that for all $\textbf{e} \le \textbf{\textrm{dim}}(M)$, $\operatorname{Gr}_{\textbf{e}}(M_\mathbb{C})$ has the same Euler characteristic as the locally-closed subset 
\begin{equation}\label{eq: homogenous}
X = \{N\in \operatorname{Gr}_{\textbf{e}}(M_\mathbb{C})\mid N\text{ has a $\partial_i$-homogeneous basis for all $i\le n$}.  \}
\end{equation}
From the aforementioned correspondence between elements of $M_\mathbb{C}$ which are $\partial_i$-homogeneous and the scalar multiples of elements in $\mathcal{B}$, any $N$ as in \eqref{eq: homogenous} must have a basis $\mathcal{B}_N\subseteq \mathcal{B}$. It is then easy to check that the span of $\mathcal{B}_N$ defines a $\mathbb{C}$-representation if and only if $\mathcal{B}_N$ is an $\FF_1$-representation of $Q$\footnote{This essentially follows from the fact that $f_{\alpha}(\mathcal{B}_N) \subseteq \mathcal{B}\cup\{ 0\}$ for all $\alpha \in Q_1$.}. Hence, $X$ is a finite set with $|\{ N \le M \mid \textbf{\textrm{dim}}(N) = \textbf{e} \}|$ elements, and so $M$ is nice.
\end{proof}

For an $\FF_1$-representation $M$ of $Q$ and its subrepresentation $N$, one can define the quotient $M/N$ by using the quotient of $\FF_1$-vector spaces in Definition \ref{definition: $F_1$-vectorspace}. By a \emph{subquotient}, we mean a quotient of a subrepresentation of $M$. The following is straightforward.  

\begin{pro}\label{proposition: sub quotient nice}
Let $M$ be an $\FF_1$-representation of $Q$ with associated winding $c_M : \Gamma_M \rightarrow Q$. Let $M'/N'$ denote a subquotient of $M$ with associated winding $c_{M'/N'}$. If $c_M$ admits a non-degenerate, positive or negative $(\partial_0,\ldots , \partial_n)$-nice grading, then so does $c_{M'/N'}$.
\end{pro} 
\begin{proof} 
It follows from \cite[Lemma 3.12]{jun2020quiver} that the coefficient quiver of $M'/N'$ is a subquiver of $\Gamma_M$. Hence, any non-degenerate~/~positive~/~negative $(\partial_0,\ldots , \partial_n)$-nice grading on $M$ will restrict to one on $M'/N'$.
\end{proof}  

\begin{mydef}\label{definition: length}
	Let $M$ be an $\FF_1$-representation of $Q$ with coefficient quiver $\Gamma := \Gamma_M$ and winding map $c :=c_M$. Recall from Definition \ref{def: distinguish} that a nice sequence $(\partial_i)$ for $M$ distinguishes vertices if for each distinct $x,y \in \Gamma_0$, there exists an $i \in \mathbb{N}$ for which $\partial_i(x) \neq \partial_i(y)$. We say that $M$ has \emph{finite nice length} if there exists a nice sequence $(\partial_i)$ for $M$ which distinguishes vertices in finitely-many steps, in the sense that there exists an $N \in \mathbb{N}$ such that for all distinct $x,y \in \Gamma_0$, $\partial_i(x) \neq \partial_i(y)$ for some $i \le N$. If $M$ has finite nice length, the \emph{nice length} of $M$ is the smallest nonnegative integer $n$ such that there exists a nice sequence $\underline{\partial} = (\partial_i)_{i=0}^{\infty}$ for which the truncated sequence $(\partial_0,\ldots , \partial_n)$ distinguishes vertices. We write ${\operatorname{nice}}(M) = n$ in this case, and $\operatorname{nice}(M) = \infty$ if $M$ does not have finite nice length.
\end{mydef} 

The following properties are clear from the definition (and Proposition \ref{rmk: distinguish vertices}): 
\begin{enumerate} 
	\item If $\operatorname{nice}(M) < \infty$ then $M$ is nice.
	\item If $M$ is an $\FF_1$-representation of $Q$, then $\operatorname{nice}(M'/N) \le \operatorname{nice}(M)$ for any subquotient $M'/N$ of $M$. In particular, if $\operatorname{nice}(M) < \infty$ then $\operatorname{nice}(M'/N) < \infty$.
\end{enumerate}

\begin{myeg}\label{ex: irelli} 
Let $M$ be an $\mathbb{F}_1$-representation of $Q$, with $\Gamma := \Gamma_M$ and $c:=c_M$. In Theorem 1 of \cite{irelli2011quiver}, Cerulli Irelli considers a map $d : \Gamma_0 \rightarrow \mathbb{Z}$ which satisfies the two conditions: 
\begin{enumerate} 
\item [(D1)] For all $v \in Q_0$, the function $d$ restricts to an injection on $c^{-1}(v)$. 
\item [(D2)] The function $d$ is a nice grading on $M$.
\end{enumerate} 
If such a $d$ exists, then certainly $\operatorname{nice}(M) \le 1$, say by Proposition 6.1 of \cite{Haupt2012euler}. In fact, we claim $\operatorname{nice}(M) = 0$. To do this, we must find an injective nice grading of $M$. First note that if $f : Q_0 \rightarrow \mathbb{Z}$ is any function, then $f$ induces a nice grading $f^* : \Gamma_0 \rightarrow \mathbb{Z}$ via the formula $f^*(v) = f(c(v))$. It is straightforward to check that $d + f^*$ still satisfies (D1)-(D2). Pick an ordering $v_0, \ldots , v_m$ of $Q_0$, and for each $i = 0, \ldots , m$, set $B_i := c^{-1}(v_i)$. Define a function $f : Q_0 \rightarrow \mathbb{Z}$ as follows: 
\[ 
f(v_i) = 2^i\cdot \max_{u,v \in \Gamma_0}\left({|d(u) - d(v)|} +1 \right), \text{ $i = 0, \ldots , m$.}
\] 
Setting $\partial := d + f^*$ yields a nice grading of $M$ that distinguishes any two vertices of $B_i$, for $i = 0,\ldots, m$. If $u \in B_i$ and $v \in B_j$ with $i < j$, a straightforward computation reveals that $\partial(v) - \partial(u) \geq 1$. Hence $\partial$ distinguishes all vertices of $\Gamma$ and $\operatorname{nice}(M) = 0$, as claimed. Conversely, any representation $M$ with nice length $0$ admits a grading that satisfies (D1)-(D2). Hence, we may identify such $\FF_1$-representations as those with nice length $0$.  
\end{myeg}

\begin{myeg} \label{example: nice length two}
 In light of Corollary \ref{c: torsion}, the computation in Example \ref{ex: nice1} shows that the string module described in Example \ref{ex: string} has nice length $1$. Similarly, one can show that the string module described in Example \ref{ex: string2} has nice length $2$. Indeed, consider the vertices $a,b,\dots,i$ defined as follows:
\begin{equation}
\begin{tikzcd}  
 & a \arrow[d,blue,"\alpha_1",swap] & & & & \\
	& b \arrow[r,red,"\alpha_2"] & \bullet \arrow[r,green,"\alpha_3"] & \bullet & \bullet \arrow[l,red,"\alpha_2",swap] & c \arrow[l,green,"\alpha_3",swap]  \\ 
  & \arrow[d,blue,"\alpha_1",swap] e  \arrow[r,red,"\alpha_2"] & \bullet \arrow[r,green,"\alpha_3"] & \bullet  & \arrow[l,red,"\alpha_2",swap] x & \arrow[u,blue,"\alpha_1",swap] \arrow[l,green,"\alpha_3",swap] d  \\ 
 & f \arrow[r,green,"\alpha_3"]   &  \bullet \arrow[r,red,"\alpha_2"] &  \bullet  & \arrow[l,green,"\alpha_3",swap] \bullet & \arrow[l,red,"\alpha_2",swap] g \\ 
& i \arrow[r,green,"\alpha_3"] & \bullet \arrow[r,red,"\alpha_2"] & \bullet & \arrow[l,green,"\alpha_3",swap] \bullet  & \arrow[l,red,"\alpha_2",swap] \arrow[u,blue,"\alpha_1",swap] h \\
\end{tikzcd}
\end{equation}  
Fixing $a$ as a basepoint, we have  
\[
X^{(0)}_a = X^{(0)}_d = X^{(0)}_e = X^{(0)}_h = X^{(0)}_i = 0 
\] 
\[ 
 X^{(0)}_b = X^{(0)}_c = X^{(0)}_f = X^{(0)}_g = \alpha_1.  
\]
The quiver $Q^{(1)}$ is given as follows: 
\begin{equation}  
Q^{(1)} =
\begin{tikzcd} 
  & \alpha_1+\alpha_2+\alpha_3 & \\ 
\alpha_1+\alpha_2 \arrow[ur] & T\curvearrowright & \alpha_1+\alpha_3 \arrow[ul] \\ 
& \alpha_1 \arrow[ur] \arrow [ul] & \\ 
& 0 \arrow[u, "\lambda",swap] \arrow[dr] \arrow[dl] & \\ 
\alpha_2 \arrow[dr] & D \curvearrowright & \alpha_3 \arrow[dl] \\ 
& \alpha_2 + \alpha_3 & \\
\end{tikzcd}
\end{equation} 
Here, each vertex is labeled by its $0$-nice variable, $T$ and $D$ are considered as elements of $H_1(Q^{(1)},\mathbb{Z})$ with the clockwise orientation, and $\lambda$ is an arrow of $Q^{(1)}$. We now have $X^{(1)}_a = X^{(1)}_i = 0$, $X^{(1)}_b = \lambda$, $X^{(1)}_c = \lambda +T$, $X^{(1)}_d = T$, $X^{(1)}_e = T+D$, $X^{(1)}_f = \lambda + T+D$, $X^{(1)}_g = \lambda+D$, $X^{(1)}_h = D$. Since $X^{(1)}_a = X^{(1)}_i$, it follows that the nice length of this representation is at least $2$. Since the nice sequence $(\partial_0, \partial_1,\partial_2, 0,0,\ldots )$ described in Example \ref{ex: string2} distinguishes vertices, the nice length of this representation is exactly $2$.
\end{myeg}


Let $c:\Gamma \to Q$ be a winding and $\alpha \in Q_1$. In the following, it will be convenient to endow $c^{-1}(\alpha)$ with the structure of a subquiver. We let $c^{-1}(\alpha)$ mean the arrow-induced subquiver of $\Gamma$ whose set of arrows is $\{\beta \in \Gamma_1 \mid c(\beta)=\alpha\}$ and whose vertex set consists of the sources and targets of those $\beta$.

\begin{pro}\label{p: pos gradings}
Let $M$ be an $\mathbb{F}_1$-representation of $Q$ and let $c_M : \Gamma_M \rightarrow Q$ denote the associated winding. Let $\partial$ be a nice grading on $M$. Then the following hold:  
\begin{enumerate} 
\item 
If $\partial$ is positive or negative, then $M$ is nilpotent. 
\item 
Suppose that $\partial$ is positive or negative, and that $Q$ has no loops. If for each $\alpha \in Q_1$, $\partial$ restricts to an injection on the set $\{s(\beta) \mid \beta \in (\Gamma_M)_1, c_M(\beta) = \alpha \}$, then $\operatorname{nice}(M) \le 1$ and $M$ is nice.
\item 
If $\partial$ is non-degenerate and $c_M^{-1}(\alpha)$ is connected for all $\alpha \in Q_1$, then $\operatorname{nice}(M)\le 1$ and $M$ is nice.  
\end{enumerate}
\end{pro} 
\begin{proof} 
For (1) and (2), we only prove the case that $\partial$ is positive, as the negative case is similar. 

(1) Suppose that $\beta_1\cdots \beta_d$ is an oriented cycle in $\Gamma_M$ with $v:=s(\beta_1) = t(\beta_d)$. Then 
\[\partial(v) = \sum_{i=1}^d{\Delta_{c_M(\beta_i)}^{\partial}}+ \partial(v),\] 
which implies $\sum_{i=1}^d{\Delta_{c_M(\beta_i)}^{\partial}} = 0$, contradicting the positivity of $\partial$. It follows that $\Gamma_M$ is acyclic, so $M$ is nilpotent. 

(2)  For each $\alpha \in Q_1$, let $S_{\alpha}$ denote the subquiver of $\Gamma_M$ whose vertex set is $(\Gamma_M)_0$ and whose arrow set is $A_\alpha:=\{\beta \in (\Gamma_M)_1 \mid c(\beta)=\alpha\}$. Since $Q$ has no loops, for each $\alpha \in Q_1$ the connected components of $S_{\alpha}$ are isolated vertices or arrows. Combining this with the assumption that $\partial$ restricts to an injection on $\{ s(\beta) \mid \beta \in A_\alpha\}$, it follows that for any two arrows $\beta, \gamma \in A_\alpha$ we have $\partial(s(\beta)) \neq \partial(s(\gamma))$. Then Conditions \eqref{eq: con1} and \eqref{eq: con3} can never hold simultaneously, and so any grading is a $\partial$-nice grading. In particular, any injective function $\partial_1 : (\Gamma_M)_0 \rightarrow \mathbb{Z}$ is a $\partial$-nice grading. Then $(\partial, \partial_1,0,0,0\ldots )$ is a nice sequence for $M$ distinguishing vertices and hence $\operatorname{nice}(M) \le 1$. 

(3) Since $c_M^{-1}(\alpha)$ is connected, it must be an oriented path of length $\geq 0$. It follows from the non-degeneracy of $\partial$ that for any two vertices $v, w \in c_M^{-1}(\alpha)$, $\partial(v) \neq \partial (w)$. Then there is an injective $\partial$-nice grading on $M$, and hence $\operatorname{nice}(M) \le 1$. 
\end{proof}

Let $c : \Gamma \rightarrow Q$ and $c' : \Gamma' \rightarrow Q$ be two winding maps. Suppose that $v \in \Gamma_0$ and $v' \in \Gamma_0'$ satisfy $c(v) = c'(v')$. Then one can define the amalgam $\Gamma\sqcup_{v\sim v'}\Gamma'$ and associated quiver map $c\sqcup_{v\sim v'}c' : \Gamma\sqcup_{v\sim v'}\Gamma' \rightarrow Q$\footnote{For more details, see \cite[Definition 3.4]{jun2020quiver}.}. Note that if $c(\Gamma_1) \cap c(\Gamma_1') = \emptyset$ then $c\sqcup_{v\sim v'}c'$ is again a winding. The next proposition shows how nice sequences behave with respect to such amalgams.

\begin{pro}\label{p: gluing} 
Let $M$ and $N$ be $\mathbb{F}_1$-representations of $Q$ with $c_M : \Gamma_M \rightarrow Q$ and $c_N: \Gamma_N \rightarrow Q$ their associated windings. Suppose that $c_M(\Gamma_M)$ and $c_N(\Gamma_N)$ have disjoint arrow sets, and that there are vertices $u \in (\Gamma_M)_0$ and $v \in (\Gamma_N)_0$ with $c_M(u) = c_N(v)$. Let $A$ denote the $\mathbb{F}_1$-representation of $Q$ associated to the amalgam $c_M \sqcup_{u\sim v} : \Gamma_M \sqcup_{u\sim v}\Gamma_N \rightarrow Q$. If $\operatorname{nice}(M) < \infty$ and $\operatorname{nice}(N)< \infty$, then $\operatorname{nice}(A) < \infty$ and $A$ is nice. 
\end{pro} 

\begin{proof} 
We will write $\Gamma_A = \Gamma_M \sqcup_{u\sim v}\Gamma_N$ to ease notation. Choose the points $u$, $v$ and $u=v$ as basepoints for $M$, $N$ and $A$. Let $X^{(i)}$, $Y^{(i)}$ and $Z^{(i)}$ denote the $i$-nice variables (defined in Construction \ref{con: universal}) for $M$, $N$ and $A$. Let $\mathcal{B}_M$ and $\mathcal{B}_N$ be the cycle bases of $\Gamma_M$ and $\Gamma_N$ associated to any two spanning trees containing the vertices $u$ and $v$, respectively. Then $\mathcal{B}_M \sqcup \mathcal{B}_N$ can be considered as a cycle basis for $\Gamma_A$. Since $c(\Gamma_M)$ and $c(\Gamma_N)$ have no arrows in common, the inclusions $\Gamma_M \hookrightarrow \Gamma_A$ and $\Gamma_N \hookrightarrow \Gamma_A$ induce a sequence of group isomorphisms 
\begin{equation}\label{eq: decomposition}
\mathcal{V}^{(i)}_A \xrightarrow[\cong]{f^{(i)}} \mathcal{V}^{(i)}_M \oplus \mathcal{V}^{(i)}_N
\end{equation}
satisfying $f^{(i)}(Z_w^{(i)}) = X^{(i)}_w$ whenever $w \in (\Gamma_M)_0$ and $f^{(i)}(Z_w^{(i)}) = Y_w^{(i)}$ whenever $w \in (\Gamma_N)_0$. Since any two vertices of $\Gamma_M$ can be distinguished by a nice sequence for $M$, it follows that they can be distinguished by a nice sequence for $A$ as well. Similarly, any two vertices of $\Gamma_N$ and can be distinguished by a nice sequence for $A$. By Remark \ref{rmk: weave}, one may begin any nice sequence for $A$ with the (finitely-many) gradings necessary to distinguish these vertices. Hence, it only remains to show that a vertex $w \in (\Gamma_M)_0$ can be distinguished from a vertex $z \in (\Gamma_N)_0$ by a nice sequence for $A$. By Corollary \ref{c: torsion}, this is equivalent to showing that $Z^{(i)}_w- Z^{(i)}_z \neq 0$ for some $i$. If $Z^{(i)}_w- Z^{(i)}_z = 0$ for all $i$, then $f^{(i)}(Z^{(i)}_w- Z^{(i)}_z) = X^{(i)}_w - Y^{(i)}_z = 0$ for all $i$ as well. But this would imply that $X^{(i)}_w$ and $Y^{(i)}_z$ were both $0$ for all $i$ from the decomposition \eqref{eq: decomposition}. But then $w$ cannot be distinguished from $u$ and $z$ cannot be distinguished from $v$, a contradiction. Hence, $Z^{(i)}_w - Z^{(i)}_z \neq 0$ for some $i$, which concludes the proof.
\end{proof}

In general, determining the $\FF_1$-representations of $Q$ which admit a positive grading is a subtle problem. Below, we describe a family of representations which admit positive and injective nice gradings. In particular, such representations are nice. These representations are interesting because they are defined by windings $c :\Gamma \rightarrow Q$, where $H_1(\Gamma,\mathbb{Z})$ can have arbitrarily large rank. The main results of \cite{Haupt2012euler} explicitly depend on $\Gamma$ being a tree or a band.

\begin{pro}\label{p: connected fibers}
Let $M$ be a nilpotent $\FF_1$-representation of $Q$ with associated winding $c_M$. Suppose that $c_M^{-1}(\alpha)$ is connected for all $\alpha \in Q_1$, and that $\Gamma_M$ contains a set of $\mathbb{Z}$-linearly independent\footnote{When considered as elements of $H_1(\Gamma_M,\mathbb{Z})$.} cycles $\{ X_1,\ldots , X_n\}$ with following properties: 
\begin{enumerate} 
\item [(i)] The cycles $[\iota\circ H_1(c_M)](X_1)$, $\ldots , [\iota\circ H_1(c_M)](X_n)$ form a $\mathbb{Q}$-basis for $\mathbb{Q}\otimes_{\mathbb{Z}}\operatorname{Im}(\iota \circ H_1(c_M))$, where $\iota\circ H_1(c_M)$ is as in \eqref{eq: construction homology}.
\item[(ii)] 
For all $i$ we can write $X_i = p_i - q_i$, where $p_i$ and $q_i$ are directed paths of positive length in $\Gamma_M$ with common source and target, but no interior vertices in common. 
\item[(iii)] 
For each $i\le n$, either $c_M(p_i)$ or $c_M(q_i)$ consists of arrows that do not appear in $c_M(X_j)$ for $j \neq i$, where we consider $c_M(X_j)$ as a subquiver of $Q$.
\end{enumerate}  
Then $\operatorname{nice}(M)\le 1$ and $M$ is nice.
\end{pro} 
\begin{proof}
For each $i$, write $p_i = \alpha_1^{(i)}\cdots \alpha_{d_i}^{(i)}$ and $q_i = \beta_1^{(i)}\cdots \beta_{e_i}^{(i)}$, where the $\alpha_j^{(i)}$ and $\beta_j^{(i)}$ are arrows in $\Gamma_M$. For each $i=1,\dots,n$, we let $Y_i=[\iota\circ H_1(c_M)](X_i)$. To begin, note that Property (i) implies that $\{ Y_1,\ldots , Y_n\}$ generates a full-rank subgroup of $\operatorname{Im}(\iota\circ H_1(c_M))$. This means that for any $x \in \operatorname{Im}(\iota \circ H_1(c_M))$, there exists a positive integer $m$ such that $mx$ is a $\mathbb{Z}$-linear combination of  $\{Y_1,\ldots , Y_n\}$. But by Remark \ref{r: arrows}, a nice grading on $M$ may be viewed as a function $ Q_1 \rightarrow \mathbb{Z}$ whose induced group homomorphism $\mathbb{Z}Q_1 \rightarrow \mathbb{Z}$ factors through $\operatorname{Im}(\iota \circ H_1(c_M))$. Let $\Delta : Q_1 \rightarrow \mathbb{Z}$ be any function such that the induced map (also denoted $\Delta$) satisfies $\Delta(Y_i) = 0$ for all $i$. Then in particular $0 = \Delta (mx) = m\Delta (x)$, hence $\Delta (x) = 0$ since $m$ is positive. It now follows from Property (ii) that a nice grading on $M$ is a map $\Delta : Q_1 \rightarrow \mathbb{Z}$ satisfying 
\[ 
\sum_{j=1}^{d_i}{\Delta(c_M(\alpha_j^{(i)}))} = \sum_{j=1}^{e_i}{\Delta(c_M(\beta_j^{(i)}))}, \text{ for all $i = 1,\ldots , n$.}
\] 
We now prove the claim by induction on $n$. If $n=1$, then defining $\Delta(\alpha_j^{(1)}) = e_1$ for all $j\le d_1$, $\Delta(\beta_j^{(1)}) = d_1$ for all $j \le e_1$, and $\Delta(\alpha) = 1$ otherwise yields such a map. This $\Delta$ is a positive nice grading, so that since $c_M^{-1}(\alpha)$ is connected for all $\alpha \in Q_1$, $M$ is nice by Proposition \ref{p: pos gradings}. Now suppose the claim holds for all $k < n$. By induction, we can define a map  
\[ 
\Delta' : \bigcup_{i<n}{(c_M(X_i))_1} \rightarrow \mathbb{N}\setminus\{0\}
\] 
such that  
\begin{equation}
\sum_{j=1}^{d_i}{\Delta'(c_M(\alpha_j^{(i)}))} = \sum_{j=1}^{e_i}{\Delta'(c_M(\beta_j^{(i)}))},\text{ for all $i<n$.}
\end{equation}
By multiplying $d_n$, we have the following:
\begin{equation}\label{e: soe strand}
\sum_{j=1}^{d_i}{d_n\Delta'(c_M(\alpha_j^{(i)}))} = \sum_{j=1}^{e_i}{d_n\Delta'(c_M(\beta_j^{(i)}))},\text{ for all $i<n$.}
\end{equation}
Define the map $\Delta$ to be the scalar multiple $\Delta = d_n\Delta'$:
\begin{equation}
\Delta : \bigcup_{i<n}{(c_M(X_i))_1} \rightarrow \mathbb{N}\setminus\{0\}, \quad c_M(\alpha^{i}_j) \mapsto d_n\Delta'(c_M(\alpha^{i}_j)).
\end{equation}
After possibly relabeling $p_n$ and $q_n$, we may assume via Property (iii) that $c_M(p_n)$ contains no arrows in $c_M(X_j)$ for $j<n$. Furthermore, since $c_M^{-1}(\alpha)$ is connected for all $\alpha \in Q_1$, Property (i) implies that $c_M(p_n)$ and $c_M(q_n)$ have no arrows in common. Then the condition that $\Delta$ extends to a nice grading on $\Gamma_M$ is precisely the requirement that there exist integers  
\begin{itemize}
\item $\Delta(c_M(\alpha_k^{(n)}))$ for $k = 1,\ldots d_n$, 
\item $\Delta(c_M(\beta_k^{(n)}))$, whenever $c_M(\beta_k^{(n)})$ not an arrow in $c_M(X_i)$ with $i<n$,  
\end{itemize}
\noindent  such that 
\begin{equation}\label{e: induction step}
\sum_{j=1}^{d_n}{\Delta(c_M(\alpha_j^{(n)}))} = \sum_{j=1}^{e_n}{\Delta(c_M(\beta_j^{(n)}))}.
\end{equation}
We construct such a $\Delta$ as follows: first, for any $c_M(\beta_k^{(n)})$ which are not arrows in $c(X_i)$ ($i<n$), define $\Delta(c_M(\beta_k^{(n)})) = d_n$. Then the right hand side of \eqref{e: induction step} is a positive integer divisible by $d_n$. Define 
\[ 
\Delta(c_M(\alpha_j^{(n)})) = \frac{1}{d_n} \sum_{j=1}^{e_n}{\Delta(c_M(\beta_j^{(n)}))}, \text{ for all $ j = 1,\ldots , d_n$.}
\] 
Finally, define $\Delta(\alpha) = 1$ for any remaining arrows $\alpha \in Q_1$. Then $\Delta$ is a positive grading on $M$, so that $\operatorname{nice}(M) \le 1$ by Proposition \ref{p: pos gradings}.
\end{proof}  

We now illustrate this result with some examples.  

\begin{myeg} 
Let $Q = \wild_3$. For the sake of illustration, let $\{ \alpha, \beta, \gamma\}$ denote the arrow set of $\wild_3$. Then the representation $M$ with coefficient quiver 
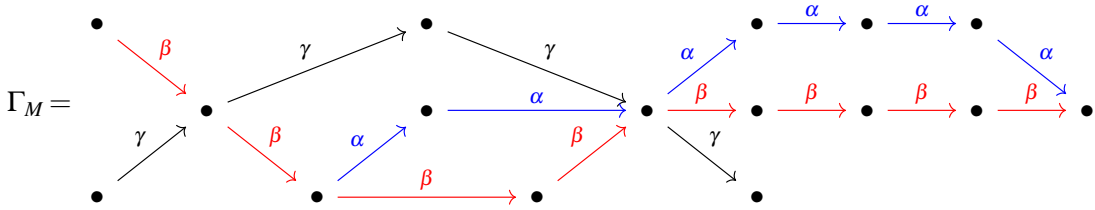
\[ 
\Gamma_M = 
\begin{tikzcd}
\bullet \arrow[dr,red,"\beta"] &            &        & \bullet \arrow[drr,black,"\gamma"] &       &       &  \bullet  \arrow[r,blue,"\alpha"]    & \bullet \arrow[r,blue,"\alpha"] &   \bullet     \arrow[dr,blue,"\alpha"]     &  &\\  
                                                & \bullet \arrow[dr,red,"\beta"]\arrow[urr,black,"\gamma"] &          & \bullet \arrow[rr,blue,"\alpha"] & &\bullet \arrow[dr,black,"\gamma"] \arrow[ur,blue,"\alpha"] \arrow[r,red,"\beta"] & \bullet \arrow[r,red,"\beta"] & \bullet \arrow[r,red,"\beta"] & \bullet \arrow[r,red,"\beta"] & \bullet & \\ 
\bullet \arrow[ur,black,"\gamma"] &           & \bullet \arrow[rr,red,"\beta"]\arrow[ur,blue,"\alpha"]&         &\bullet \arrow[ur,red,"\beta"] &     &   \bullet     &  &            &  &
\end{tikzcd}
\] 
satisfies the hypotheses of Proposition \ref{p: connected fibers} with 
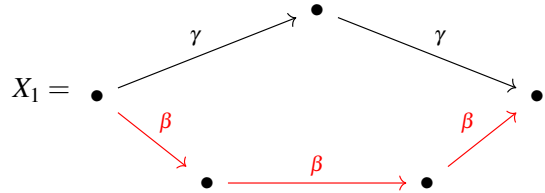
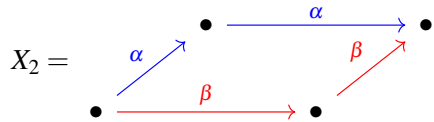
\[ 
X_1 =  
\begin{tikzcd} 
 & & \bullet \arrow[drr,black,"\gamma"] & &  \\ 
\bullet \arrow[urr,black,"\gamma"] \arrow[dr,red,"\beta"] & & & & \bullet \\ 
& \bullet \arrow[rr,red,"\beta"] & & \bullet \arrow[ur,red,"\beta"] & 
\end{tikzcd}
\]  
\noindent and
\[ 
X_2 =  
\begin{tikzcd} 
& \bullet \arrow[rr,blue,"\alpha"] & &  \bullet \\ 
\bullet \arrow[ur,blue,"\alpha"] \arrow[rr,red,"\beta"] & & \bullet \arrow[ur,red,"\beta"] & 
\end{tikzcd}
\]  
Here, we interpret the cycles clockwise so that $[\iota \circ H_1(c_M)](X_1) = 2\gamma - 3\beta$ and $[\iota \circ H_1(c_M)](X_2) = 2(\alpha - \beta)$. We exclude the cycle
\[ 
Y=
\begin{tikzcd} 
& \bullet \arrow[r,blue,"\alpha"]  & \bullet \arrow[r,blue,"\alpha"] &  \bullet \arrow[dr,blue,"\alpha"] & \\ 
\bullet \arrow[ur,blue,"\alpha"] \arrow[r,red,"\beta"]   & \bullet \arrow[r,red,"\beta"] & \bullet \arrow[r,red,"\beta"]  & \bullet \arrow[r,red,"\beta"]  & \bullet
\end{tikzcd} 
\] 
since $[\iota\circ H_1(c_M)](Y) = 4(\alpha - \beta) = 2[\iota \circ H_1(c_M)](X_2)$.
Note that the $1$-dimensional subrepresentations of $M$ are precisely the sinks of $\Gamma_M$. The $2$-dimensional subrepresentations are either a direct sum of two simples or indecomposable: the indecomposables are parameterized by sets of the form $\{ s, t\}$, where $s$ and $t$ are vertices of $\Gamma_M$, $t$ is a sink, and there is exists an arrow $s \xrightarrow[]{\alpha} t$ in $\Gamma_M$. It follows that $\chi_1(M_\mathbb{C}) = 2$ and $\chi_2(M_\mathbb{C}) = 4$. Similar considerations allows one to compute the remaining Euler characteristics.
\end{myeg}

\begin{myeg} \label{definition: strand rep}
Let $Q = \mathbb{L}_2$, the quiver with one vertex and two loops. We will let $(\mathbb{L}_2)_1 = \{\alpha_1, \alpha_2\}$, with $\alpha_1$-colored arrows appearing in blue and $\alpha_2$-colored arrows appearing in red. An $\mathbb{F}_1$-representation $M$ of $Q$ will be called a \emph{$2$-strand representation} if there exists a $2\times d$-matrix $X = \left( \begin{array}{cccc} m_1 & m_2 & \cdots & m_d \\ n_1 & n_2 & \cdots & n_d  \end{array} \right)$ with entries in $\mathbb{Z}_{>0}$ such that $\Gamma_M$ can be described as follows:  
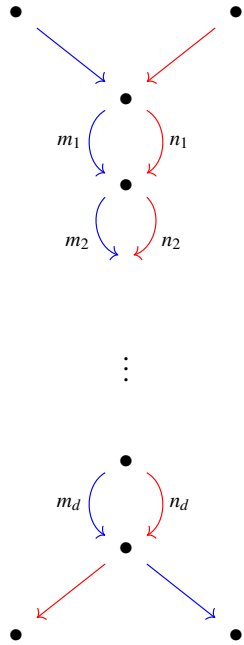
\[
\begin{tikzcd} 
\bullet \arrow[dr, blue] & & \bullet \arrow[dl, red] \\ 
 & \bullet \arrow[d, bend right = 60, swap, blue, "{\color{black}m_1}"] \arrow[d, bend left = 60, red, "{\color{black}n_1}"] &  \\  
 & \bullet \arrow[d, bend right = 60, swap, blue, "{\color{black}m_2}"] \arrow[d, bend left = 60, red, "{\color{black}n_2}"] &  \\ 
& \text{} & \\ 
& \vdots &  \\
 & \bullet \arrow[d, bend right = 60, swap, blue, "{\color{black}m_d}"] \arrow[d, bend left = 60, red, "{\color{black}n_d}"] &  \\ 
& \bullet \arrow[dr, blue] \arrow[dl, red] & \\ 
\bullet & & \bullet \\
\end{tikzcd}
\] 
where each colored arrow above represents an oriented path of length $m_i$ (resp. $n_i$), and an arrow with no label is allowed to be of arbitrary length. Up to translations, a nice grading on $M$ is a pair $(\Delta_1,\Delta_2) \in \mathbb{Z}^2$ satisfying 
\[ 
m_i\Delta_1 = n_i\Delta_2,\text{ for all $i = 1,\ldots , d$.}
\]
 Using Proposition \ref{p: connected fibers}, it follows that $M$ admits a positive grading if and only if $\operatorname{rank}(X) = 1$. Indeed, if $\operatorname{rank}(X) = 1$ then the columns of $X$ are $\mathbb{Q}$-linear multiples of each other and $\mathbb{Q}\otimes \operatorname{Im}(\iota \circ H_1(c_M)) = \mathbb{Q}(n_1\alpha_2 - m_1\alpha_1)$. In this case, we only need  
\[ 
X_1 = 
\begin{tikzcd} 
\bullet \arrow[d, bend right = 60, swap, blue, "{\color{black}m_1}"] \arrow[d, bend left = 60, red, "{\color{black}n_1}"]  \\ 
\bullet
\end{tikzcd}
\] 
to satisfy the hypotheses of Proposition \ref{p: connected fibers}. Note that Condition (iii) is vacuous since $n=1$. If $\operatorname{rank}(X)>1$, then there exist two columns $\left(\begin{array}{c}m_i \\ n_i\end{array} \right)$ and $\left(\begin{array}{c} m_j \\ n_j \end{array}\right)$ of $X$ that are linearly independent over $\mathbb{Q}$. In particular, 
\[ 
\det\left( \begin{array}{cc} m_i & m_j \\ -n_i & -n_j \end{array} \right)  = -\det\left( \begin{array}{cc} m_i & m_j \\ n_i & n_j \end{array} \right) \neq 0.
\]
 Then $m_i\Delta_1 = n_i\Delta_2$ and $m_j\Delta_1 = n_j\Delta_2$ implies
\[ 
\left(\begin{array}{cc} m_i & m_j \\ -n_i & -n_j \end{array} \right) \left(\begin{array}{c} \Delta_1 \\ \Delta_2 \end{array} \right) = \left( \begin{array}{c} 0 \\ 0 \end{array} \right),
\] 
 which has only the trivial solution $\left(\begin{array}{c} \Delta_1 \\ \Delta_2 \end{array} \right) = \left( \begin{array}{c} 0 \\ 0 \end{array} \right)$.
\end{myeg} 

The following are more explicit examples. 

\begin{myeg} 
Let $M$ be a representation $M=(M_0,f_1,f_2)$ of $\mathbb{L}_2$, where $M_0=\{0, 1,2,3\}$,
	\[
	\begin{tikzcd}
		\bullet \arrow[loop left,looseness=20,"f_1"]
		\arrow[loop right, looseness=20,"f_2"]
	\end{tikzcd}
	\]
	such that
	\[
	f_1(1)=2, \quad f_1(2)=3,\quad  f_1(3)=0, 
	\]
	and
	\[
	f_2(1)=2\quad f_2(2)=3,\quad f_2(3)=0.
	\]
	Then, $\Gamma_M$ can be described as follows ($f_1$ is in blue and $f_2$ is in red):  
	
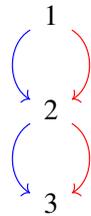
\[
	\begin{tikzcd} 
		& 1 \arrow[d, bend right = 60, swap, blue, "{\color{black}}"] \arrow[d, bend left = 60, red, "{\color{black}}"] &  \\  
		& 2 \arrow[d, bend right = 60, swap, blue, "{\color{black}}"] \arrow[d, bend left = 60, red, "{\color{black}}"] &  \\ 
		& 3  & 
	\end{tikzcd}
	\] 
		Then $M_\mathbb{C}$ is a representation of $\wild_2$ over $\mathbb{C}$ given as follows:
	
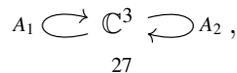
\[
	\begin{tikzcd}
		\mathbb{C}^3 \arrow[loop left,looseness=10,"A_1"]
		\arrow[loop right, looseness=10,"A_2"]
	\end{tikzcd},
	\]
	where 
	\[
	A=A_1=A_2=\begin{bmatrix}
		0 & 0 & 0\\
		1 & 0 & 0\\
		0 & 1 & 0
	\end{bmatrix}
	\]
	So, in this case $\textrm{Gr}_{1}(M_\mathbb{C})$ consists of the lines of $\mathbb{C}^3$ invariant under the linear transformation $A$. One can easily see that this set consists only of the line spanned by the vector $\begin{bmatrix}
		0\\
		0\\
		1
	\end{bmatrix}$. It follows that $\chi_{1}(M_\mathbb{C})=1$. On the other hand, following the notation in Example \ref{definition: strand rep}, $M$ is a $2$-strand representation in $\Rep(\wild_2,\FF_1)_{\nil}$ with associated matrix 
	\[
	X=\begin{bmatrix}
		1 & 1\\
		1 & 1
	\end{bmatrix}
	\]
	It follows from the discussion in Example \ref{definition: strand rep} that $M$ is nice, in particular, we have
	\begin{equation}
		\chi_{1}(M_\mathbb{C}) = |\{ N \le M \mid \textbf{\textrm{dim}}(N) = 1 \}|=1.
	\end{equation}
When $\textbf{e}=2$, one may obtain $\chi_{\textbf{e}}(M_\mathbb{C})=1$ from the duality in $\mathbb{C}$ (or applying the same matrix computation as in the case for $d=1$). This is also clear as follows:
	\begin{equation}
		\chi_{2}(M_\mathbb{C}) = |\{ N \le M \mid \textbf{\textrm{dim}}(N) = 2 \}|=1
	\end{equation}
\end{myeg}  

\begin{myeg} 
Let $M$ be a representation $M=(M_0,f_1,f_2)$ of $\mathbb{L}_2$, where $M_0=\{0, 1,2,3\}$ such that
\[
	f_1(1)=3, \quad f_1(2)=f_1(3)=0, \quad 
	\]
	and
	\[
	f_2(1)=2\quad f_2(2)=3,\quad f_2(3)=0.
	\]
	Then, $\Gamma_M$ can be described as follows ($f_1$ is in blue and $f_2$ is in red):  
	
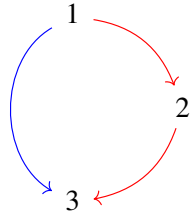
\[
	\begin{tikzcd} 
		& 1 \arrow[dd, bend right = 60, swap, blue, "{\color{black}}"]  \arrow[dr, bend left = 30, red, "{\color{black}}"]&  \\  
		&  &  2 \arrow[dl, bend left = 30, red, "{\color{black}}"]\\ 
		& 3  & 
	\end{tikzcd}
	\] 
	Then $M_\mathbb{C}$ is a representation of $\wild_2$ over $\mathbb{C}$ given as follows:
	\[
	\begin{tikzcd}
		\mathbb{C}^3 \arrow[loop left,looseness=10,"A_1"]
		\arrow[loop right, looseness=10,"A_2"]
	\end{tikzcd},
	\]
	where 
	\[
A_1=\begin{bmatrix}
		0 & 0 & 0\\
		0 & 0 & 0\\
		1 & 0 & 0
	\end{bmatrix}, \quad A_2=\begin{bmatrix}
		0 & 0 & 0\\
		1 & 0 & 0\\
		0 & 1 & 0
	\end{bmatrix}
	\]
Again, in this case $\textrm{Gr}_{1}(M_\mathbb{C})$ consists only of the line spanned by the vector $\begin{bmatrix}
	0\\
	0\\
	1
\end{bmatrix}$.  It follows that $\chi_{1}(M_\mathbb{C})=1$. On the other hand, $M$ is a $2$-strand representation with associated matrix 
	\[
	X=\begin{bmatrix}
		1 \\
		2
	\end{bmatrix}
	\]
	It follows from Example \ref{definition: strand rep} that $M$ is nice, and we have
	\begin{equation}
		\chi_{1}(M_\mathbb{C}) = |\{ N \le M \mid \textbf{\textrm{dim}}(N) = 1 \}|=1.
	\end{equation}
\end{myeg}  

The following example illustrates how our gluing procedure in Proposition \ref{p: gluing} creates a nice grading. 

\begin{myeg} 
Let $Q = \mathbb{L}_4$, the quiver with one vertex and four loops. We will write $Q_1 = \{\alpha_1,\alpha_2,\alpha_3,\alpha_4\}$. Consider the $2$-strand representation 
\[ M=
\begin{tikzcd}  
 & v_1 \arrow[dl,blue,swap,"\alpha_1"]\arrow[dr,red,"\alpha_2"]  & \\ 
v_2 \arrow[dr,blue,"\alpha_1"] & & v_3 \arrow[dl,red,"\alpha_2"]  \\ 
 & v_4 &  \\
\end{tikzcd}.
\]   
By the discussion above, there is a nice sequence $\underline{\partial}$ distinguishing its vertices. For instance, if we define $\partial_0$ as 
\[ 
\partial_0 = 
\begin{tikzcd}  
 & 0 \arrow[dl,blue,swap,"\alpha_1"]\arrow[dr,red,"\alpha_2"]  & \\ 
1 \arrow[dr,blue,"\alpha_1"] & & 1 \arrow[dl,red,"\alpha_2"]  \\ 
 & 2 &  \\
\end{tikzcd},
\] 
then any map $(\Gamma_M)_0 \rightarrow \mathbb{Z}$ is a $\partial_0$-nice grading. Hence, we can take $\partial_1(v_i) = i$ for all $i \le 4$, and set $\partial_j = 0$ for all $j> 1$. 

Now define $N$ to be the string representation 
\[ 
N = 
\begin{tikzcd} 
v_1' \arrow[r,green,"\alpha_3"] & v_2' \arrow[r,"\alpha_4"] & v_3' & v_4' \arrow[l,green,"\alpha_3",swap] & v_5' \arrow[l,"\alpha_4",swap] \\
\end{tikzcd}.
\] 
Note that any vector $(\Delta_1,\Delta_2,\Delta_3,\Delta_4) \in \mathbb{Z}^4$ induces a nice grading on $N$, but no choice distinguishes $v_1'$ and $v_5'$. Nevertheless, there is a nice sequence $\underline{\partial}'$ for $N$ which distinguishes its vertices. For instance, we can set 
\[ 
\partial_0' =  
\begin{tikzcd} 
0 \arrow[r,green,"\alpha_3"] & 1 \arrow[r,"\alpha_4"] & 3 & 2 \arrow[l,green,"\alpha_3",swap] & 0 \arrow[l,"\alpha_4",swap] \\
\end{tikzcd},
\] 
so that any map $(\Gamma_N)_0 \rightarrow \mathbb{Z}$ is a $\partial_0'$-nice grading. As before, set $\partial_1'(v_i') = i$ for $i \le 5$ and $\partial_j' = 0$ for $j> 1$.  

We can glue $M$ and $N$ by identifying $v_1 = v_1'$ to obtain another representation $A$:  

\[ A=
\begin{tikzcd}  
 &(v_1=v_1') \arrow[dl,blue,swap,"\alpha_1"]\arrow[dr,red,"\alpha_2"]  \arrow[r,green,"\alpha_3"] & v_2' \arrow[r,"\alpha_4"] & v_3' & v_4' \arrow[l,green,"\alpha_3",swap] & v_5' \arrow[l,"\alpha_4",swap] \\ 
v_2 \arrow[dr,blue,"\alpha_1"] & & v_3 \arrow[dl,red,"\alpha_2"] & & & \\ 
 & v_4 & & & &\\
\end{tikzcd}.
\] 
Proposition \ref{p: gluing} ensures that $A$ admits a nice sequence which distinguishes vertices.   
\end{myeg}


\subsection{Low-dimensional niceness results} In this subsection we give new proofs of results first published in \cite{irelli2011quiver,Haupt2012euler}. These new proofs fix apparent gaps in the proofs of Lemmas 6.3-6.4 of \cite{Haupt2012euler}, see the appendix at the end of this article for more details. In particular, we show that an $\mathbb{F}_1$-representation $M$ of $Q$ is nice if $\Gamma_M$ is either a tree or a (primitive) affine Dynkin quiver of type $\tilde{\mathbb{A}}_n$. We then apply these results to classify the representations of finite nice length when $Q$ is a pseudotree.

\begin{rmk}\label{r.identification}
Let $M$ be an $\FF_1$-representation of $Q$. Suppose that $S \subseteq \Gamma_M$ is a subquiver that is a deformation retract, and let $N$ denote the representation induced by the restriction of $c_M$ to $S$. Then $\Gamma_N = S$, $H_1(\Gamma_N, \mathbb{Z}) \cong H_1(\Gamma_M, \mathbb{Z})$, and we can find a cycle basis for $\Gamma_N$ which is also a cycle basis for $\Gamma_M$. In other words, $H_1(\Gamma_N, \mathbb{Z})$ and $H_1(\Gamma_M, \mathbb{Z})$ correspond to the same subgroup of $\mathbb{Z}Q_1$. It follows that there is a commutative diagram 
 \begin{equation}
\begin{tikzcd} 
H_1(\Gamma_N, \mathbb{Z}) \arrow[r,"\psi"] \arrow[d,"\phi_N"]  & H_1(\Gamma_M, \mathbb{Z})  \arrow[d,"\phi_M"] \\ 
\mathbb{Z}Q_1 \arrow[r,"\operatorname{id}"] & \mathbb{Z}Q_1 
\end{tikzcd}
\end{equation}
where $\psi$ is an isomorphism, $\phi_M = \iota_1 \circ H_1(c_M)$ and  $\iota_N = \iota_2 \circ H_1(c_N)$, where $\iota_1$ (resp. $\iota_2)$ is the inclusion of the image of $H_1(c_M)$ (resp. $H_1(c_N)$) into $\mathbb{Z}Q_1$. This allows us to identify $\operatorname{coker}(\iota_1\circ H_1(c_M))$ with $\operatorname{coker}(\iota_2\circ H_1(c_N))$, which in turn allows us to identify nice gradings of $M$ with nice gradings of $N$. We will make such identifications freely throughout this text.
\end{rmk}

Recall that any winding $c:\Gamma \to Q$ can be considered as a certain coloring of $\Gamma$, where $\Gamma_0$ (resp.~$\Gamma_1$) is colored by $Q_0$ (resp.~$Q_1$). In what follows, we interchangeably use the terms colored quivers and coefficient quivers. 

\begin{lem} 
Let $c_1,\dots,c_m$ be a basis for $H_1(\Gamma_M)$. Suppose that $v$ is a vertex of $\Gamma_M$ that lies in none of the cycles $c_1,\ldots , c_m$. Let $\Gamma_M\setminus\{v\}$ denote the colored subquiver of $\Gamma_M$ obtained by deleting $v$ and all arrows incident to it. Then $\Gamma_M$ admits a non-trivial (resp. non-degenerate, positive, negative) grading if and only if $\Gamma_M\setminus\{v\}$ does.
\end{lem}  
\begin{proof} 
Since $\Gamma_M\setminus\{v\}$ is a deformation retract of $\Gamma_M$, this directly follows from Remark \ref{r.identification}.
\end{proof}

\begin{lem}\label{l: uni walks} 
Let $c : \Gamma \rightarrow Q$ be a winding with $\Gamma$ a Dynkin quiver of type $\mathbb{A}_n$, and let $M$ denote the associated representation. Pick a basepoint $b \in \Gamma_0$. Then there exists an $N$ such that for all $i \geq N$:
\begin{enumerate} 
\item The universal $i$-nice grading $X^{(i)} : \Gamma_0 \rightarrow \mathcal{V}^{(i)}_M$ is injective. 
\item If the distance from $v$ to $b$ is $d$, then $X^{(i)}_v$ is a linear combination of $d$ distinct elements of $Q^{(i)}_1$ with non-zero coefficients.
\end{enumerate}
 In particular, from Corollary \ref{c: torsion}, $\operatorname{nice}(M)<\infty$ and $M$ is nice.
\end{lem} 
\begin{proof} 
Write $\Gamma = \alpha_1^{\epsilon_1}\cdots \alpha_{n-1}^{\epsilon_{n-1}}$ throughout. We start by proving the following claim: if there is no such $N$ such that (1) holds for all $i\geq N$, then $\Gamma$ contains distinct arrows $\beta, \gamma \in \Gamma_1$ with $c(\beta) = c(\gamma)$ and either $s(\beta) = s(\gamma)$ or $t(\beta) = t(\gamma)$, contradicting the assumption that $c$ is a winding.
 The cases $n = 2, 3$ can be verified through direct computation. Suppose the claim holds for all $k \leq n$. If no such $N$ exists then Corollary \ref{c: torsion} implies that there exist vertices $u, v \in \Gamma_0$ such that $X^{(i)}_u = X^{(i)}_v$ for all $i$. Indeed, $X^{(i)}_u \neq X^{(i)}_v$ implies $X^{(i+1)}_u \neq X^{(i+1)}_v$ since $H_1(\tau^{(i+1)})(X^{(i+1)}_z) = X^{(i)}_z$ for all $z \in \Gamma_0$ (see Remark \ref{rmk: universal}). If $\{u,v \} \neq \{ s(\alpha_1^{\epsilon_1}), t(\alpha_{n-1}^{\epsilon_{n-1}}) \}$ then the claim follows from induction, so without loss of generality assume $u = s(\alpha_1^{\epsilon_1})$ and $v = t(\alpha_{n-1}^{\epsilon_{n-1}})$. By Remark \ref{r: basepoint} we may assume $b = u$ so that $X^{(i)}_u = X^{(i)}_v = 0$ for all $i$. Since $\Gamma$ is simply connected as a topological space, $\mathcal{V}^{(i)}_M$ is contained in $\mathbb{Z}Q_1^{(i)}$ for each $i$. Furthermore,  
\[0 = X^{(i)}_v = \epsilon_{n-1}\alpha_{n-1}^{(i)} + X^{(i)}_{s(\alpha_{n-1}^{\epsilon_{n-1}})}.\]  
This implies that there exists a $j <n-1$ such that the following equations hold for all $i$:
\begin{equation}\label{e: walks1}
\epsilon_j + \epsilon_{n-1} = 0,
\end{equation}
\begin{equation}\label{e: walks2}
\alpha_j^{(i)} = \alpha_{n-1}^{(i)}.
\end{equation} 
Indeed, for each $i$ it is clear that a $j$ depending on $i$ exists. To prove that $j$ can be chosen independently of $i$, for each $j< n-1$ let $S_j = \{ i \in \mathbb{N} \mid \text{\eqref{e: walks1} and \eqref{e: walks2} hold for $\alpha_j^{(i)}$} \}$. By the pigeonhole principle, $S_j$ must be unbounded for a fixed $j$. But then \eqref{e: walks2} implies that $\alpha_j^{(i')} = \alpha_{n-1}^{(i')}$ for all $i'<i$, again using Remark \ref{rmk: universal}. It follows that $S_j$ is an unbounded, successor-closed subset of the poset $(\mathbb{N},\le)$, and so $S_j = \mathbb{N}$. Having established the existence of such a $j$, the definition of $Q^{(i)}$ implies
\[ 
X^{(i)}_{s(\alpha_j^{\epsilon_j})} = X^{(i)}_{s(\alpha_{n-1}^{\epsilon_{n-1}})},
\] 
which in turn implies that the sub-walk from $s(\alpha_j^{\epsilon_j})$ to $s(\alpha_{n-1}^{\epsilon_{n-1}})$ contains the desired subquiver by the induction hypothesis. Hence Claim (1) holds, since $c: \Gamma \rightarrow Q$ is a winding. So, choose an $N'$ such that $X^{(i)}$ is injective for all $i \geq N'$. Then by construction $\sigma^{(N'+1)}: \Gamma \rightarrow \Gamma^{(N'+1)}$ will be injective on arrows. Hence, Claims (1) and (2) will hold for $N = N'+1$. 
\end{proof} 

The following corollary confirms the truth of \cite[Lemma 6.3]{Haupt2012euler}. See the appendix at the end of this article for a discussion of the original proof.

\begin{cor}[cf. Lemma 6.3 of \cite{Haupt2012euler}]\label{c: distinguishing trees}
Let $c : \Gamma \rightarrow Q$ be a winding with associated representation $M$. Suppose that $\Gamma$ is a tree. Then $\operatorname{nice}(M)<\infty$ and $M$ is nice.
\end{cor} 
\begin{proof} 
Suppose that no nice sequence distinguishes $x, y \in \Gamma_0$. Let $p$ denote the unique walk of minimal length from $x$ to $y$, and select $x$ as the basepoint. Since $\Gamma$ is retractable, Remark \ref{r.identification} implies that the $i$-nice variables of vertices in $p$ are the same for $c$ and the restricted representation $c|_p$, for all $i \geq 0$. In other words, the restricted representation $c|_p : p \rightarrow Q$ cannot distinguish $x$ and $y$, contradicting Lemma \ref{l: uni walks}.
\end{proof}

\begin{mydef} \label{definition: primitive}
Let $c : \Gamma \rightarrow Q$ be a winding with $\Gamma$ an affine Dynkin quiver of type $\tilde{\mathbb{A}}_n$. Write $\Gamma = \beta_1^{\delta_1}\cdots \beta_d^{\delta_d}$, where $s(\beta_1^{\delta_1}) = t(\beta_d^{\delta_d})$. If the cycle $c(\Gamma)$ is primitive (in the sense that $c(\Gamma) = c(\beta_1)^{\delta_1}\cdots c(\beta_d)^{\delta_d} \neq q^{k}$ for a cycle $q$ of $Q$ with $k>1$) then we say that $c$ is \emph{primitive}. We make a similar definition for a cycle in a general winding.
\end{mydef}

The following lemma confirms the truth of \cite[Lemma 6.4]{Haupt2012euler} in the case of an $\mathbb{F}_1$-representation. See the appendix at the end of this article for a discussion of the original proof.

\begin{mythm}[cf. Lemma 6.4 of \cite{Haupt2012euler}]\label{l: distinguishing bands}
Let $c : \Gamma \rightarrow Q$ be a winding with $\Gamma$ an affine Dynkin quiver of type $\tilde{\mathbb{A}}_n$. Let $M$ be the associated $\mathbb{F}_1$-representation of $Q$. Then $\operatorname{nice}(M)<\infty$ if and only if $c$ is primitive.
\end{mythm}  
\begin{proof} 
$(\implies)$ Suppose that $c$ is not primitive. Then there exists a cycle $q = \alpha_1^{\epsilon_1}\cdots \alpha_d^{\epsilon_d}$ in $Q$ such that $c(\Gamma) = q^k$ for some $k>1$. Since $c : \Gamma \rightarrow Q$ factors through the inclusion $q \hookrightarrow Q$, we may assume without loss of generality that $q = Q$. Then $c : \Gamma \rightarrow Q$ corresponds to wrapping around the cycle $q$ $k$-times. In other words, we may write  
\[
\Gamma = \beta_{11}^{\epsilon_1}\beta_{12}^{\epsilon_2}\cdots \beta_{1d}^{\epsilon_d}\beta_{21}^{\epsilon_1}\cdots\beta_{2d}^{\epsilon_d}\cdots\beta_{kd}^{\epsilon_d} 
\]
where $\beta_{ij} \in \Gamma_1$ satisfies $c(\beta_{ij}) = \alpha_j$ for all $i$ and $j$. Note that $c(\beta_{i1}^{\epsilon_1}\cdots \beta_{id}^{\epsilon_d}) = c(\beta_{i1}^{\epsilon_1})\cdots c(\beta_{id}^{\epsilon_d})= q$ for all $i \le k$. Pick $b = s(\beta_{11}^{\epsilon_1})$ as basepoint. Note that $\operatorname{Im}(\iota \circ H_1(c))$ is generated by $c(\Gamma) = kq \in \mathbb{Z}Q_1$. It follows that $q$ is torsion in $\mathbb{Z}Q_1/\operatorname{Im}(\iota \circ H_1(c))$, and so $q = 0$ as an element of $\mathcal{V}^{(0)}_M$. From this, it follows that  
\[ 
X^{(0)}_{t(\beta_{ij}^{\epsilon_j})} = \sum_{s=1}^{j}{\epsilon_s\alpha_s}
\] 
for all $i \le k$ and $j\le d$. In particular, $X^{(0)}_{t(\beta_{ij}^{\epsilon_j})} = X^{(0)}_{t(\beta_{1j}^{\epsilon_j})}$ for all $j \le d$. Note that $\Gamma^{(1)}$ can then be identified with $q$, and $\sigma^{(1)} : \Gamma \rightarrow \Gamma^{(1)}$ is the map which wraps around $q$ $k$-times. Repeating the above argument for each $s\geq 0$, we see that $X^{(s)}_{t(\beta_{ij}^{\epsilon_j})} = X^{(s)}_{t(\beta_{1j}^{\epsilon_j})}$ for all $s$. It follows that for each $i \le k$, there is no nice sequence for $M$ distinguishing the vertices $\{ t(\beta_{ij}^{\epsilon_j}) \mid j = 1,\ldots, d \}$.  

$(\impliedby)$ Suppose that $c$ is primitive, but no nice sequence distinguishes the vertices $u, v \in \Gamma_0$. Let $p$ and $q$ be the two walks (considered as subquivers) in $\Gamma$ from $u$ to $v$. Suppose that $p = \alpha_1^{\epsilon_1}\cdots \alpha_d^{\epsilon_d}$ and $q^{-1} = \beta_1^{\delta_1}\cdots \beta_e^{\delta_e}$ and that $\ell(p) \le \ell(q)$, where $\ell(p)$ and $\ell(q)$ are the lengths of $p$ and $q$ respectively. Without loss of generality, we may assume that $u$ and $v$ are chosen minimal, in the sense that if the distance between $u'$ and $v'$ is strictly less than $\ell(p)$ then there is a nice sequence which distinguishes $u'$ and $v'$. We first claim that there must exist an $i$ such that $X^{(i)}_u = X^{(i)}_v$, but $X^{(i)}_w \not\in \{ X^{(i)}_u,X^{(i)}_v\}$ for all interior vertices $w \in p_0$. In particular, one has
$\{ X^{(i)}_w \mid w \in q_0\} \not\subseteq \{ X^{(i)}_z \mid z \in p_0\}$.
Indeed, to prove the claim, assume by way of contradiction that $\{ X^{(i)}_w \mid w \in q_0\} \subseteq \{ X^{(i)}_z \mid z \in p_0\}$. We may further assume that $X^{(i)}_v = X^{(i)}_u=0$.  Consider the variables in $q^{-1} = \beta_1^{\delta_1}\cdots \beta_e^{\delta_e}$ starting with the vertex $t(\beta_1^{\delta_1})$: this vertex is adjacent to $s(\beta_1^{\delta_1}) = v$, whose $i$-nice variable is $X^{(i)}_v = 0$. The only two vertices in $p$ that are adjacent to vertices whose $i$-nice variables are $0$, are $t(\alpha_1^{\epsilon_1})$ and  $s(\alpha_d^{\epsilon_d})$: by $\FF_1$-linearity\footnote{From which we do not allow subquivers of the form $\bullet \xrightarrow[]{\alpha} \bullet \xleftarrow[]{\alpha} \bullet$ or $\bullet \xleftarrow[]{\alpha} \bullet \xrightarrow[]{\alpha} \bullet$, with $\alpha \in Q_1$.} $X^{(i)}_{t(\beta_1^{\delta_1})} \neq X^{(i)}_{s(\alpha_d^{\epsilon_d})}$, so we must have $X^{(i)}_{t(\beta_1^{\delta_1})}= X^{(i)}_{t(\alpha_1^{\epsilon_1})}$. In particular, this forces $c(\beta_1^{\delta_1}) = c(\alpha_1^{\epsilon_1})$. We can proceed in this fashion for all vertices of $q^{-1}$: given any vertex of $q^{-1}$, there is exactly one way to choose a variable for it from $\{ X^{(i)}_z \mid z \in p_0\}$ which does not contradict $\FF_1$-linearity. Thus we are forced to conclude that $c(q^{-1}) = c(p)^m$ for some $m$, contradicting primitivity. Hence, there must exist some $w \in q_0$ such that $X^{(i)}_w \not\in \{ X^{(i)}_z \mid z \in p_0\}$, in particular
\begin{equation}\label{eq: 131}
\{ X^{(i)}_w \mid w \in q_0\} \not\subseteq \{ X^{(i)}_z \mid z \in p_0\}
\end{equation} 
as we wished to show.

Now, we construct a nice sequence distinguishing $u$ and $v$ by using \eqref{eq: 131} which would give us a contradiction. Let $T$ be the tree quiver obtained from $\Gamma$ by splitting $v$ into two separate vertices $\lambda$ and $\rho$ (we will say that $p$ goes from $u$ to $\lambda$ and $q$ goes from $u$ to $\rho$). There is a winding map $\tilde{c} : T \rightarrow \Gamma$ which maps $\lambda, \rho\mapsto v$ and acts as the identity on the remaining vertices. This map induces a bijection on arrows, so we identify $T_1$ with $\Gamma_1$. For each $i$, let $Y^{(i)}_z$ denote the $0$-nice variables on the representation $\sigma^{(i)}\circ \tilde{c}: T \rightarrow Q^{(i)}$ with $u$ as basepoint. Note that $Y^{(i)}_z$ is a coset representative for $X^{(i)}_{\tilde{c}(z)}$ for all $i$ and $z$. We verify this for the vertices in $p$: the argument for vertices in $q$ is similar. If $z = t(\alpha_j^{\epsilon_j})$ then  
\begin{align*} 
Y^{(i)}_z & = \sum_{k=1}^j{\epsilon_k\alpha_k^{(i)}} +\operatorname{Im}\left(\iota^{(i)}\circ H_1(\sigma^{(i)}\circ \tilde{c})\right)   \\ 
& = \sum_{k=1}^j{\epsilon_k\alpha_k^{(i)}} + (0) \in \mathbb{Z}\Gamma^{(i)}_1/(0) \cong \mathbb{Z}\Gamma^{(i)}_1,
\end{align*}  
since $T$ is a tree and hence induces the zero map on homology (note again that we write $\tilde{c}(\alpha_k) = \alpha_k$ since $\tilde{c}$ induces a bijection on arrows). Of course, the walk from $u$ to $\tilde{c}(z)$ in $\Gamma$ is just $\alpha_1^{\epsilon_1}\cdots \alpha_j^{\epsilon_j}$ so that 
\begin{equation*} 
X^{(i)}_{\tilde{c}(z)} = \sum_{k=1}^j{\epsilon_k\alpha_k^{(i)}} +\operatorname{Im}\left(\iota^{(i)}\circ H_1(\sigma^{(i)})\right) 
\end{equation*} 
as claimed. On the other hand, for each $z \in p_0$ (resp. $z \in q_0$) the variable $Y^{(i)}_z$ may be identified with the $i$-nice variable for $\tilde{c}(z)$ with respect to the restricted winding $c|_p$ with basepoint $\tilde{c}(u) = u$ (resp. $c\mid_q$ with basepoint $u$). By Lemma \ref{l: uni walks} there exists an $i'$ such that the $Y^{(i')}$ are injective on $p_0$ and $q_0$, and such that $Y^{(i')}_z$ is a sum of $t$ distinct elements of $Q^{(i')}_1$, were $t$ is the distance from $u$ to $z$ in $T$ (which is the same as the distance from $u$ to $\tilde{c}(z)$ in $\Gamma$). Since the arrows of $Q^{(i')}$ are determined by the nice variables $X^{(j)}$ with $j \le i'$, and since $\{X^{(i)}_w \mid q_0\} \not\subseteq \{X^{(i)}_z \mid z \in p_0 \}$, it follows that $Y^{(N+1)}_{\rho}$ contains an arrow of $Q^{(N+1)}$ that does not appear in $Y^{(N+1)}_{\lambda}$, where $N = \max\{i,i'\}$. In other words, $Y^{(N+1)}_{\rho}$ and $Y^{(N+1)}_{\lambda}$ are linearly independent in $\mathbb{Q}Q^{(N+1)}_1$. If  
\[(\pmb{\Delta},Y^{(i)}_z) \mapsto Y^{(i)}_z\mid_{\pmb{\Delta}}\] 
denotes the usual pairing
\[ 
(\mathbb{Q}Q^{(N+1)}_1)^*\times \mathbb{Q}Q^{(N+1)}_1 \rightarrow \mathbb{Q} 
\] 
induced by $Q^{(N+1)}_1$ and its dual basis\footnote{More explicitly: if $Q$ is \emph{any} quiver, then $Q_1$ is a basis for $\mathbb{Q}Q_1$ and the dual basis for $(\mathbb{Q}Q_1)^*$ is $\{ \chi_{\alpha} \mid \alpha \in Q_1\}$, where $\chi_{\alpha}(\beta) = \delta_{\alpha \beta}$ for all $\beta \in Q_1$. If $ Y = \sum{y_{\alpha}\alpha} \in \mathbb{Q}Q_1$ and $\pmb{\Delta} = \sum{z_{\alpha}\chi_{\alpha}} \in (\mathbb{Q}Q_1)^*$, then $Y\mid_{\pmb{\Delta}} = \pmb{\Delta}(Y) = \sum_{\alpha, \beta}{z_{\alpha}y_{\beta}\chi_{\alpha}(\beta)}$.}, then the $\mathbb{Q}$-linear map  
\[ (\mathbb{Q}Q^{(N+1)}_1)^* \rightarrow \mathbb{Q}^2 \] 
\[ 
\pmb{\Delta} \mapsto \left( \begin{array}{c} Y^{(N+1)}_{\lambda}\mid_{\pmb{\Delta}} \\ Y^{(N+1)}_{\rho}\mid_{\pmb{\Delta}}  \end{array} \right) 
\] 
has rank $2$, and so it intersects the diagonal of $\mathbb{Q}^2$ nontrivially. In other words, there is a $\pmb{\Delta} \neq 0$ such that 
\[ 
Y^{(N+1)}_{\lambda}\mid_{\pmb{\Delta}} = Y^{(N+1)}_{\rho}\mid_{\pmb{\Delta}} \neq 0.
\]   
If necessary, we can replace $\pmb{\Delta}$ with a suitably large integral multiple to assume that $\pmb{\Delta}$ is an integer vector. In turn, this implies the existence of a nice sequence $\underline{\partial}$ such that   
\[ 
\partial_{N+1}(v) = \partial_{N+1}(u) + Y^{(N+1)}_{\lambda}\mid_{\pmb{\Delta}},
\] 
a contradiction!
\end{proof}  

If $Q$ is a pseudotree with central cycle $C$, and $M$ is an indecomposable $\mathbb{F}_1$-representation of $Q$, then the restriction $\operatorname{Res}_C(M)$ of $M$ to $C$ is also indecomposable (see \cite{jun2020quiver}). Since indecomposables for $C$ are either (extended) Dynkin quivers of type $\mathbb{A}_n$ or $\tilde{\mathbb{A}}_n$, it follows from a straightforward argument that $\Gamma_M$ is a pseudotree. Hence, it makes sense to talk about the central cycle of $\Gamma_M$ (if it has one).


\begin{cor} \label{corollary: fintie nice length for pseudotree}
Let $Q$ be a pseudotree, and let $c : \Gamma \rightarrow Q$ be a winding with associated representation $M$. Then $\operatorname{nice}(M)< \infty$ if and only if the central cycle of $\Gamma_M$ is primitive.
\end{cor} 

\begin{proof} 
If $\Gamma$ is a tree or of type $\tilde{\mathbb{A}}_n$ there is nothing to show, so we may assume that $\Gamma$ is a proper pseudotree with central cycle $C$. First suppose that $C$ is primitive and hence $\operatorname{Res}_C(M)$ has finite nice length by Theorem \ref{l: distinguishing bands}. Note that $\Gamma$ may be obtained by gluing a tree (possibly trivial) to each vertex of $C$, and that these trees have no arrow colors in common with each other or with $C$. Since tree representations have finite nice length by Corollary \ref{c: distinguishing trees}, the result now follows from Proposition \ref{p: gluing}. The converse is clear from Theorem \ref{l: distinguishing bands}.
\end{proof}

\section{Hall algebras arising from $\FF_1$-representations}  \label{section: Hall}

\subsection{Hall algebras and coefficient quivers} \label{s: Hall Algebras} 

In \cite{szczesny2011representations}, Szczesny discussed the following problem: given a finite graph $G$ and two orientations $Q$ and $Q'$ of $G$, compare $\Rep(Q,\FF_1)$ and $\Rep(Q',\FF_1)$. Note that one cannot appeal to reflection functors, since the usual definition does not make sense over $\FF_1$. An alternate strategy would be to compare the associated Hall algebras $H_Q$ and $H_{Q'}$, or their nilpotent variants $H_{Q,\nil}$ and $H_{Q',\nil}$. Since the Milnor-Moore Theorem applies to these Hopf algebras, one could even compare their Lie subalgebras of primitive elements. In this section, we study the Lie algebra $\mathfrak{n}_Q$ of $H_{Q,\nil}$ when $\overline{Q}$ is a tree or affine Dynkin diagram of type $\tilde{\mathbb{A}}_n$. Such quivers are precisely the ones of bounded representation type over $\mathbb{F}_1$ \cite{jun2020quiver}. We show that $\mathfrak{n}_Q \cong \mathfrak{n}_{Q'}$ when $Q$ and $Q'$ are different orientations of a tree $T$, and that if $Q$ is acyclic of type $\tilde{\mathbb{A}}_n$, then $\mathfrak{n}_Q$ is a central extension of the equioriented quiver of type $\tilde{\mathbb{A}}_n$. Explicit descriptions of $H_{Q,\nil}$ for all $Q$ of type $\tilde{\mathbb{A}}_n$ can then be deduced: previously, only the equioriented case had been described in the literature.

\begin{mydef}
Let $T$ be an undirected tree, and $Q$ a quiver with $\overline{Q} = T$. Let $\mathcal{S}_T$ denote the set of nonempty connected subgraphs of $T$, and let $\mathfrak{n}_T$ denote the free $\mathbb{C}$-vector space on $\mathcal{S}_T$. 
\end{mydef} 

 It was shown in \cite[Theorem 5]{szczesny2011representations} that indecomposable $\FF_1$-representations of $Q$ are in bijective correspondence with $\mathcal{S}_T$. Hence, the underlying vector space of $\mathfrak{n}_Q$\footnote{As mentioned above, $\mathfrak{n}_Q$ is the Lie sub-algebra of the Hall algebra $H_{Q,\nil} = H_Q$ consisting of primitive elements.} may be identified with $\mathfrak{n}_T$. Let the following
\[
[\cdot,\cdot]_Q : \mathfrak{n}_T\otimes\mathfrak{n}_T \rightarrow \mathfrak{n}_T
\]
denote the Lie bracket of $\mathfrak{n}_Q$, considered as a Lie bracket on $\mathfrak{n}_T$.  

The set $\mathcal{S}_T$ is a poset under the subgraph relation, i.e., for $T_1,T_2 \in \mathcal{S}_T$ $T_1 \leq T_2$ if and only if $T_1$ is a subgraph of $T_2$. Since $T$ is a tree, any two elements $ S, S' \in \mathcal{S}_T$ have a supremum, which we denote by $S\vee S'$\footnote{Elements of $\mathcal{S}_T$ do not always have an infimum since $\emptyset \not\in \mathcal{S}_T$. If $T$ is not a tree, there may not be a \emph{unique} least upper bound for $S$ and $S'$}. Indeed, $S\vee S'$ may be characterized as the full subgraph whose vertex set consists of $S_0\cup S'_0$, along with all vertices in the paths joining an element of $S_0$ to $S'_0$.  

Of course, the nonempty connected subquivers of $Q$ are in bijective correspondence with the elements of $\mathcal{S}_T$. By an abuse of notation, we will let $S$ denote the $Q$-representation corresponding to $S \in \mathcal{S}_T$ and we will write $\Gamma_S = S$, identifying the coefficient quiver simultaneously with $S$ and the subquiver of $Q$ it determines. Note that if $(S\vee S')_0 = S_0 \sqcup S'_0$, then $S\vee S'$ is obtained by adding exactly one new edge connecting a vertex of $S$ to a vertex of $S'$.  

\begin{mydef} 
	Suppose that $(S\vee S')_0 = S_0 \sqcup S'_0$. Then we write $S\vee S' = S\xrightarrow[]{\alpha} S'$ if $S\vee S'$, considered as a subquiver of $Q$, is obtained by adding an $\alpha$-colored arrow with $s(\alpha) \in S_0$ and $t(\alpha) \in S'_0$, where $\alpha \in Q_1$. We define $S\xleftarrow[]{\alpha} S'$ analogously. We write $S\alpha S'$ to mean either $S\xrightarrow[]{\alpha}S'$ or $S\xleftarrow[]{\alpha}S'$.
\end{mydef} 

Note that $[S,S']_Q \neq 0$ if and only if $S\vee S' = S\alpha S'$ for some $\alpha \in Q_1$ (See \cite[Lemma 3.12 (3)]{jun2020quiver}). More specifically, we can write 
\[ 
[S,S']_Q = \mu_{S,S'}^QS\vee S',
\] 
where $\mu_{S,S'}^Q \in \mathbb{C}$ is defined via the formula 
\[ \mu_{S,S'}^Q = 
\begin{cases} 
	+1 & \text{ if $S\vee S' = S\xrightarrow[]{\alpha} S'$ for some $\alpha \in Q_1$} \\ 
	-1 & \text{ if $S\vee S' = S\xleftarrow[]{\alpha}S'$ for some $\alpha \in Q_1$}\\ 
	0 & \text{ otherwise.}\\
\end{cases}
\] 

\begin{pro}\label{p: trees}
	Let $Q$ and $Q'$ be different orientations of a tree $T$. Then there is an isomorphism $\mathfrak{n}_Q \rightarrow \mathfrak{n}_{Q'}$. 
\end{pro} 

\begin{proof} 
	Let $r_{\alpha}(Q)$ denote the quiver that is obtained from $Q$ by reversing the direction of the arrow $\alpha$. Then $Q' = r_{\alpha_k}\cdots r_{\alpha_1}(Q)$, where $\alpha_1,\ldots ,\alpha_k$ are the arrows of $Q$ that have a different orientation than $Q'$. The result will therefore follow if we can show $\mathfrak{n}_Q \cong \mathfrak{n}_{r_{\alpha}(Q)}$ for any $\alpha \in Q_1$, and so without loss of generality we assume that there is an $\alpha$ for which $Q' = r_{\alpha}(Q)$. Define $\epsilon : \mathcal{S}_T \rightarrow \{ \pm 1\}$ via the formula 
	\[ 
	\epsilon(S) =  
	\begin{cases} 
		+1 & \text{ if $\alpha \not\in S_1$} \\ 
		-1 & \text{ if $\alpha \in S_1$.}
	\end{cases}
	\] 
	Here we identify $\alpha$ with its corresponding (undirected) edge in $T$. We claim that the $\mathbb{C}$-linear automorphism $\phi : \mathfrak{n}_T \rightarrow \mathfrak{n}_T$ defined via the formula 
	\[ 
	S \mapsto \epsilon(S)S
	\] 
	induces an isomorphism $\mathfrak{n}_Q \rightarrow \mathfrak{n}_{Q'}$. On the one hand, 
	\begin{align*}
		\phi\left([S,S']_Q \right) & = \phi\left(\mu_{S,S'}^QS\vee S'\right) \\ 
		& = \epsilon(S\vee S')\mu_{S,S'}^Q S\vee S'. 
	\end{align*} 
	On the other, 
	\begin{align*} 
		[\phi(S),\phi(S')]_{Q'} & = \epsilon(S)\epsilon(S')[S,S']_{Q'} \\ 
		& = \epsilon(S)\epsilon(S')\mu_{S,S'}^{Q'}S\vee S'.
	\end{align*} 
	Hence, we will be done if we can show  
	\begin{equation}\label{e: tree claim}
		\epsilon(S\vee S')\mu_{S,S'}^Q = \epsilon(S)\epsilon(S')\mu_{S,S'}^{Q'} 
	\end{equation} 
	for all $S,S' \in \mathcal{S}_T$. Note that $\mu_{S,S'}^Q = 0$ if and only if $\mu_{S,S'}^{Q'} = 0$, so we may assume that neither are zero. There are then two cases to consider: \\
	\noindent {\bf{Case 1:}} Suppose $S\vee S' = S\beta S'$ for some $\beta \neq \alpha$. Then $S\vee S' = S\xrightarrow[]{\beta}S'$ in $Q$ if and only if $S\vee S' = S\xrightarrow[]{\beta} S'$ in $Q'$, and $S\vee S' = S\xleftarrow[]{\beta}S'$ in $Q$ if and only if $S\vee S' = S\xleftarrow[]{\beta}S'$ in $Q'$. Hence $\mu_{S,S'}^{Q} = \mu_{S,S'}^{Q'}$ and \eqref{e: tree claim} reduces to $\epsilon(S\vee S') = \epsilon(S)\epsilon(S')$. But since $S\vee S' = S\beta S$, $\alpha \in (S\vee S')_1$ if and only if $\alpha$ is in \emph{exactly one} of $S$ or $S'$. The relation $\epsilon(S\vee S') = \epsilon(S)\epsilon(S')$ now readily follows. \\ 
	\noindent {\bf{Case 2:}} Suppose $S\vee S' = S\alpha S'$.  Then $S\vee S' = S\xrightarrow[]{\alpha}S'$ in $Q$ if and only if $S\vee S' = S\xleftarrow[]{\alpha} S'$ in $Q'$, and $S\vee S' = S\xleftarrow[]{\alpha}S'$ in $Q$ if and only if $S\vee S' = S\xrightarrow[]{\alpha}S'$ in $Q'$. Hence $\mu_{S,S'}^{Q} = -\mu_{S,S'}^{Q'}$ and \eqref{e: tree claim} reduces to $\epsilon(S\vee S') = -\epsilon(S)\epsilon(S')$. But $\epsilon(S) = \epsilon(S') = +1$ and $\epsilon(S\vee S') = -1$ for $S\vee S' = S\alpha S'$, so the equality is clear. \\
	We have proved that \eqref{e: tree claim} holds for all $S$ and $S'$, so $\phi$ is indeed an isomorphism of Lie algebras.
\end{proof}   

\begin{rmk} 
	Szczesny proved in \cite[Theorem 8]{szczesny2011representations} that there is a surjective map $U(\mathfrak{n}_{+}) \rightarrow H_Q$ when $\overline{Q}$ is a tree, and provides an example to show that this map is in general not injective. Here, $\mathfrak{n}_{+}$ denotes the positive part of the symmetric Kac-Moody algebra associated to $Q$. 
\end{rmk}

Set $\mathfrak{g} := \mathfrak{gl}_n(\mathbb{C})$ and let $\hat{\mathfrak{g}} = \mathfrak{g}[t,t^{-1}] \oplus \mathbb{C}c$ be the associated affine algebra. The bracket of $\hat{\mathfrak{g}}$ is given by the formula 

\[ 
[x\otimes t^m, y\otimes t^n] = [x,y]\otimes t^{m+n} + \operatorname{tr}(xy)n\delta_{n,-m}c.
\] 

Then $\hat{\mathfrak{g}}$ admits a triangular decomposition $\hat{\mathfrak{g}} = \mathfrak{a}_- \oplus \mathfrak{h}_n \oplus \mathfrak{a}_+$, where $\mathfrak{a}_- = t^{-1}\mathfrak{gl}_n(\mathbb{C})[t^{-1}]\oplus N_-$, $\mathfrak{h}_n = D_n \oplus \mathbb{C}c$, $\mathfrak{a}_+ = t\mathfrak{gl}_n(\mathbb{C})[t]\oplus N_+$, and $N_-$, $D_n$, $N_+$ denote the upper triangular, diagonal, and lower triangular matrices in $\mathfrak{gl}_n(\mathbb{C})$, respectively.  

Recall that if $Q$ is a type $\tilde{\mathbb{A}}_n$ quiver, with vertices oriented cyclically as $1,\ldots , n$\footnote{Note that our terminology is slightly different from that of \cite{szczesny2011representations}, where $\tilde{\mathbb{A}}_n$ is required to have $n+1$ vertices.}, then its nilpotent indecomposables are described through two families: $I_{[d,i]}$ and $\tilde{I}_d$, where $1 \le i \le n$ and $d\geq 1$. If $Q$ is equioriented, then only the first family yields nilpotent representations. $I_{[d,i]}$ is an $n$-dimensional string module, corresponding to wrapping a walk of length $n-1$ around $Q$ starting at the vertex $i$. For $Q$ acyclic, $\tilde{I}_d$ is a $dn$-dimensional $\FF_1$-band\footnote{Technically, only the thin representation $\tilde{I}_1$ yields a band in the traditional sense. See Definition \ref{d: trees and bands}.}, corresponding to a closed loop around $Q$ which wraps around $d$ times. For further details, we refer the reader to \cite[Constructions 5.10 and 5.11]{jun2020quiver}. Here is an example. 

\begin{myeg}
Consider the following acyclic quiver of type $\tilde{\mathbb{A}}_3$:
\[
\begin{tikzcd}
& v_1 \arrow[dr, "\alpha_1"] \arrow[dl,swap, "\alpha_3"]& \\
v_3 & & v_2 \arrow[ll, "\alpha_2"]
\end{tikzcd}
\]
Then as in \cite[Construction 5.10]{jun2020quiver} we have $I_{[8,3]}=\{M_1,M_2,M_3\}$, where
\[
M_1=\{k \mid 1\leq k \leq 8, k\equiv 1-3+1 (\textrm{mod} 3)\}=\{2,5,8\}. 
\]
Similarly, $M_2=\{3,6\}$ and $M_3=\{1,4,7\}$. The coefficient quiver is as follows:
\[
	\begin{tikzcd} 
	1  &	2 \arrow[l,green,swap, "\alpha_3"] \arrow[r,blue,"\alpha_1"] & 3 \arrow[r,red,"\alpha_2"] & 4& 5   \arrow[l,green,swap, "\alpha_3"] \arrow[r,blue,"\alpha_1"] & 6 \arrow[r,red,"\alpha_2"] & 7 & 8 \arrow[l,green,swap, "\alpha_3"]
	\end{tikzcd}
\] 
The coefficient quiver of $\tilde{I}_2$ is as follows: 
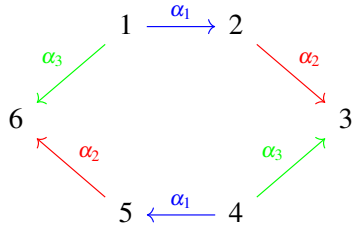
\[  
\begin{tikzcd}
& 1 \arrow[dl,green,swap,"\alpha_3"]\arrow[r,blue,"\alpha_1"] & 2 \arrow[dr,red,"\alpha_2"] & \\ 
6 & & & 3 \\
& 5 \arrow[ul,red,swap,"\alpha_2"] & 4 \arrow[l,blue,swap,"\alpha_1"] \arrow[ur,green,"\alpha_3"] & \\ 
\end{tikzcd}
\]
\end{myeg}

The following lemma is due to Szczesny \cite{szczesny2011representations}. We recall it here for convenience.

\begin{lem}\label{l.equiaff}
	Let $Q$ be the equioriented Dynkin quiver of type $\tilde{\mathbb{A}}_n$. Then $\mathfrak{n}_Q \cong \mathfrak{a}_{+}$. 
\end{lem}  

\begin{proof} 
	Let $\{ E_{ij}\}_{i,j \le n}$ denote the standard basis for $\mathfrak{g}$. Then as in \cite[Section 11]{szczesny2011representations}, a direct computation verifies that the map 
	\[ 
	\psi : \mathfrak{a}_+ \rightarrow \mathfrak{n}_Q
	\] 
	\[ 
	E_{ij}\otimes t^m \mapsto I_{[j-i+mn, i]}
	\] 
	is an isomorphism of Lie algebras.
\end{proof}


We now turn to a discussion of type $\tilde{\mathbb{A}}_n$ quivers with acyclic orientations. If $Q$ is such a quiver, then its indecomposable $\mathbb{F}_1$-representations are $I_{[d,i]}$ and $\tilde{I}_d$, where $1\le i \le n$ and $d \geq 1$. Note that any short exact sequence $0 \rightarrow N \rightarrow E \rightarrow M \rightarrow 0$ with $N$ or $M$ isomorphic to some $\tilde{I}_d$ necessarily splits. This means that the elements $[\tilde{I}_d]$ are all central in $\mathfrak{n}_Q$. As the proposition below shows, we can obtain $\mathfrak{n}_Q$ from $\mathfrak{a}_+$ via a central extension.  

For the proposition below, we will use $\delta$'s to denote indicator functions on subsets of $\{ (i,j) \mid 1 \le i ,j \le n \}$. Subsets which are defined by equations will be denoted by those equations, for example, $\delta_{i \equiv j (\operatorname{mod} n)}$ is the indicator function for the set $\{ (i,j) \mid i \equiv j (\operatorname{mod} n)\}$. When $n$ is understood from context we will just write $i \equiv j$. The indicator function for the subset $\{ (i,j), (j,i)\}$ will be denoted by the standard Kronecker notation $\delta_{i,j}$.

Let $Q = \alpha_1^{\epsilon(1)}\cdots \alpha_{n-1}^{\epsilon(n-1)}$ be an acyclic quiver of type $\tilde{\mathbb{A}}_n$. Furthermore let $(i,j,q)$ and $(k,l,s)$ be two elements of $\mathbb{N}^3$ such that $1 \le i,j,k,l \le n$, either $i \le j$ or $q>0$, and either $k \le l$ or $s>0$. Then the representations $I_{[j-i+qn,i]}$ and $I_{[l-k+sn,k]}$ are well-defined. One can compute the following identity in $H_Q$, where we have set $d:= (j+l)-(i+k) + (q+s)n$ to ease notation: 
\vspace{0.2cm}
\begin{center} 
	$[I_{[j-i+qn,i]}]\cdot [I_{[l-k+sn,k]}] =  [I_{[l-k+sn,k]}\oplus I_{[j-i+qn,i]}] + \delta_{j,k}\delta_{\epsilon(j-1),1}[I_{[d,i]}] + \delta_{i,l}\delta_{\epsilon(l-1),-1}[I_{[d,k]}] + \delta_{j,k}\delta_{i,l}\delta_{\epsilon(j-1),1}\delta_{\epsilon(l-1),-1}(q+s)[\tilde{I}_{q+s}].$ 
\end{center} 
\vspace{0.2cm}
This immediately implies the following computation in $\mathfrak{n}_Q$: 
\begin{align*} 
	\left[[I_{[j-i+qn,i]}], [I_{[l-k+sn,k]}]\right] & =  \delta_{j,k}\left( \delta_{\epsilon(j-1),1} - \delta_{\epsilon(j-1),-1} \right)[I_{[d,i]}] \\ 
	&+ \delta_{i,l}\left( \delta_{\epsilon(l-1),-1} - \delta_{\epsilon(l-1),1} \right)[I_{[d,k]}] \\ 
	& + \delta_{j,k}\delta_{i,l}\left( \delta_{\epsilon(k-1),1} \delta_{\epsilon(l-1),-1} -  \delta_{\epsilon(l-1),1} \delta_{\epsilon(k-1),-1} \right)(q+s)[\tilde{I}_{q+s}] \\ 
	& = \delta_{j,k}\epsilon(j-1)[I_{[d,i]}] - \delta_{i,l}\epsilon(l-1)[I_{[d,k]}]  \\ 
	&+ \delta_{j,k}\delta_{i,l}\delta_{\epsilon(k-1),-\epsilon(l-1)}\epsilon(k-1)(q+s)[\tilde{I}_{q+s}].
\end{align*} 

We are now ready to identify $\mathfrak{n}_Q$ in the acyclic case.

\begin{pro}\label{p: affine}
	Let $Q$ be a Dynkin quiver of type $\tilde{\mathbb{A}}_n$. Define $Z_Q \le \mathfrak{n}_Q$ be the zero ideal when $Q$ is equioriented, and the central ideal spanned by $\{ [\tilde{I}_d] \mid d \geq 1 \}$ otherwise. Then there is an isomorphism $\mathfrak{n}_Q/Z_Q \cong \mathfrak{a}_+$. In particular, $\mathfrak{n}_Q$ is a central extension of $\mathfrak{a}_+$.  
\end{pro} 

\begin{proof}  
	Write $Q = \alpha_1^{\epsilon(1)}\cdots \alpha_{n-1}^{\epsilon(n-1)}$ with $\epsilon(i) = \pm 1$. The equioriented case is already proved by Szczesny (Lemma \ref{l.equiaff}). So, we may assume that $\{ \epsilon(i) \mid 1 \le i \le n-1\}$ is not a proper subset of $\{-1, +1\}$. Let $(i,j,q)$ and $(k,l,s)$ be two elements of $\mathbb{N}^3$ such that $1 \le i,j,k,l\le n$, either $i \le j$ or $q > 0$, and either $k \le l$ or $s>0$. Then the representations $I_{[j-i+qn,i]}$ and $I_{[l-k+sn,k]}$ are well-defined. From the discussion above we have 
		\begin{equation}\label{e.commutator}
		\left[[I_{[j-i+qn,i]}], [I_{[l-k+sn,k]}]\right] = \delta_{j,k}\epsilon(j-1)[I_{[d,i]}] - \delta_{i,l}\epsilon(l-1)[I_{[d,k]}]+ \lambda_{(i,j,q):(k,l,s)}[\tilde{I}_{q+s}],
	\end{equation}
where the coefficient $ \lambda_{(i,j,q):(k,l,s)}$ is given by the formula 
	\begin{equation} 
		\lambda_{(i,j,q):(k,l,s)} =  \delta_{j,k}\delta_{i,l}\delta_{\epsilon(k-1),-\epsilon(l-1)}\epsilon(k-1)(q+s).
	\end{equation}
	Then we define a map $\psi_Q : \mathfrak{a}_+ \rightarrow \mathfrak{n}_Q/Z_Q$ by the formula 
	\begin{equation} 
		E_{ij}\otimes t^q \mapsto \epsilon(j-1)I_{[j-i+qn,i]} + Z_Q.
	\end{equation} 
	\eqref{e.commutator} immediately implies that $\psi_Q$ is a Lie algebra morphism. Showing that it is a bijection is essentially the same as in Lemma \ref{l.equiaff}.
\end{proof}



\subsection{The Hall algebra of representations with finite nice length}\label{s: nice length}  

In this section, we associate a Hall algebra $H_Q^{\operatorname{nice}}$ (resp. $H_{Q,\nil}^{\operatorname{nice}}$) to the category of (resp. nilpotent) $\FF_1$-representations $M$ of $Q$ with $\operatorname{nice}(M)<\infty$. We then relate these Hall algebras to $H_Q$ and $H_{Q,\nil}$. A description of $H_{Q,\nil}^{\operatorname{nice}}$ when $Q$ is a (not-necessarily proper) pseudotree is obtained. In this case, the representations $M$ with $\operatorname{nice}(M)< \infty$ are related to \emph{absolutely-indecomposable} $\FF_1$-representations of quivers with bounded representation type over $\FF_1$. Absolutely-indecomposable $\FF_1$-representations remain mysterious for general $Q$, and will be the subject of a future article. 

Recall from Section \ref{s: Euler} that if $\operatorname{nice}(M)<\infty$, then any subquotient $S$ of $M$ also satisfies $\operatorname{nice}(S)<\infty$. It follows that the full subcategory $\Rep(Q,\FF_1)^{\operatorname{nice}}$ (resp. $\Rep(Q,\FF_1)_{\nil}^{\operatorname{nice}}$) of representations (resp. nilpotent representations) of finite nice length is finitary and proto-exact. Therefore, we can associate to it a Hall algebra $H_Q^{\operatorname{nice}}$ (resp. $H_{Q,\nil}^{\operatorname{nice}}$).   

It is easy to describe $H_Q^{\operatorname{nice}}$ and $H_{Q,\nil}^{\operatorname{nice}}$ in terms of $H_Q$ and $H_{Q,\nil}$, respectively. 

\begin{pro}\label{p: hopf ideal} 
	Let $Q$ be a fixed quiver. Then the ideal of $H_Q$ (resp. $H_{Q,\nil}$) generated by representations with infinite nice length is a Hopf ideal, and we have isomorphisms of Hopf algebras:
	\[ 
	H_Q^{\operatorname{nice}} \cong H_Q/\langle [M] \mid \operatorname{nice}(M) = \infty \rangle,
	\] 
	\[ 
	H_{Q,\nil}^{\operatorname{nice}} \cong H_{Q,\nil}/\langle [M] \mid \operatorname{nice}(M) = \infty \rangle.
	\] 
\end{pro} 

\begin{proof}  
Set $I = \langle [M] \mid \operatorname{nice}(M) = \infty \rangle \subseteq H_Q$, the ideal generated by representations with infinite nice length. Note that any element in $I$ can be written as a $\mathbb{C}$-linear combination of isomorphism classes $[M]$ with $\operatorname{nice}(M) = \infty$. This follows from the fact that for any short exact sequence
	\[ 
	0 \rightarrow M_1 \rightarrow M_2 \rightarrow M_3 \rightarrow 0
	\] 
in $\Rep(Q,\FF_1)$, $\operatorname{nice}(M_1) = \infty$ or $\operatorname{nice}(M_3) = \infty$ implies $\operatorname{nice}(M_2) = \infty$. Hence, $I$ is the vector space with basis given by $\{ [M] \mid \operatorname{nice}(M) = \infty\}$. Given $[M] \in I$, we may write $\Delta([M])$ as
\[ 
\Delta([M]) =\sum_{(A,B) \in D_2(M)}{[A]\otimes [B]},
\] 
where $D_2(M) = \{ (A,B) \in \operatorname{Iso}(Q)^2 \mid M \cong A\oplus B\}$. But $M \cong A \oplus B$ if and only if $\Gamma_M$ can be written as the disjoint union of $\Gamma_A$ and $\Gamma_B$. Then $\operatorname{nice}(M) = \infty$ implies that either $\operatorname{nice}(A) = \infty$ or $\operatorname{nice}(B) = \infty$, and hence  
\[ 
\Delta([M]) =\sum_{(A,B) \in D_2(M)}{[A]\otimes [B]} \in I\otimes H_Q + H_Q \otimes I.
\]
It follows that $I$ is a coideal of $H_Q$. To show that $I$ is invariant under the antipode, simply note that $S([M])$ is a sum of products of the form $[M_1]\cdots [M_k]$, where $M \cong M_1\oplus \cdots \oplus M_k$. Then $[M] \in I$ implies that $[M_i] \in I$ for at least one index $i$, and hence $S([M]) \in I$ as well. It follows that $I$ is a Hopf ideal, so that $H_Q/I$ is a Hopf algebra. 

There is a $\mathbb{C}$-linear isomorphism $H_Q^{\operatorname{nice}} \rightarrow H_Q/I$ defined by $[M] \mapsto [M] +I$. This is an algebra morphism, since a short exact sequence in $\Rep(Q,\FF_1)^{\operatorname{nice}}$ is the same as a short exact sequence 
\[ 
0 \rightarrow M_1 \rightarrow M_2 \rightarrow M_3 \rightarrow 0
\]
with $\operatorname{nice}(M_i)< \infty$ for all $i$. It is also a coalgebra morphism, since the direct sum decompositions of $M$ in $\Rep(Q,\FF_1)^{\operatorname{nice}}$ are exactly the direct sum decompositions of $M$ in $\Rep(Q,\FF_1)$. Hence, we have the desired isomorphism $H_Q^{\operatorname{nice}} \cong H_Q/I$. The proof for nilpotent representations is similar.
\end{proof}

The results from \cite{szczesny2011representations} and Section \ref{s: Hall Algebras} above allow us to compute $H_{Q,\nil}^{\operatorname{nice}}$ when $Q$ has bounded representation type over $\FF_1$.\footnote{We proved in \cite[Theorem 5.3]{jun2020quiver} that when $Q$ is connected, $Q$ has bounded representation type over $\FF_1$ if and only if $Q$ is either a tree or type $\tilde{\mathbb{A}}_n$. } In this case, indecomposable representations in $\Rep(Q,\FF_1)_{\nil}^{\operatorname{nice}}$ are precisely the \emph{absolutely indecomposable} representations which we define below. 

\begin{mydef} 
Let $Q$ be a quiver and $M$ an $\mathbb{F}_1$-representation of $Q$. We say that $M$ is \emph{absolutely indecomposable} if $M\otimes_{\mathbb{F}_1}k$ is indecomposable for any algebraically-closed field $k$.
\end{mydef}  

If $k$ is a field and $Q$ is a quiver, a $k$-representation $M$ of $Q$ is said to be \emph{thin} if the components of the dimension vector $\textbf{\textrm{dim}}(M)$ are either $0$ or $1$. We have an analogous definition for $\FF_1$-representations of $Q$. If $S$ is a subquiver of $Q$, then one obtains a thin representation $\mathbbm{1}_S$ by taking $k$ at each vertex of $S$, the identity map $k\rightarrow k$ at each arrow of $S$, and zero spaces and maps at the remaining vertices and arrows. Of course $\mathbbm{1}_S$ is a scalar extension of an $\FF_1$-representation which we also call $\mathbbm{1}_S$. A straightforward computation reveals that $\operatorname{End}_{kQ}(\mathbbm{1}_S) \cong k^c$, where $c$ is the number of connected components of $S$. It follows that $\mathbbm{1}_S$ is indecomposable if and only if $S$ is connected. 

Recall from \cite{jun2020quiver} that if $Q$ is an acyclic quiver of type $\tilde{\mathbb{A}}_n$, then the indecomposable representations are either string modules $I_{[d,i]}$, whose coefficient quivers are are oriented line graphs, or $\FF_1$-bands $\tilde{I}_d$, whose coefficient quivers are of type $\tilde{\mathbb{A}}_{dn}$ (for $d \geq 1$). 

\begin{cor}\label{c: nice bounded}
Let $Q$ be a connected quiver of bounded representation type and $M$ an indecomposable nilpotent $\mathbb{F}_1$-representation of $Q$. Then $\operatorname{nice}(M)<\infty$ if and only if $M$ is absolutely indecomposable. Furthermore, exactly one of the following holds:
\begin{enumerate} 
\item $Q$ is a tree. In this case, $H_{Q,\nil}^{\operatorname{nice}} = H_{Q}^{\operatorname{nice}} = H_{Q,\nil} = H_Q$. 
\item $Q$ is equioriented of type $\tilde{\mathbb{A}}_n$. In this case, $H_{Q,\nil}^{\operatorname{nice}} = H_{Q,\nil}$.
\item $Q$ is acyclic of type $\tilde{\mathbb{A}}_n$. In this case, $\langle [\tilde{I}_d]\mid d>1 \rangle = \langle [M] \mid \operatorname{nice}(M) = \infty \rangle$ and $H_{Q,\nil}^{\operatorname{nice}} = H_Q^{\operatorname{nice}} = H_Q/\langle [\tilde{I}_d]\mid d>1 \rangle$.
\end{enumerate}
\end{cor} 

\begin{proof} 
Let $k$ be an algebraically-closed field and $M$ an indecomposable nilpotent $\mathbb{F}_1$-representation of $Q$. In cases (1) and (2), $\Gamma_M$ is a tree and so $\operatorname{nice}(M)<\infty$ by Corollary \ref{c: distinguishing trees}. Since $\Gamma_M$ is a tree, $M\otimes_{\mathbb{F}_1}k$ is a tree module and indecomposable \cite{gabriel1981quivers}. In case (3), $\operatorname{nice}(I_{[d,i]})<\infty$ and $I_{[d,i]}$ is absolutely indecomposable for all $d$ by the same argument. Hence, it only remains to determine which of the $\tilde{I}_d$ are of finite nice length. For $d=1$, $\tilde{I}_1$ is isomorphic to the thin representation $\mathbbm{1}_Q$: then $\operatorname{nice}(\tilde{I}_1)< \infty$ and $\tilde{I}_1$ is absolutely indecomposable since $Q$ is connected. For $d>1$, $\operatorname{nice}(\tilde{I}_d) = \infty$ by Corollary \ref{l: distinguishing bands} and it is straightforward to verify that $\tilde{I}_d\otimes_{\mathbb{F}_1}k$ is decomposable. The claim now follows.
\end{proof}  

\begin{cor}\label{c: nice pseudotree}
Let $Q$ be a connected proper pseudotree with central cycle $C$. Let $M$ be an indecomposable nilpotent $\FF_1$-representation of $Q$. Then the following are equivalent:  
\begin{enumerate} 
\item $\operatorname{nice}(M)< \infty$.
\item $M$ is a thin module or a tree module. 
\item $\operatorname{Res}_C(M)$ is absolutely indecomposable. 
\end{enumerate} 
In particular, all $M$ with $\operatorname{nice}(M)<\infty$ are absolutely indecomposable.
\end{cor}

\begin{proof} 
Clearly (2)$\Rightarrow$ (1). To prove (1)$\Rightarrow$ (3), first note that the coefficient quiver of $\operatorname{Res}_C(M)$ is a coefficient subquiver of $\Gamma_M$, so $\operatorname{nice}(M)<\infty$ implies $\operatorname{nice}(\operatorname{Res}_C(M))< \infty$. The result then follows from Corollary \ref{c: nice bounded}. To prove (3)$\Rightarrow$(2), note that $\operatorname{Res}_C(M)$ absolutely indecomposable implies that $\operatorname{Res}_C(M)$ is a tree module or $\tilde{I}_1$. Then either $M$ is a tree module or thin, as we wished to show.
\end{proof}  

\begin{cor}\label{c: nice pseudotree hall} 
Let $Q$ be a connected proper pseudotree with central cycle $C$. Then Hall algebra $H_{Q,\nil}^{\operatorname{nice}}$ is described as follows: 
\begin{enumerate}   
\item If $C$ is equioriented, then $H_{Q,\nil}^{\operatorname{nice}} = H_{Q,\nil}$. 
\item If $C$ is not equioriented, then $H_{Q,\nil}^{\operatorname{nice}} = H_Q^{\operatorname{nice}} = H_{Q}/\langle [M] \mid \operatorname{Res}_C(M) \cong \tilde{I}_d, d>1 \rangle$. 
\end{enumerate}
\end{cor}  

\begin{proof} 
If $C$ is equioriented then every nilpotent indecomposable $\FF_1$-representation is a tree module. If $C$ is not equioriented, then $Q$ is acyclic and the claim follows from Corollary \ref{c: nice pseudotree}(3). 
\end{proof}

Let $Q$ be a connected proper pseudotree. Then an absolutely indecomposable $\FF_1$-representation does not necessarily have finite nice length as the following example illustrates. 

\begin{myeg}
Let $Q$ be the following proper pseudotree:
\[ 
Q = 
\begin{tikzcd} 
s \arrow[r,bend left = 20, "\alpha"] \arrow[r,bend right = 20,swap,"\beta"] & t \arrow[r,"\gamma"] & u
\end{tikzcd}
\]
Let $M$ be the $\FF_1$-representation of $Q$ with the following coefficient quiver:
\[ 
\Gamma_M=
\begin{tikzcd}  
 & s \arrow[dl,blue,swap,"\alpha"]\arrow[dr,red,"\beta"]  & & \\ 
t  & & t \arrow[r,black,"\gamma"] & u  \\ 
 & s \arrow[ul,red,"\beta"] \arrow[ur,blue,swap,"\alpha"]& &  \\
\end{tikzcd}.
\]    
A routine calculation shows that $M_k$ is indecomposable for any algebraically-closed field $k$. Hence $M$ is absolutely indecomposable, but $\operatorname{nice}(M) = \infty$ by Corollary \ref{c: nice pseudotree}.
\end{myeg}

Let $\FF_q$ denote the finite field with $q$ elements. For a quiver $Q$ and dimension vector $\textbf{d}$, there are finitely-many isomorphism classes of $\textbf{d}$-dimensional, absolutely indecomposable $\FF_q$-representations of $Q$. One can then associate a function $A_Q(\textbf{d},q)$ which counts the isoclasses in each dimension. It is shown in \cite{kac1980absolutely} that $A_Q(\textbf{d},q)$ is a polynomial in $q$ with integer coefficients. These functions are shown to satisfy important combinatorial identities in \cite{hua2000counting}, and the behavior of $A_{\wild_n}(\textbf{d},q)$ at $q= 1$ is described via tree modules\footnote{Using a broader notion of tree modules than the one considered in this paper.} in \cite{helleloid2009counting,kinser2013tree}. 

It is natural to ask how much of this theory carries over to $\FF_1$-representations. Certainly, there are finitely-many isomorphism classes of absolutely indecomposable, $\textbf{d}$-dimensional $\FF_1$-representations of $Q$. One can then define a counting function $A_Q(\textbf{d})$ analogous to that of $A_Q(\textbf{d}, q)$. Two questions immediately become apparent. 

\begin{question} 
How can one characterize the absolutely indecomposable $\FF_1$-representations of a given quiver $Q$?
\end{question} 

\begin{question} 
How can one compute $A_Q(\textbf{d})$ for a given quiver $Q$?
\end{question}

\noindent One would hope that absolutely indecomposable $\FF_1$-representations of $Q$ could be used to interpret the numbers $A_Q(\textbf{d}, 1)$. However, no obvious relationship exists between $A_Q(\textbf{d})$ and $A_Q(\textbf{d}, 1)$. This leads us to the following question.  

\begin{question} 
Let $Q$ be a quiver. Does there exist a finitary category $\mathcal{F}_Q$ and a faithful functor $F: \Rep(Q,\FF_1)\rightarrow \mathcal{F}_Q$ with the following properties? 
\begin{enumerate} 
\item For each field $k$, there is a functor $G_k : \mathcal{F}_Q \rightarrow \Rep(Q,k)$ such that the following diagram commutes: \\ 
\[
\begin{tikzcd} 
\Rep(Q,\FF_1) \arrow[r,"F"] \arrow[dr,swap,"k\otimes_{\FF_1} -"] & \mathcal{F}_Q \arrow[d,"G_k"]\\ 
& \Rep(Q,k) 
\end{tikzcd} 
\]
\item For any dimension vector $\textbf{d}$, there is a suitable notion of an absolutely indecomposable, $\textbf{d}$-dimensional object in $\mathcal{F}_Q$.
\item There are finitely-many absolutely indecomposable objects in $\mathcal{F}_Q$ for each $\alpha$, and the associated counting function is precisely $A_Q(\textbf{d},1)$. 
\end{enumerate}
\end{question} 

\noindent For instance, one might ask whether $\mathcal{F}_Q$ could be obtained by suitably modifying the category $\mathcal{C}_Q$ described in Section \ref{s: slice}. These questions will be the topic of a future paper.


\appendix 

\section{The Proofs of Lemmas 6.3-6.4 of \cite{Haupt2012euler}}
In Lemmas 6.3-6.4 of \cite{Haupt2012euler}, Haupt asserts that Equation \eqref{e: nice} holds for tree and band modules. Unfortunately, we have reason to believe that Haupt's original proofs of these two lemmas contain significant mistakes, although the statments themselves are ultimately true.\footnote{At least for $\mathbb{F}_1$-representations, which is the concern of the present work. We note that Lemma 6.4 is stated for band modules, which are not always defined over $\mathbb{F}_1$.} We outline our concerns with each result below, and then discuss how our results address these concerns. We have no reason to doubt the validity of any other result in \cite{Haupt2012euler}. In particular, the authors still believe the geometric material preceeding these results to be sound. 

\subsection{Gaps in the original proof of \cite[Lemma 6.3]{Haupt2012euler}} 

 Haupt's proof of Lemma 6.3 begins with the following language: ``By Proposition 6.1 it is enough to treat the cases when $F_0:S_0 \rightarrow Q_0$ is surjective and not injective. If $i,j \in S_0$ exist with $F_0(i) = F_0(j)$ and $i \neq j$, we construct a nice grading $\partial$ of $F_*(1_S)$ such that $\partial(i) \neq \partial(j)$.'' However, Example \ref{ex: string} shows that this statement cannot be taken literally: \emph{any} nice grading on $M$ fails to distinguish the first and last vertices. This is not surprising, as $\operatorname{nice}(M) = 1$. It would be necessary to first find a nice grading $\partial_0$ on $M$, and then construct a $\partial_0$-nice grading which distinguishes $i$ and $j$.

Proceeding with the argument, Haupt defines $S'$ to be a ``minimal connected subquiver such that there exist $i, j \in S_0'$ with $F_0(i) = F_0(j)$ and $i \neq j$.'' Note that Haupt has changed the meaning of $i$ and $j$ from the previous paragraph: in terms of Example \ref{ex: string}, the candidates for $S'$ would be any of the four arrows. He then constructs a nice grading distinguishing the endpoints of this $S'$. This grading is analogous to the nice grading in Equation \eqref{eq: e1}, which takes different values at adjacent vertices. This still leaves the problem of distinguishing the endpoints of $S'$ when it is not ``minimal'' in the sense of this proof, since ``minimality'' is unconnected to the size of $|S_0| - |Q_0|$.

To complete Haupt's argument, we would need to assume that the result has been proved in the case when a type $\mathbb{A}$ Dynkin subquiver contains $k$ vertices lying in a single fiber of the winding, and then prove that it holds when the subquiver contains $k+1$ vertices lying in a single fiber. However, Haupt's argument breaks down at this stage: specifically, the line where Haupt writes ``...so for all $1<k<l$ the equation $F_1(s_1) \neq F_1(s_k)$ holds'' would no longer be valid (consider the first and third arrows of Example \ref{ex: string}). Note that Haupt does not attempt to move to $1$-nice gradings (the construction in \cite[Proposition 6.1]{Haupt2012euler}), as we do in Example \ref{ex: string}. Also note that the quiver Haupt calls $S'$ is unrelated to the construction of $Q'$ in \cite[Proposition 6.1]{Haupt2012euler}. We are left to conclude that the original proof presented in \cite[Lemma 6.3]{Haupt2012euler} is incomplete. 

\subsection{Gaps in the original proof of \cite[Lemma 6.4]{Haupt2012euler}} 

Haupt attempts to demonstrate that if $F : S \rightarrow Q$ is a winding with $S$ a affine Dynkin quiver of type $\tilde{\mathbb{A}}_n$, and that none of the nice gradings $\partial^{(a,b)}$ distinguish a minimal pair of vertices lying in a single fiber of $F_0$, then $F$ cannot be primitive. This claim appears to be false. Consider the acyclic, primitive winding of $\mathbb{L}_3$ whose coefficient quiver is given by: 
\begin{center}
\begin{tikzcd}
 & & \bullet \arrow[rr,"\gamma"] & & \bullet  & &\\  
 & \bullet \arrow[ur,"\beta"] & & & & \bullet \arrow[ul,"\alpha", swap] & \\ 
\bullet \arrow[ur,"\alpha"] \arrow[dr,"\gamma",swap] & & & & & & \bullet \arrow[ul,"\beta", swap] \\
& \bullet \arrow[dr,"\gamma",swap] & & & & \bullet \arrow[ur,"\gamma",swap] & \\ 
& & \bullet \arrow[r,"\gamma",swap] & \bullet \arrow[r,"\gamma",swap] & \bullet \arrow[ur,"\gamma",swap] & & \\ 
\end{tikzcd}  
\end{center}

 The condition for an integer-valued map on the vertices of this quiver to be a nice grading is simply $\Delta_{\gamma} = 0$, so the source and target of each $\gamma$-colored arrow must have the same image under any nice grading. In particular, the source and target of any $\gamma$-colored arrow on the bottom oriented path corresponds to a minimal pair $i$ and $j$ that cannot be distinguished. This again requires us to at least consider $1$-nice gradings, which is not done in \cite[Lemma 6.4]{Haupt2012euler}.
 
Proceeding with the argument, Haupt asserts that ``$\epsilon_k\rho(F_1(s_k)) = \epsilon_m\rho(F_1(s_m))$ for all $k, m \in S_0$ with $F(s_k) = F(s_m)$'' and uses this to conclude that $S$ is not primitive. However, this statement is false for the $\gamma$-colored arrows in the above example. For instance, if we traverse the cycle in the clockwise direction, the expression $\epsilon_k\rho(F_1(s_k))$ is $-5$ for the $\gamma$-colored arrow on top, and $+5$ for any $\gamma$-colored arrow on the bottom. Again, we are forced to conclude that the original proof in \cite[Lemma 6.4]{Haupt2012euler} is incomplete.

\subsection{Patches to the gaps of \cite[Lemmas 6.3-6.4]{Haupt2012euler}}

A common issue with \cite[Lemmas 6.3-6.4]{Haupt2012euler} is that at some point in the middle of the argument, it is \emph{necessary} to switch from the winding $S \rightarrow Q$ to the winding $S \rightarrow Q'$ described by \cite[Proposition 6.1]{Haupt2012euler}. Indeed, Example \ref{ex: string} and the $\mathbb{F}_1$-band in the previous subsection both have finite nice length, as one could see from iterating once.\footnote{Note that Example \ref{ex: string2} shows that it may be necessary to iterate more than once.} Unfortunately, Haupt only invokes Proposition 6.1 at the \emph{start} of each proof, to reduce to the case of distinguishing $i, j \in S_0$ with $F(i) = F(j)$. Furthermore, Haupt's construction depends on a choice of grading, and there are many unhelpful choices (e.x. constant functions are always nice gradings).

The material we develop in Sections 4 and 5 of this article is meant to address these exact issues. Indeed, nice sequences, $i$-nice variables, universal $i$-nice gradings, and nice length all show how one can iterate the construction of \cite[Proposition 6.1]{Haupt2012euler} in a logically coherent and efficient fashion. The proofs of  Lemma \ref{l: uni walks}, Corollary \ref{c: distinguishing trees} and Theorem \ref{l: distinguishing bands} in this article are different from Haupt's proofs in that the use of $i$-nice variables allows us to systematically switch between windings as many times as is necessary to avoid the errors outlined above.

\bibliography{quiver}\bibliographystyle{alpha}
	
\end{document}